\documentclass{amsart}
\numberwithin{equation}{section}
\def\fall{\ \ {\rm for} \ {\rm all } \ \ }
\def\epsilon{\varepsilon}

\def\Id{\rm Id}
\def\BB{\mathcal B}
\def\BBB{\mathcal B \mathcal B}
\def\norm{ \| }
\def\parti#1#2{\frac{\partial #1 } {\partial #2} }
\def\of{\circ}
\def\subsub{\subset \subset}
\def\diam{{\rm diam } }
\def\de{\delta}
\def\inj{\rm inj }
\def\cut{ {\rm cut}}
\def\upto{\nearrow}
\def\downto{\searrow}
\def\Sc{{\rm R}}
\def\Ricci{{\rm{Ricci}}}
\def\Rc{{\rm{Rc}}}
\def\ti{\tilde}

\def\ga{\gamma}

\def\beq{ \begin{equation}  \begin{split} }
\def\eeq{ \end{split}  \end{equation} }
\def\be{\beta}
\def\lap{\Delta}
\def\sing{ {\rm Sing} }
\def\reg{ {\rm Reg} }
\def\regt{{\rm Reg}_t}
\def\al{\alpha}
\def\ep{\epsilon}

\def\partt{\frac{\partial }{\partial t} }

\def\partr{\frac{\partial }{\partial r } }
\def\phi{\varphi}
\def\R{\mathbb R}

\def\Sp{\mathbb S}
\def\boundary{\partial}
\def\N{\mathbb N}

\def\M^n#1#2{\mathcal M^{#1}\left(#2\right)}

\def\gradh{{{}^{{}^h}\!\nabla}}

\def\si{\sigma}
\def\ep{\epsilon}

\def\part{\partial}

\def\curlL{\mathcal L}
\def\curlH{\mathcal H}
\def\grad{\nabla}

\DeclareMathOperator{\vol}{vol}

\DeclareMathOperator{\Riem}{Riem}
\DeclareMathOperator{\Weil}{Weil}
\DeclareMathOperator{\Rm}{Rm}
\DeclareMathOperator{\Ric}{Ric}

\DeclareMathOperator{\dist}{dist}

\DeclareMathOperator{\CUT}{CUT}

\newtheorem{theo}{Theorem}[section]
\newtheorem{lemma}[theo]{Lemma}
\newtheorem{lemm}[theo]{Lemma}
\newtheorem{corollary}[theo]{Corollary}
\newtheorem{coro}[theo]{Corollary}

\newtheorem{prop}[theo]{Proposition}
\theoremstyle{definition}
\newtheorem{defi}[theo]{Definition}
\newtheorem{definition}[theo]{Definition}

\theoremstyle{remark}
\newtheorem{remark}[theo]{Remark}

\begin{document}
\title[Extending 4D Ricci Flow]{Extending four dimensional Ricci flows with
  bounded scalar curvature}


\author{Miles Simon}
\address{Miles Simon: Otto von Guericke University, Magdeburg, IAN,
  Universit\"atsplatz 2, Magdeburg 39104, Germany}
\curraddr{}
\email{msimon@gmx.de}

\subjclass[2000]{53C44}
\date{\today}

\dedicatory{}

\keywords{Ricci flow, scalar curvature}

\begin{abstract}
We consider solutions $(M,g(t)),  0 \leq t <T$,  to Ricci flow on
compact, connected  four dimensional
manifolds without boundary.
We assume that the scalar curvature is bounded uniformly, and that $T< \infty$.
In this case, we show that the metric space $(M,d(t))$ associated to $(M,g(t))$
converges uniformly in the $C^0$ sense to $(X,d)$, as $t \upto T$,
where $(X,d)$  is a {\it $C^0$ Riemannian orbifold} with at most
finitely many orbifold points. 
Estimates on the rate of convergence near and away from the orbifold
points are given. We also show that it is possible  to continue the flow past $(X,d)$
using the {\it orbifold } Ricci flow.
\end{abstract}

\maketitle
\section{Introduction}

In this paper we consider smooth solutions to
Ricci flow, $\partt g(t) = -2\Rc(g(t))$ for all $t \in [0,T)$, on
closed, connected four manifolds without boundary. We assume that $T < \infty$ and that
the scalar curvature satisfies $\sup_{M\times [0,T)} |\Sc|\leq 1$. 
In a previous paper, see Theorem 3.6 in \cite{Si1}, we showed that this implies

\hfill\break
(i) {\bf  Integral bounds for the Ricci and Riemannian curvature }\\
\hspace{0.5cm}
\begin{eqnarray*}
&&\sup_{ t\in [0,T)} \int_M |\Riem(\cdot,t)|^2 d\mu_{g(t)} \leq c_1 <
\infty  \cr
&&\int_0^T\int_M |\Rc|^4(\cdot,t)d\mu_{g(t)} dt\leq c_2 < \infty 
\end{eqnarray*}
for explicit  constants $c_1=c_1(M,g(0),T)$ and $c_2(M,g(0),T)$.
An estimate of the first type was independently proved, using
different methods, in a recent
paper, \cite{BZ} (see Theorem 1.8).

In this paper we show the following.
\hfill\break
(ii) {\bf Estimates for the singular and regular regions}\\
A point $p\in M$ is said to be {\it regular}, if there exists an $r >0$ such
that $$\int_{{{}^tB}_r(p)} |\Riem|^2(\cdot,t) d\mu_{g(t)} \leq \ep_0$$ for
all $t \in (0,T)$, for
some fixed small $\ep_0$ (not depending on $p$) which is specified in
the proof of Theorem \ref{regularregion}.
 In Definition \ref{defregII}, an
alternative definition of {\it regular} is given.
The {\it singular}  points
are those which are not regular.
In Theorem
\ref{regularregion} (and the Corollaries \ref{correg1} and
\ref{uniformdcoro} thereof) and Theorem \ref{singularregion} we obtain estimates for the
evolving metric in the 
singular and regular regions of the manifold. 

\hfill\break
(iii) {\bf Uniform continuity of the distance function in time.}\\
Using the estimates mentioned in (ii) we show the following (see
Theorem \ref{contdist}). For all $\ep >0$ there
exists a $\de >0$ such that 
\begin{eqnarray}
&&|d(x,y,t) - d(x,y,s)|  \leq \ep
\end{eqnarray}
for all $x,y \in M$ for all $ t,s \in [0,T)$ with $|t-s| \leq \de$.

\hfill\break
(iv) {\bf Convergence of  $(M,d(g(t)))$ to a $C^0$ Riemannian orbifold $(X,d)$ as
$t \upto T$.}
\hfill\break
Using the estimates mentioned in (i),(ii) and (iii),
we show that $(M,d(g(t)) \to (X,d_X)$ as $t \upto T$ in the Gromov-Hausdorff sense, where $(X,d_X)$ is a $C^0$-Riemannian
orbifold with finitely many orbifold points, and that the Riemannian
orbifold metric on $X$ is smooth away from the orbifold points. 
Also: the convergence is smooth away from the orbifold points (see Lemma \ref{fdeflemm} and  Theorems
\ref{manifoldstructure}, \ref{rmanifoldstructure}, \ref{orbifoldstructure}).

\hfill\break 
(v) {\bf The flow may be continued past time $T$ using the orbifold
Ricci flow.}\\
There exists a smooth solution $(N,h(t))_{t\in (0,\hat T)}$ to the orbifold
Ricci flow, such that\\ $(N,d(h(t))) \to (X,d_X)$ in the
Gromov-Hausdorff sense as $t \downto 0$
(see Theorem \ref{extend}).
\hfill\break

In another paper, \cite{BZ}, which recently appeared, the authors 
also consider Ricci flow of four manifolds with bounded scalar
curvature, and they also investigate the structure
of the limiting space one obtains by letting  $t \upto
T$: see Theorem 1.8 and Corollary 1.11 of \cite{BZ}.

\section{Setup, background, previous results and notation}
In this paper we often consider solutions $(M^4,g(t))_{t\in [0,T)}$ which satisfy the following {\it basic assumptions}.
\begin{itemize}
\item[(a)] $M^4$ is a smooth,  compact, connected four dimensional manifold without boundary
\item[(b)]  $(M^4,g(t))_{t \in [0,T)}$ is a smooth solution to the
  Ricci flow $\partt g(t) = -2\Ricci(g(t))$  for all $t \in [0,T)$ 
\item[(c)] $T< \infty$
\item[(d)] $\sup_{M^4\times[0,T)} |\Sc(x,t)| \leq 1$
\end{itemize}
If instead of $(d)$ we only have $\sup_{M\times[0,T)} |\Sc(x,t)| \leq
K < \infty$ for some constant $1<K<\infty$, then we may rescale the
solution $\ti g(\cdot,\ti t):= K g(\cdot, \frac {\ti t}{K})$ to obtain
a new solution $(M,\ti g(\ti t))_{t \in [0,\ti T)}$, where $\ti T := K T$, which satisfies the basic assumptions. 
As we mentioned in the introduction, any solution satisfying the basic
assumptions also satisfies
\begin{eqnarray}
&&\sup_{ t\in [0,T)} \int_M |\Riem(\cdot,t)|^2 d\mu_{g(t)} \leq K_0 <
\infty \label{riemint1} \\
&&\int_0^T\int_M |\Rc|^4(\cdot,t)d\mu_{g(t)} dt\leq c_2 < \infty. \label{ricciint1}
\end{eqnarray}
See Theorem 3.6 in \cite{Si1}. The estimate \eqref{riemint1} was
independently obtained in \cite{BZ} (see Theorem 1.8 of that paper),
using different methods to those used in \cite{Si1}.

There are many papers in which conditions are considered which imply
that the solution to Ricci flow defined on $[0,T)$ may be
extended. Generally, in the real case, this extension is a smooth
extension, and the conditions imply that the solution may be smoothly
extended to a time interval $[0,T+\ep)$ for some $\ep >0$: that is, the solution does
not form a singularity as $t\upto T$. 
Here we list some of these conditions. This is by no means an
exhaustive list and further references may be found in the papers we
have listed here.
In the following we assume that $(M^n,g(t))_{t\in [0,T)}$ is a
smooth solution to Ricci flow on a compact $n-$dimensional manifold
without boundary, and we write the condition which guarantees, that one can
extend the solution past time $T$, followed by an appropriate reference.
$\sup_{M^n \times [0,T} |\Riem| < \infty$ \cite{HaThree}.
$\sup_{M^n \times [0,T)} |\Ricci| < \infty$  \cite{Sesum}. $\limsup_{t \upto T} |g(t)-h|
\leq \ep(n)$ for some smooth metric $h$  \cite{SimC0} (see  also \cite{KL}). 
$\sup_{ (x,t) \in M^n\times[0,T) } |\Riem(x,t)|(T-t) + |\Sc(x,t)| <
\infty$ \cite{TME} (see also \cite{SesumLe}).\\
$\int_0^T \int_{M^n} |\Rm|^{\al} (\cdot,t))
d\mu_{g(t)} dt <\infty$ for some $\al \geq \frac{(n+2)}{2}$  \cite{Wang1}.\\
$\int_0^T \int_{M^n} |\Weil|^{\al} (\cdot,t)+ |\Sc|^{\al}(\cdot,t)
d\mu_{g(t)} dt <\infty$, where $\al \geq \frac{(n+2)}{2}$  
\cite{Wang1}. \\See also \cite{Wang1}, \cite{Wang2}, \cite{ChenWang} for further results on extending Ricci flow.

If one considers  solutions to the K\"ahler Ricci flow, $\partt
g_{i\bar j} = -2\Ric_{i\bar j} $, then the following is
known: If $\sup_{M^n \times [0,T)} |\Sc| < \infty $, then one can
extend the flow smoothly past time $T$ \cite{Zhang}.

The situation in this paper is somewhat
different. We consider solutions with bounded scalar curvature, and
we do not rule out the possibility that singularities can form as $t
\upto T$. However, using our integral curvature estimates (and other
estimates) we show that there is a singular limiting
space as $t \upto T$, and that this singular space is a $C^0$
Riemannian orbifold which can then be evolved
by the orbifold Ricci flow: the limiting space is immediately
smoothed out by the orbifold Ricci flow.

The possibility of flowing to a singular time and then continuing with
another flow (for example orbifold Ricci  flow or a weak K\"ahler Ricci
flow) has been considered in other papers.
In the real case, see for example \cite{CTZ}. 

In the K\"ahler case see for example Theorem 1.1 in
\cite{SongWeinkove2} (see also \cite{SongWeinkove1}, \cite{EGZ} and
\cite{EGZII} for related papers).
Further references can be found in the papers mentioned above.

 In \cite{ChenWang}, the authors investigate the moduli space of 
solutions to Ricci flow which have: bounded curvature in the  $L^{n/2}$
sense, bounded scalar curvature and are non-collapsed.

There are examples of solutions to Ricci flow which are smooth on
$[0,T)$,  singular at time $T$,
 and then become immediately smooth again after this time: see the
 neck-pinching examples given in \cite{ACK}. See also \cite{KlLo}
and \cite{FIK}.
This notion of {\it extending the flow} is once again different to the
one we are considering, and different to
the notion of smooth extension discussed above.

The Orbifold Ricci flow and related flows has been studied in many papers. Here is a (by
no means exhaustive) list of some of them: \cite{CTZ},
\cite{ChenYWangI}, \cite{ChenYWangII}, \cite{ChowII}, \cite{ChowWu},
\cite{HaThreeO}, \cite{KLThree}, \cite{LiuZhang}, \cite{WuLF},
\cite{Yin}, \cite{YinII}.

\hfill\break
{\bf Notation}:
\hfill\break
We use the Einstein  summation convention, and
we use the notation of Hamilton \cite{HaThree}.\hfill\break
For $i \in \{1,\ldots,n\}$,
$ \parti{}{x^i}$ denotes a coordinate vector, and $dx^i$  is the corresponding one
form.\hfill\break
$(M^n,g)$ is an $n$-dimensional  Riemannian manifold with Riemannian metric $g$.
\hfill\break
$g_{ij} = g(\parti{}{x^i},\parti{}{x^i}) $ is the Riemannian metric $g$
with respect to this coordinate system.\hfill\break
$g^{ij}$ is the inverse of the Riemannian metric ($g^{ij}g_{ik} = \delta_{jk}$).\hfill\break
$d\mu_{g}$ is the volume form associated to $g$.\hfill\break
$\Rm(g)_{ijkl} = {{}^g\Riem}_{ijkl} = \Riem(g)_{ijkl} = \Sc_{ijkl}$ is the full
Riemannian curvature Tensor.\hfill\break
$\Weil(g)_{ijkl}$ is the Weil Tensor.\hfill\break
${{}^g\Rc}_{ij} = \Ricci_{ij} = \Sc_{ij}:= g^{kl}\Sc_{ikjl}$ is the Ricci curvature.\hfill\break
$\Sc:= \Sc_{ijkl}g^{ik}g^{jl}$ is the scalar curvature. \hfill\break
${}^{g}\grad T = \grad T$ is the covariant derivative of $T$ with respect to $g$. For example,
locally $\grad_i T_{jk}^s = (\grad T)( \parti{}{x^i}, 
\parti{}{x^j}, \parti{}{x^k}, d x^s)$
(the first index denotes the direction in which the covariant derivative
is taken) if locally $T = T_{jk}^s dx^j \otimes dx^k \otimes \parti{}{x^s}$. \hfill\break
$|T| = {{}^g|T|}$ 
is the norm of a tensor with respect to a metric $g$. For example
for $T = T_{jk}^s dx^j \otimes dx^k \otimes \parti{}{x^s}$.
$|T|^2 = g^{im}g^{jn}g_{ks} T_{ij}^s T_{mn}^k$.\hfill\break
Sometimes we  make it clearer which Riemannian metric we are considering by
including the metric in the definition.
For example $\Sc(h)$ refers to the scalar curvature of the Riemannian
metric $h$.\hfill\break
We suppress the $g$ in the notation used for the norm, $|T| = {}^g|T| $,
and for other quantities,
in the case that is is clear from the context which Riemannian metric
we are considering.\\
A ball of radius $r>0$ in a metric space $(X,d)$ will be denoted by \\
${{}^d B}_{r}(z) := \{ x \in X \ | \ d(x,z) <r
\}$.\\
An annulus of inner radius $0\leq s$ and outer radius $r>s$  on a metric space $(X,d)$ will be
denoted by
${{}^d B}_{r,s}(z) := \{ x \in X \ | \ s <d(x,z) <r
\}$.\\
Note then that 
${{}^d B}_{0,s}(z) := \{ x \in X \ | \ 0 <d(x,z) <r
\} = {{}^d B}_{s}(z)  \backslash \{z\}.$\\
The sphere of radius $r>0$ and centre point $p$ in a metric space $(X,d)$ will be
denoted by\\
${{}^d S}_r(p) := \{ x \in X \ | \ d(x,p) =r\}$. \\
$D_{r,R} \subseteq \R^n$ is the standard
open annulus of inner radius $r\geq 0$ and outer radius $R\leq
\infty$, ($r<R$) centred at $0$:\\
 $D_{r,R} = \{ x \in \R^n \ | \ |x| >r, |x| <R \}$.\\
$D_r$ represents the open disc of radius $r$ centred at $0$:\\$D_r:=
\{ x \in \R^n \ | \ |x| <r \}$. \\Note $D_{0,R} = \{ x \in \R^n \ | \
|x| >0, |x| <R \} = D_R \backslash \{0\}$.
\\
$S^{n-1}_r(c):= \{ x \in \R^n \ | \ |x-c| =r\}$ is the
$(n-1)$-dimensional sphere 
 of radius $r>0$ and centre point $c \in \R^n$ in $\R^n$.\\
$\omega_n $ is the volume of  a ball of radius one in $\R^n$ with
respect to the Lebesgue meaure.\\
If $\Gamma$ is a finite subgroup of $O(n)$ acting on $\R^n$, then 
$((\R^n\backslash \{0\})/\Gamma, g) $ is the quotient manifold with
the induced (flat) metric coming from $\pi: \R^n \backslash \{0\} \to
(\R^n\backslash \{0\})/\Gamma$, $\pi(x) := \{ [x] \ | \ x \in
\R^n\backslash \{0\} \}$, where $[x] := \{ G x \ | \ G \in \Gamma \}$.\\
$({{}^g B}_{r,s}(0),g)  \subseteq ((\R^n\backslash \{0\})/\Gamma, g) $
refers to the set\\
${{}^g B}_{r,s}(0)  := \{ \pi(x) \ | \ x \in D_{r,s} \}$ with the
Riemannian metric $g$.\\

\section{Volume control, and the Sobolev inequality}

In \cite{Ye} and \cite{Zhang1,Zhang2} the first inequality appearing
below was proved, and in \cite{Zhang3} (and in \cite{ChenWang}) the second inequality appearing below was proved.
\begin{theo}(Ye, R.\cite{Ye}, Zhang, Q.\cite{Zhang1,Zhang2,Zhang3}
  (see \cite{ChenWang} also))\\
Let $(M^n,g(t))_{t\in[0,T)}$, $T< \infty$,  be a smooth solution to Ricci
flow on a closed manifold with  $\sup_{M\times[0,T)} |\Sc(x,t)| \leq 1 < \infty$. 
Then there exist constants $0<\si_0 , \si_1< \infty$ depending only on $(M,g_0)$
and $T$  such that
\begin{eqnarray}
\si_1 \leq \frac{\vol({{}^t B}_r(x))}{r^n}\leq \si_2 \mbox{ for all } x \in
M , 0\leq t <T \mbox{ and } r \leq 1. \label{nocollapsenoinflate}
\end{eqnarray}

\end{theo}

We use the following notation in this paper which was introduced by
Q. Zhang.
A solution which satisfies the first inequality is said to be  {\it
 $\si_1$  non-collapsed on scales less than $1$}. This condition is 
similar to but stronger than Perelman's {\it non-collapsing } condition
(see \cite{Pe1}), as we make no requirements on the curvature within
the balls $B_r(x)$ appearing in \eqref{nocollapsenoinflate}.
A solution which satisfies the second inequality is said to be 
{\it $\si_2$ non-inflated on  scales less than $1$}. 
\begin{remark}
Let $(M^n,g(t))_{t\in [0,T)}$ be be a smooth solution to Ricci flow
which satisfies the inequalities \eqref{nocollapsenoinflate}, and define
$\ti g(\ti t) : =cg(\cdot,
\frac{\ti t}{c})$ for a constant $c>0$. Then
\begin{eqnarray}
\si_1 \leq \frac{\ti \vol( {{}^{\ti t} \ti B}_{\ti r}(x)}{{\ti r}^n}) \leq \si_2 \mbox{ for all } x \in
M , 0\leq \ti t <\ti T :=  cT\mbox{ and } \ti r \leq \sqrt{c}, 
\end{eqnarray}
that is $(M,\ti g(\ti t))_{\ti t\in [0,\ti T)}$ is 
{ \it $\si_1$  non-collapsed and $\si_2$ non-inflated on scales less than $\sqrt{c}$}.
This is because: $ \frac{\ti \vol \ti B_{\ti r}(x,\ti t)}{{\ti r}^n} =
\frac{\vol  B_r(x, t)}{{r}^n} $ for $\ti r := \sqrt c r$ and $\ti t
:= c t$, and $r = \frac{ \ti  r }{\sqrt c} \leq 1$ for $\ti r \leq \sqrt{c}$.
Hence, we can't say if  the solution  $(M,\ti g(\ti t))_{\ti t\in
  [0,\ti T)}$ is $\si_1$ non-collapsed and
$\si_2$ non-inflated on scales less than $1$, if we scale by a
constant $c<1$, but the scale improves if we multiply by constants $c>1$.
\end{remark}

In the papers \cite{Ye} and \cite{Zhang1,Zhang2}
it is also shown that  for any Ricci flow satisfying the basic
assumptions a
Sobolev inequality holds in which the constants may be chosen to be time independent.
Here, we only write down the four dimensional version of their theorem.

\begin{theo}(Ye, R. \cite{Ye}, Zhang, Q.\cite{Zhang1,Zhang2})\\
Let $(M^4,g(t))_{t\in[0,T)}$, $T< \infty$,  be a smooth solution to
Ricci flow satisfying the basic assumptions.
Then there exists a constant $A = A(M,g_0,T) < \infty$ such that
\begin{eqnarray}
&& (\int_M |f|^4 d\mu_{g(t)})^{\frac 1 2} \leq A \Big( \int_M ^{g(t)}|\grad
f|^2 d\mu_{g(t)} + \int_M |f|^2 d\mu_{g(t)} \Big) \label{sobolevyezhang}
\end{eqnarray}
for all smooth $f:M \to \R$
\end{theo}
Note that this Sobolev inequality is not scale invariant,
as the last term scales incorrectly.
However, we have a scale-invariant version for small balls, as we see
in the following:
\begin{corollary}
Let $(M^4,g(t))_{t\in[0,T)}$, $T< \infty$  be a smooth solution to
Ricci flow satisfying the basic assumptions.
Then there exists a constant $r^2 =
r^2(M,g(0),T) =   \frac{1}{2\sqrt{\si_2}A}
   >0$ such that
\begin{eqnarray}
(\int_M |f|^4 d\mu_{g(t)})^{\frac 1 2} \leq 2A  \int_M {{}^{g(t)}|\grad
f|}^2 d\mu_{g(t)}\label{sobo}
\end{eqnarray}
for all smooth $f:M \to \R$ whose support is contained in a ball
$^tB_r(x)$, for some $x \in M$, where $A$ is the constant occurring in
the Sobolev inequality \eqref{sobolevyezhang} above.
If $\ti g (\cdot, \ti t):= cg(\cdot, \frac {\ti t}{c})$ is a scaled
solution with $c\geq 1$ then the estimate 
\begin{eqnarray}
(\int_M |f|^4 d\mu_{\ti g(\ti t)})^{\frac 1 2} \leq 2A  \int_M
 {{}^{g(t)}|\grad f|}^2 d\mu_{\ti g(\ti t)}\label{sobo2}
\end{eqnarray}
holds for all 
$f:M \to \R$ whose support is contained in a ball
${}^{\ti t}B_{\ti r}(x)$ where $\ti r := r \sqrt c  \geq r$.
\end{corollary}

\begin{proof}
Let $r$ be chosen so that $r^2\sqrt{\si_2} \leq \frac{1}{2A}$, where $A$ is the
constant occurring in the Sobolev inequality and $\si_2$ is the
non-inflating constant defined above.
Using H\"older's inequality and the above Sobolev inequality we get
\begin{eqnarray}
(\int_M |f|^4 d\mu_{g(t)})^{\frac 1 2} &&\leq A \int_M |\grad
f|^2 d\mu_{g(t)} + A \int_M |f|^2 d\mu_{g(t)} \cr
&& \leq A  \int_M |\grad
f|^2 d\mu_{g(t)} + A(\int_M |f|^4 d\mu_{g(t)})^{\frac 1 2} (\vol
B_r(x,t))^{\frac 1 2}\cr
&&\leq A  \int_M |\grad
f|^2 d\mu_{g(t)} + A(\int_M |f|^4 d\mu_{g(t)})^{\frac 1 2}
(\sqrt{\si_2}r^2)\cr 
&& \leq A  \int_M |\grad
f|^2 d\mu_{g(t)}  + \frac 1 2 (\int_M |f|^4 d\mu_{g(t)})^{\frac 1 2}
\end{eqnarray}
which implies the result, after subtracting $\frac 1 2 (\int_M |f|^4
d\mu_{g(t)})^{\frac 1 2}$ from both sides of this inequality.
The second inequality follows immediately from the fact that 
\begin{eqnarray}
&& (\int_M |f|^4 d\mu_{\ti g(\ti t)})^{\frac 1 2} -2A  \int_M |\ti \grad
f|^2 d\mu_{\ti g(t)}\cr
&& = c (\int_M |f|^4 d\mu_{g( t)})^{\frac 1 2} -2A  \int_M | \grad
f|^2 d\mu_{ g(t)}) 
\end{eqnarray}
if we scale as in the statement of the theorem.
\end{proof}

It is well know that,  for a solution satisfying the basic assumptions,
the volume of $M$ is changing at a controlled rate:
\begin{eqnarray}
 \vol(M,g(t)) \geq -\int_M R d \mu_{g(t)} = \partt \vol(M,g(t)) \geq
-\vol(M,g(t))\label{volin}
\end{eqnarray}
($-\int_M R d \mu_{g(t)} = \partt  \vol (M,g(t)) $ was shown in \cite{HaThree}).
Integrating in time we see that  $e^{T} \vol(M,g(0)) \geq \vol(M,g(t)) \geq e^{-T} \vol(M,g(0))$.

Notice that the estimates of Peter Topping (see \cite{Topping}) and these 
volume bounds combined with the non-inflating estimate guarantee that the diameter is bounded from
above and below:
\begin{lemma}(Topping, P. \cite{Topping},  Zhang, Q. \cite{Zhang1,Zhang2})\label{diamlemma}\\
Let $(M^4,g(t)))_{t\in [0,T)}$ be a solution satisfying the basic
assumptions (in particular $T< \infty$ and $|\Sc| \leq 1$ at all
times and points).
Then there exists $d_0= d_0(M,g_0,T) >0$ such that
\begin{eqnarray}
\infty > d_0 \geq \diam(M,g(t)) \geq \frac 1 {d_0} >0 \label{toppingresult}
\end{eqnarray}
for all $t \in [0,T)$.
 \end{lemma}
\begin{proof}
The diameter bound from above follows immediately from Theorem 2.4 (
see also Remark 2.5 there)  of
\cite{Topping} combined with the fact that $\int_M |R|^{\frac{3}{2}} \leq  \vol(M,g(0))e^T$
for a solution satisfying the basic assumptions. 
The diameter bound from below is obtained as follows.
Assume that there are times $t_i \in [0,T)$ with $\ep_i:= \diam(M,g(t_i)) \to
0 $ as $ i \to \infty$. Due
to smoothness, we must have $t_i \upto T$.
From the volume estimates above, we must have
$\vol(M,g(t)) \geq e^{-T}\vol(M,g(0))=:v_0 >0$ for all $t \in [0,T)$.
Combining this with the non-inflating estimate we get:
$$v_0 \leq \vol(M,g(t_i)) = \vol({{}^{t_i}B}_{\ep_i}(x_0)) \leq \si_2
(\ep_i)^4 \to 0$$ as $i \to \infty$, which is a contradiction.
\end{proof}

\section{The regular part of the flow}
We wish to show that the limit as $t \upto T$ (in some to
be characterised sense) of
$(M,g(t))$ is an $C^0$ Riemannian orbifold $(X,d_X)$ with at most finitely many
orbifold points and that $(X,d_X)$ is smooth away from the orbifold
points. In the static case, M. Anderson showed
results of this type for sequences of Einstein manifolds 
whose curvature tensor is bounded in the $L^{n/2}$ sense: see 
for example Theorem 1.3 in \cite{And1}.  
Similar results were shown independently by 
\cite{BKN} (see Theorem 5.5 in \cite{BKN}). See also \cite{Tian}. In the paper \cite{AnCh}, 
the condition that the manifolds have Ricci
curvature bounded from above and below 
or bounded  Einstein constant was replaced by the
condition that  the  Ricci curvature is bounded from below. To deal
with this situation the authors introduced the $W^{1,p}$ harmonic
radius, which we also use here.

To prove the convergence to an orbifold and to obtain information on the orbifold
points we require {\it regularity estimates }  for regions where 
$\int_{^t B_r(x)} |\Riem(g(t))|^2 d\mu_{g(t)} $ is {\it small}.
Regularity estimates  in the static case (for example the Einstein
case) were shown for example in
Lemma 2.1 in \cite{And2}.  
We show that for certain so called {\it good times} $t<T$, which are
close enough to $T$,  that if
$\int_{^t B_{r(t)} (x)} |\Riem(g(t))|^2 d\mu_{g(t)} \leq \ep_0 $ is
{\it small enough}, where $r(t) =R\sqrt{T-t}$ for some large $R>0$,
then we will have time dependent  bounds on the metric 
on  the  ball  ${{}^t B}_{r(t)/2}(x)$ for later times $s$, $t \leq s <
T$:
see Theorem \ref{regularregion} below for the explicit bounds (the constants $\ep_0, R$ appearing above, will not depend on $x$).
That is, we have a {\bf fixed set  ${{}^t B}_{r(t)/2}(x)$  where we
obtain our estimates for later times $s \in [t,T)$ (that is, the set ${{}^t
  B}_{r(t)/2}(x)$ doesn't depend on $s$).}
Furthermore, we show that the metric $g(s)$ on the ball ${{}^t B}_{r(t)/2}(x)$ is $C^0$ close to the metric $g(l)$ on ${{}^t B}_{r(t)/2}(x)$ if  $s,l \in [t,T)$ and $|s-l|$ is small enough.

In order to obtain our regularity estimates we require a number of
ingredients. The estimates from 
the previous section, a slightly modified version of a result from
\cite{And1} and \cite{AnCh} on the $W^{1,p}$ harmonic radius (see also Lemma 4.5 of
\cite{Petersen}), a Nash-Moser-de Giorgi argument, and the  {\it
  Pseudolocality} result of G. Perelman (see Theorem 10.1 of \cite{Pe1}) being
the main ones.
The Nash-Moser-de Giorgi argument which we use is a modified
version of that given in the paper \cite{Li}. The
proofs in the paper of \cite{Li} are written for a four
dimensional setting, and can be adapted to our setting.

Before stating the theorem we introduce some notation, which we will
also use in the subsequent sections of this paper.

Let $(M^4,g(t))_{t\in [0,T)}$ be a solution to Ricci flow satisfying
the basic assumptions.
In Theorem 3.6 of \cite{Si1}, it was shown that
\begin{eqnarray} 
\int_{S}^R \int_M |\Rc|^4(\cdot,t) d\mu_{g(t)} dt \leq K_0 =
K_0(M,g_0,T) < \infty \label{Ricci4}
\end{eqnarray} 
for $S<R \leq  T$.
In particular, for any $0<r<\frac T 4$, and $1 \geq\si>0$, we can find a $t \in
[T-(1+\si)r,T-r]$ such that 
\begin{eqnarray} 
\int_M |\Rc|^4(\cdot,t) d\mu_{g(t)}  \leq  \frac {2K_0} {\si r} \label{goodtime1}
\end{eqnarray} 
If not, then we can find $\si$ and $r$ such that
$ \int_M |\Rc|^4(\cdot,t) d\mu_{g(t)}  > \frac {2K_0} {\si
  r} $ for all $t \in
[T-(1+\si)r,T-r]$, and hence
$$ \int_{T-(1+\si) r}^{T-r} \int_M |\Rc|^4(\cdot,t) d\mu_{g(t)}
>\si r  \frac {2K_0} {\si r} = 2K_0$$ which contradicts  equation
\eqref{Ricci4}.

If $ t:= T-r <T$ is given, where  $ r< \frac{T}{10}$, then the argument above shows that we can
always find a (nearby) $\ti t \in [T-2r,T-r]$ such that 
\begin{eqnarray}
 \int_M |\Rc|^4(\cdot,\ti t) d\mu_{g(\ti t)}  \leq \frac
{2K_0} {r}  =\frac{2K_0}{T-t}   \leq \frac {4K_0} {T-\ti t}. \label{4k_0goodtime}
\end{eqnarray}

A time $\ti t$ which satisfies \eqref{4k_0goodtime} will  be known as a $4K_0$ {\it good time}. 
More generally, we make the following definition.

\begin{defi}
Let $(M,g(t))_{t\in [0,T)}$ be a smooth solution to Ricci flow.
Any $t \in [0,T)$ which satisfies 
\begin{eqnarray}
 \int_M |\Rc|^4(\cdot,t) d\mu_{g( t)}  \leq \frac {C} {T- t} \label{Cgoodtime}
\end{eqnarray}
($C>0$) shall be called a {\bf $C$-good time}.
If $C=1$, then we call such a $t$ a {\bf good time}.
\end{defi}

By modifying the above argument we see that the following is true.

\begin{lemm}\label{goodtimelemma}
Let $(M^4,g(t))_{t\in [0,T)}$ be a solution to Ricci flow satisfying
the basic assumptions and let $C >0$ be given. Then there exists an
$\ti r>0$ such that for all $0<r<\ti r$ the following holds.
For any $\ti t \in [0,T)$ with $r :=  T - \ti t$ there exists a $t \in
[\ti t-r, \ti t]= [  T-2r , T-r]$
which is a $C$ good time.
\end{lemm}
\begin{remark}
$\ti r$ will possibly depend on $C$, $(M,g(0))$ and $T$ as can be seen in
the proof below. 
\end{remark}

\begin{proof}
Fix $C>0$ and 
assume the conclusion of the theorem doesn't hold. Then we can find a sequence $r_i \to 0$
and $\ti t_i:= T-r_i \upto T$ 
such that every $t \in [T-2r_i, T-r_i]$ is  {\bf not} a $C$ good
time. That is $\int_M |\Rc|^4(\cdot,t) d\mu_{g( t)}  > \frac {C}
{T- t} $ for all $t \in  [T-2r_i, T-r_i]$. 
Integrating in time from $T-2r_i$ to $T-r_i$ we get
\begin{eqnarray*}
\int_{T-2r_i}^{T-r_i} \int_M |\Rc|^4(\cdot,t) d\mu_{g( t)} dt &&> C
\int_{T- 2r_i}^{T-r_i} \frac{1}{T-t}dt \cr
&& \geq  \frac{C}{2r_i} \int_{T- 2r_i}^{T-r_i} dt \cr
&& = \frac{C}{2}.
\end{eqnarray*}
Without loss of generality the intervals $[T-2r_i, T-r_i]_{i\in \N}$
are pairwise disjoint (since $r_i \to 0$).  
Summing over $i \in \N$ we get
\begin{eqnarray*}
\int_{0}^{T} \int_M |\Rc|^4(\cdot,t) d\mu_{g( t)} dt &&
\geq \sum_{i=1}^{\infty} \int_{T-2r_i}^{T-r_i} \int_M
|\Rc|^4(\cdot,t) d\mu_{g( t)} dt \cr
&& \geq \sum_{i=1}^{\infty}\frac{C}{2} = \infty
\end{eqnarray*}
which contradicts the fact that $\int_{0}^{T} \int_M
|\Rc|^4(\cdot,t) d\mu_{g( t)} dt < \infty$.
\end{proof}
Let $0<t_i \upto T$, $i \in \N$  be a sequence 
of times approaching $T$ from below.   We wish to show that
$(M,g(t_i)) \to (X,d)$ as $i\to \infty$ in some to be
characterised sense, where $(X,d)$ is a $C^0$ Riemannian orbifold with only finitely
many orbifold points.
These orbifold points will be characterised by the fact that they are
points where the $L^2$ integral of curvature concentrates as $t_i \upto T$.
To explain this more precisely we introduce some notation.
\begin{defi}\label{defreg}
Let $(M^4,g(t))_{t\in [0,T)}$ be a solution to Ricci flow with $T<
\infty$ satisfying the basic assumptions.
A point $p \in M$ is a {\it regular point in $M$} (or $p \in M$ is
{\it regular}) if there exists an $r = r(p)
>0$ such that
$$\int_{^{t} B_r(p)} |\Riem|^2(\cdot,t)d\mu_{g(t)} \leq \ep_0$$ for all times  $t \in
[0,T)$, where $\ep_0>0$ is a small fixed constant depending on $(M^4,g(0))$ and $T$,
which will be specified in the proof of Theorem \ref{regularregion} below.
A point $p \in M$ is a {\it singular point in $M$} (or $p\in M$ is
{\it singular}) if $p\in M$ is not a regular
point. In this case, due to smoothness of the flow on $[0,T)$, there must exist a sequence of
times $s_i \upto T$ and a sequence of numbers $0<r_i \downto 0$  as $i
\to \infty$ 
such that $ \int_{  {{}^{s_i} B}_{r_i}(p)} |\Riem|^2 > \ep_0$ for all $i\in
\N$. We denote the set of regular points in $M$ by $\reg(M) := \{ p \in M \ | \ p$ is
regular $\}$ and the set of singular points in $M$ by
$\sing(M):= \{ p \in M \ | \ p$ is singular $\}$.
\end{defi}

In this section we obtain information about regular points.
In particular we will give another characterisation of the property
{\it regular}. This characterisation is implied by  the  following
theorem (see the Corollary directly after the statement of the
Theorem).
\begin{theo}\label{regularregion}
Let $k \in \N$ be fixed, and let $(M,g(t))_{t\in [0,T)}$ be a solution to Ricci flow satisfying the
basic assumptions.
There exists a (large) constant $R>0$, and  (small) constants
$v,\ep_0>0$, and constants $c_1, \ldots, c_k$  such that if
\begin{eqnarray}
\int_{    {{}^{t} B}_{R\sqrt{T-t} }   (p)} |\Riem|^2(\cdot,t)d\mu_{g(t)} \leq \ep_0 \label{regt}
\end{eqnarray}
 for a good time
 $t$ which satisfies $|T-t| \leq v$, then $p$ is a regular point.
We also show that 
if $p,t$ satisfy these conditions, then 
\begin{eqnarray}\label{betterreg}
 &&    \exp(-\frac{8 |r^{\frac 1 4}  - s^{\frac 1 4} |}{
     (T-t)^{\frac 1 4}} ) g(r) \leq g(s) \leq \exp(\frac{8 |r^{\frac 1 4}  - s^{\frac 1 4} |}{
     (T-t)^{\frac 1 4} } )  g(r), \mbox{  and }\label{metricest}\\
&& \frac{1}{2} g(r) \leq g(s) \leq 2
g(t) \ \
 \forall \   t \leq
r , s<T,    \mbox{ on } {{}^{t}
  B}_{ \frac{R}{2} \sqrt{T-t}}   (p) \label{metricest2}  \\
&& |\grad^j \Riem(x,s)|^2_{g(s)}  \leq  \frac{c_j}{(T-t)^{j+2}}
\label{shireg} \\
 && \forall \ t + \frac{(T-t)}{2}\leq s <T,  x \in  {{}^{t}  B}_{\frac{R}{2}  \sqrt{T-t} }(p),\cr
&& \forall \ j \in \{0,\ldots k\}.
\end{eqnarray}
The constants $\ep_0, R$ and $v$  depend only on
$\si_0,\si_1$ from \eqref{nocollapsenoinflate},  $A$ from
\eqref{sobo2}, and $c(g(0),T)$ from Theorem \ref{regularregion}, the constants $c_j$
depend only on $j,\si_0,\si_1, A$ and $c(g(0),T)$. That is, all
constants depend only on $(M,g(0))$ and $T$.

For such $p$ and $t$ we therefore have: all $x \in {{}^{t}  B}_{R \sqrt{T-t}/2}   (p) $ are
also regular (see the proof for an explanation), and there is a limit  in the {\it smooth  sense}
(and hence also in the Cheeger-Gromov sense) of $( {{}^{t}  B}_{ \frac{R}{2} 
  \sqrt{T-t}}   (p) ,g(s))$ as $s \upto T$.
\end{theo}

\begin{remark}
The condition  $\int_{    {{}^{t} B}_{R\sqrt{T-t}}   (p)} |\Riem|^2(\cdot,t)d\mu_{g(t)} \leq \ep_0$ for a good time
  $t$ which satisfies $|T-t| \leq v$  ($v$ , $\ep_0$ as in the statement of the
  Theorem above)  therefore implies
that {\it $p$ is regular} (see the proof for an explanation).
This new condition contains however more information, namely
that the estimates appearing in the statement of Theorem
\ref{regularregion} hold.  
Furthermore: to show that a point $p \in M$ is regular, we only need to
find {\bf one} good time $t$ with  $|T-t| < v$ for which  
$\int_{    {{}^{t} B}_{R\sqrt{T-t}}   (p)} |\Riem|^2(\cdot,t)d\mu_{g(t)}  \leq \ep_0$. We
do {\bf not} need to show that \\ $\int_{    {{}^{t} B}_{r(p)}   (p)}
|\Riem|^2(\cdot,t)d\mu_{g(t)}  \leq \ep_0$ for all $t<T$ for some fixed $r(p)>0$.

This characterisation is useful when it comes to showing that a
limit space (in a sense which will be explained later in this paper) $(X,d_X) := \lim_{t\upto T}(M,g(t))$ exists and when it comes to
describing its structure.
\end{remark}
\begin{defi}\label{defregII}
Let $ t \in (0,T)$.
We say $p \in \regt(M)$ if $$\int_{    {{}^{t} B}_{R\sqrt{T-t} }   (p)}
|\Riem|^2(\cdot,t)d\mu_{g(t)}  \leq \ep_0, $$ where $\ep_0,R$ are from the above theorem.
\end{defi}
\begin{remark}
Notice that this condition is scale invariant: if $(M,\ti g(\ti
t))_{\ti t \in [0,\ti T)}$ is the solution we get by setting $\ti
g(\ti t):= cg (\frac{\ti t}{c})$, $\ti T:= cT$, $\ti t =ct$, then  
\begin{eqnarray}
&&\int_{  {{}^{\ti t} B}_{R\sqrt{\ti T-\ti t} }   (p)}
|\ti \Riem|^2(\cdot,\ti t)  d\mu_{\ti g(\ti t)} =
\int_{  {{}^{ t} B}_{R\sqrt{T- t} }   (p)}
|\Riem|^2 (\cdot, t)  d\mu_{ g( t)} 
\leq \ep_0 
\end{eqnarray}

\end{remark}
\begin{coro}\label{correg1}
Theorem \ref{regularregion} above shows us that $\regt(M) \subseteq
\reg(M)$ for all good times $t \in (T-v,T)$. From the definition of
$\reg(M)$ we also see:  for all $p \in \reg(M)$ there exists a $T- v< S(p)<T$ such
that $p \in \regt(M)$ for all good times $t$ with $t \in (S(p),T)$. 
Furthermore, Theorem \ref{regularregion} above also tells us, that for every good time $t \in (T-v,T)$, and for all $\ep>0$,
there exists a $\de>0$ (depending on $t$), such that 
\begin{eqnarray}
&& (1-\ep)g(p,s) \leq g(p,r) \leq (1+\ep) g(p,s) \cr
&& \ \  \ \forall p \in \regt(M),
\fall r,s \in (t,T) \ \mbox{ with } \  |r-s| \leq \de.\label{metricreg}
\end{eqnarray}
\end{coro}
\begin{coro}\label{uniformdcoro}
For all good times $t \in (T-v,v)$ for all  $p \in \reg_t(M)$, where $v$
and $\reg_t(M)$ are as above,  we have
\begin{eqnarray}
  && \frac{1}{8} d(x,y,r) \leq d(x,y,s) \leq 8 d(x,y,r)  \cr
&& \ \ \ \ \fall r,s \in [t,T), \fall x,y \in {{}^{t} B}_{ \frac{R}{200} 
  \sqrt{T-t}}(p)   \label{uniformd}
\end{eqnarray}
\end{coro}

{\bf proof (of Theorem \ref{regularregion}) }:\hfill\break\noindent
Let  $t_i \upto T$ be a sequence of good times.
We scale (blow up) and shift (in time) the solution $g$ as follows: 
$g_i(t):=  \frac{1}{T-t_i}g(\cdot, T +   t(T-t_i))$. Then we have a
solution which is defined for $t \in [-A_i:= -\frac{T}{T-t_i}, 0)$ and
$A_i \to \infty$ as $i \to \infty$.
Furthermore, using the fact that the $t_i$ are  good times (for the
solution before scaling), we
see that  

\begin{eqnarray}
 \int_M |\Rc(g_i(-1))|^4 d\mu_{g_i( -1)}
 &&= (T-t_i)^2\int_M  |\Rc(g(t_i)|^4(t)d\mu_{g(t_i)}  \cr
&& \leq \de_i:=  (T-t_i) \to 0
\end{eqnarray}
as $i \to \infty$.
The scale invariant inequalities \eqref{nocollapsenoinflate} are also valid for $g_i(-1)$.

Let $B_R(p) = { {}^{g_i(-1)}  B}_{R}(p) \subseteq M $ be an arbitrary ball   with \\
$\int_{    { {}^{g_i(-1)}  B}_{R}(p)} |\Riem|^2 d\mu_{g_i(-1)} \leq \ep_0$ and $R\geq 4>0$.
Scaling by $\de =  \frac{4}{R^2} <1$ (*),  (that is $\ti g_i(\ti t):= \de
g_i( \frac{\ti t}{\de})$: we call the solution $\ti g_i(\ti t)$ once
again $g_i(t)$) we see that
\begin{itemize}
\item [(a)] $\int_{ { {}^{g_i(-\de) } B }_{2}(p)  } |\Riem|^2 d\mu_{g_i(-\de)} \leq \ep_0$
and $\int_M |\Rc|^4 d\mu_{g_i(-1)} \leq \ti \de_i$, \\where $\ti \de_i:=
\de_i /\de^2 = (T-t_i)/\de^2\to 0$ as $i \to \infty$ 
\item[(b)] we have control over non-inflating constants and
  non-decreasing constants: 
$\si_0 r^4 \leq \vol( { {}^{g_i(-\de) } B }_{r}(x) ) \leq \si_1 r^4$
for all $r \leq r_i \to \infty$ and for all $x \in { {}^{g_i(-\de) } B }_{2}(p)$. 
\end{itemize}
the works of Anderson \cite{And1} and Deane Yang \cite{YangD} imply that 
${}{B}_{1}(p)$ is in some $C^{0,\al}$ sense close to euclidean
space if $\ep_0$ is small enough, and $i\in \N$ is large enough (that is if $\de_i =(T-t_i)$ is
small enough). This is a fact about smooth Riemannian manifolds satisfying
(a) and (b),  and has nothing to do with the Ricci flow.
We state below a qualitative version of this
fact. Our proof method is essentially the same as the method used in the proof of Main
Lemma 2.2 in \cite{And1} (see Remark 2.3 (ii) there). We also use some notions from
\cite{AnCh} on the $W^{1,p}$ harmonic radius.

\begin{theo}\label{hrtheo}
Let  $(M^4,g)$ be a smooth connected manifold without boundary (not necessarily complete)  and
${B}_{2}(p) \subseteq M$ be an arbitrary ball which is compactly
contained in $M$. 
Assume that
\begin{itemize}
\item [(a)] $\int_{ B_{2}(p)  } |\Riem|^2 d\mu_{g} \leq \ep_0$
and $\int_M |\Rc|^4 d\mu_{g} \leq 1$, 
\item[(b)]  $\si_0 r^4 \leq \vol(B_r(x)) \leq \si_1 r^4$ for all $r
  \leq 1 $, for all $x \in B_2(p)$,
\end{itemize}
where $\ep_0= \ep_0(\si_0,\si_1)>0$ is sufficiently small.
Then there exists a constant $V= V(\si_0,\si_1)>0$ 
such that 
\begin{eqnarray}
r_h(g)(y) \geq V \dist_{g}(y, \boundary(B_{\frac 3 2}(p))), \label{harmin}
\end{eqnarray}
for all $r>0$, for all $y \in B_{3/2}(p)$, where $r_h(g)(y)$ is the
$W^{1,12}$ harmonic radius of $(M,g)$ at $y$ (see \eqref{harmonic} in
Appendix \ref{harmonicapp} for the definition of harmonic radius that we are using). 
\end{theo}
\begin{remark}
As noted above, this theorem does not require that the metrics
involved are coming from a Ricci flow.
\end{remark}
\begin{remark}
A different approach and a similar result is given in, respectively
obtained in, the paper by  Deane Yang \cite{YangD}.
\end{remark}
\begin{remark}
Compare with Theorem 2.35 of \cite{TZ}.
\end{remark}
A proof and the definition of the  $W^{1,12}$ harmonic
radius is given in Appendix \ref{harmonicapp}. The inequality from \eqref{harmin}  reads, in our case,\\
$ r_h(g_i(-\de))(y) \geq V \dist_{g_i(-\de)}(y, \boundary
({{}^{-\de}B}_{3/2}(p)))$ for all $ y \in {{}^{-\de}B}_{3/2}(p)$ if
$(T-t_i)/\de^2 \leq 1$. In particular  $ r_h(g_i(-\de))(y) \geq \frac{V}{4}
$ for all $ y \in {{}^{-\de}B}_{1}(p)$, if $(T-t_i)/\de^2 \leq 1$.
Comparing Perelman's definition of {\it almost euclidean } (see
Theorem 10.1 in \cite{Pe1} for the definition of {\it almost euclidean}) with the definition of harmonic radius we are using, we see that 
there is a constant $1>a= a(V) =a(\si_0,\si_1) >0$ such that 
${^{g_i(-\de)}B}_{a}(y)$ is almost euclidean if $(T-t_i)/\de^2 \leq 1.$
Notice that $a$ doesn't depend on $\de$ and hence,  without loss of
generality $ \de << a$:  $\de = \frac{4}{R^2}$ and $R>0$ was arbitrary up
until this point, so we choose $R^2 >> \frac{1}{4a}$.
Perelman's first Pseudolocality result (Theorem 10.1 in \cite{Pe1}) now tells us that 
\begin{eqnarray}
&&|\Riem(g_i)(x,t)| \leq \frac{1}{\de + t}, \ \mbox{ for all } \  t \ \in \ (-\de,0), \
x \in {}^{g_i(t)}B_{\ti a }(y)
\end{eqnarray} 
for some constant $\ti a = \ti a (a)>0$, for all $ y \in
{{}^{-\de}B}_{1}(p)$. Here we use that $\de<<a$
that is $0<\de=\de(V) << a(V)$ is chosen small so
that the Pseudolocality Theorem applies on the whole time
interval $(-\de,0)$. Without loss of generality $\de << \ti a$ also. Now $\de= \de(V)$ is
fixed (and small), that is $R= R(V) = \frac{2}{\sqrt{\de}} >>1$ is fixed (and large). 
Scaling back to $t=-1$ (that is we set $\ti g_i(\ti t) = \frac{R^2}{4}
g_i(\frac{4\ti t}{R^2})$ so that we are dealing with the solution we had
before blowing down  at the point (*) of the argument above: we call
the solution $\ti g_i(\ti t)$ once again $g_i(t)$ for ease of reading)
we have  
\begin{eqnarray}
|\Riem(g_i)(x,t)| \leq \frac{1}{1 + t}, \ \mbox{ for all } \  t \ \in \ (-1,0) ,
\ x \in {}^{g_i(t)}B_{\frac{R\ti a}{2} }(y) \label{curvperlinb}
\end{eqnarray} 
for all $y \in {}^{g_i(-1)}B_{\frac{R}{2} }(p).$
Using Shi's estimates (see \cite{Shi}), the non-inflating and
non-collapsing estimates, the evolution equation $\partt g = -2\Rc$, 
and the injectivity radius estimate of Cheeger-Gromov-Taylor (Theorem 4.3 in \cite{CGT}),  we get
\begin{eqnarray}
|\grad^j\Riem(g_i)(y,t)| \leq A_j, \ \mbox{ for all } \  t \ \in \
(-\frac{1}{2},0), \label{firstcurv}
\end{eqnarray} 
for all $0\leq j  \leq K$ where $K \in \N$ is fixed and large and $A_j <
\infty$ is a constant, for all $y \in {{}^{-1}B}_{\frac{R}{2}}(p)$, as long as $R \ti a$ is sufficiently large: as we chose $\de << \ti a$,
this is without loss of generality the case.
Translating in time and scaling back to the original solution, we
obtain the claimed curvature estimates \eqref{shireg}.

We explain why all $y \in {{}^{g_i(-1)}B}_{R/2} (p)$, are regular (in
particular, $p$ is regular). Choose $t$ close to $0$ and
$0<r \leq 1$ small, so that ${{}^tB}_{10^4r}(y) \subseteq
  {{}^{g_i(-1)}B}_{\frac{R}{2} }(p): $ for every $t <0$ such an $r$ must
  exist in view of the fact that the solution is smooth.
Then $|\Riem(\cdot,t)| \leq 10 $ on ${{}^tB_{10^4r}(y)   \subseteq
  {}^{g_i(-1)}B}_{\frac{R}{2} }(p)$ due to \eqref{curvperlinb}. Then
  ${{}^sB}_{r}(y)$ remains in ${{}^tB}_{10^4r}(y) \subseteq {{}^{g_i(-1)}B}_{\frac{R}{2} }(p)$ for
  all $s \in [t,0)$ due to \eqref{curvperlinb}  and the fact that the metric evolves according
  to the equation $\partt g = -\Rc(g)$, and $t$ is close to $0$.
Hence $\int_{{{}^sB}_r(y)} |\Riem(g)|^2(\cdot,s) d\mu_{g(s)} \leq
\ep_0$ for all $s \in [t,0)$, if $r$ is small enough, in view of
\eqref{curvperlinb} and the non-expanding estimate.

Although these estimates show us that $p$ is a regular point, they do {\bf not } tell us that 
$$\int_{^tB_{R/2} (p)} |\Riem|^2(\cdot,t) d\mu_{g(t)} \leq \ep_0$$
for all $t \in (-1,0)$: as $t$ gets closer to $-1$ from above, our
estimates on the curvature, \eqref{curvperlinb},  blow up.
However  by appropriately modifying the arguments in \cite{Li}   we can show that the
Riemannian metrics remain close in a $C^0$ sense to one another on some
fixed time independent region within these balls.
This fact is useful when it comes to describing $(X,d_X)$ ,  the limit as $t \upto T$ (before
scaling) of the solution $(M,g(t))_{t \in [0,T)}$, and how this limit
is obtained.

Examining the setup considered in the first part of the paper
\cite{Li} of Ye Li, we see that
we are almost in the same setup:
Scale back down to $t  = -\de$ (that is do the step (*) in the argument
above again), call the solution $g_i$ again,  and consider an arbitrary
$ y \in { {}^{g_i(-\de)} B}_{1}(p)$ as above.

 From the argument above
we have 
\begin{eqnarray}
&&|\Riem(g_i)(x,t)| \leq \frac{1}{\de + t}, \ \mbox{ for all } \  t \ \in \ (-\de,0) 
\mbox{ for all }  \  x \in { {}^{g_i(t)} B}_{\ti a }(y)
\end{eqnarray} 
for some constant $\ti a = \ti a (a)>0$, and $\de << \ti a < a$.

In order to see that we are almost in the same situation as Ye.Li, we
shift in time by $\de$: that is fix $i$ and define $g(t)= g_i(t+\de)$.
This means that the old time $0$ (where the flow possibly becomes
singular) is now time $\de$ and the old good time $-\de$ is now the good
time $0$. 
Then we have 
\begin{itemize}
\item[(i)] \begin{eqnarray}
|\Riem(g)(x,t)| \leq \frac{1}{t}, \ \mbox{ for all } \  t \ \in \ (0,\de) 
\mbox{ for all }  \ 
x \in {{}^{g(t)}B}_{\ti a  }(y), \label{perelbi2}
\end{eqnarray}
\end{itemize}
for all $y \in  { {}^{g_i(0)} B}_{1}(p),$ where $\ti a$ depends
only on $a$ which depends only on $\si_0,\si_1$, and we have chosen $\de$ so that $\de<<\ti a \leq a$. 
Without loss of generality, we may assume $\ti a = 2$ for this
argument.
If not, then scale so that it is: we still have $0<\de<<\ti a$ is still as small we we like
(but fixed).

This solution also satisfies
\begin{itemize}
\item[(ii)]  $\int_M |\Rc|^4(\dot,0)(g_0) \leq \hat \de_i $, with
  $\hat \de_i \to 0$ as $i \to \infty$ (by scaling we have changed the
  constants $\ti \de_i$ above by a fixed factor: $\hat \de_i =
  \frac{\ti \de_i}{(10\ti a)^2}$). 
\item[(iii)]  $(1/2) d\mu_{g(r)} \leq d\mu_{g(t)} \leq 2 d\mu_{g(s)}  $ for
all $ 0 \leq r\leq t \leq s < \de$ in view of the fact that we are dealing with a
solution satisfying the basic assumptions (see the inequalities \eqref{volin}),
\item[(iv)] we have a bound on the Sobolev constant $(\int_{B_r(z)}
f^4)^{1/2} \leq A \int_{B_r(z)}
|\grad f|^2 $ for all $ {{}^tB}_r(z) \subseteq {{}^tB}_2(y)$ for all $f :B_r(z)
\to \R$ which are smooth and have compact support in ${{}^tB}_r(z)$, for
all $ 0 \leq t < \de$: see \eqref{sobo} and \eqref{sobo2}. 
\item[(v)] 
\begin{eqnarray}
(\int_{^tB_2(y)} |\Riem|^3)^{\frac 1 3 } 
&& \leq \frac {1} {t^{\frac 1 3}} (\int_{M} |\Riem|^2)^{\frac 1 3 }  \cr
&& \leq \frac{1} {t^{\frac 1 3}} (K_0)^{1/3}
\end{eqnarray}
for all $ 0 \leq  t <\de$
in view of (i) and the bound $\int_M |\Riem|^2   \leq K_0:=c(g(0),M,T)$
from \eqref{riemint1}.
\end{itemize}
Examining Lemma 1, Lemma 2, Lemma 3 and Theorem 2 of \cite{Li}, we see that
this is exactly the setup of that paper, call $\mu:= (K_0)^{\frac 1
  3}$, except for the condition
$1/2 g_0 < g(t) <2g(s)$ for all $0 < t<s < \de$, which is also assumed
there. We are considering
the case that $u$ and $f$ of the paper by Ye Li are $u :=
|\Riem|$ and $f := |\Rc|$.
The extra assumption $1/2 g_0 < g(t) <2g(s)$ for all $0 <
t<s < \de$ is used in \cite{Li} to construct a time independent cut-off
function (in Lemma 3 of \cite{Li}, which is also used in Lemma 1
and Lemma 2 of \cite{Li}) for $0<r'<r$.  This cut-off function  $\phi:M \to
\R$  is smooth and satisfies $\phi|_{B_{r'}(y)} =1$, $\phi = 0$ on  $(B_r(y))^c$, $|\grad \phi |_{g_0} \leq
\frac{2} {r-r'}$ 
and $|\grad \phi |_{g(t)} \leq 2|\grad \phi |_{g_0}   \leq  \frac
{4} {r-r'}$.
We will only consider $1\geq r , r' \geq \frac{1}{4}$.
In our arguments, we will replace this function by a time dependent cut-off function $\phi(x,t)$ using the
method of Perelman. This new $\phi$ satisfies
\begin{eqnarray}
&& \partt \phi \leq \lap \phi +\frac{c}{(r-r')^2}  + \frac{c\phi }{t}\cr
&& |\grad \phi |^2_{g(t)} \leq \frac{c}{(r-r')^2} \cr
&& \phi|_{ {{}^t B}_{r'}(y)} =1, \cr
&& \phi|_{   ({{}^tB}_{r}(y))^c} =0,
\end{eqnarray}
for all $ t\leq S(c_1)$, wherever the function differentiable is,
where $S(c_1) >0$ and  $c =c(c_1)$, where $c_1$ is a constant
satisfying $|\Riem| \leq \frac{c_1}{t}$ on 
${{}^t B}_4(y)$ : in our case $c_1 =1$.
Using this new $\phi$ in the argument given in \cite{Li}, we obtain, after
making necessary modifications, the following:
\begin{eqnarray}
&& |\Rc(\cdot,t)| \leq \frac{\de^4}{t^{3/4}} \mbox{ on }
{{}^tB}_{3/4}(y),\label{Ricciint} \cr
&&  \mbox{ for all } t \in (0,\de),
\end{eqnarray}
as long as $(T-t_i) \leq \al(\si_0,\si_1,c(g(0),T),A)$ is small enough.
See Appendix \ref{Moserapp} for the details.
In particular,  translating and
scaling back to the solution we had before we performed the step (*),
we see that $|\Rc(y,t)| \leq \frac{\de}{t^{3/4}}$ for all $y  \in  {
  {}^{g_i(-1)} B}_{R/2}(p),$
for all $t \in (-1,0).$
Hence, integrating the  evolution equation $\partt g(y,s) = -2
\Rc(g)(y,s)$, we get

\begin{eqnarray} 
 g(y,s)e^{-8\de|s^{\frac 1 4} - r^{\frac 1 4}| } \leq g(y,r) \leq g(y,s)e^{ 8\de|s^{\frac 1 4} - r^{\frac 1 4}|} \label{inproofg}
\end{eqnarray}
for all $ r,s \in [-1,0)$, for all $y  \in  {
  {}^{g_i(-1)} B}_{R/2}(p),$ where $\de>0$ is small.
Translating in time and scaling back to the original solution, we
obtain \eqref{metricest}.
Before scaling back, note that it also implies
\begin{eqnarray} 
  \frac{1}{2}g(y,s) \leq g(y,r) \leq 2 g(y,s) \label{metricestproof}
\end{eqnarray}
for all $ r,s \in [-1,0)$, for all $y  \in  {
  {}^{g_i(-1)} B}_{R/2}(p).$
This condition is scale invariant, so translating and scaling back to the original
solution, we obtain \eqref{metricest2}. 

For later, notice, that \eqref{inproofg} implies that:  for all $\si>0$, there exists a $\ti \de>0$
such that,
\begin{eqnarray} 
 g(\cdot,s)(1-\si) \leq g(\cdot,r) \leq g(\cdot,s)(1+\si) \label{uniformmetric}
\end{eqnarray}
for all $ r,s \in (-1,0]$ with $|r-s|\leq \ti \de$ on $ {{}^{-1}
  B}_{\frac R 2}(p).$ 
Examining the argument above, we see that the results are correct for
{\it any} good time $t_i \in (0,T)$, as
long as $ (T-t_i)  \leq v(\si_0,\si_1,A,c(g(0),T) )$ is small enough.
This finishes the proof.\\
({\bf End of the proof of Theorem \ref{regularregion}}) $\Box$ .\\
{\bf proof of the Corollary \ref{uniformdcoro}}\\
Let $x,y$ $t,s$ be as in the statement of the corollary. Scale to the
situation as in the proof of Theorem \ref{regularregion}.  Let $\ga:[0,1] \to M$ be
a length minimising geodesic between $x$ and $y$ with respect to the
metric $g(-1)$. The curve doesn't leave ${{}^{-1} B}_{\frac{
    R}{10}}(p)$, and hence, using \eqref{metricestproof},
$d(x,y,s) \leq L_s(\ga) \leq 2L_{-1}(\ga) = 2d(x,y,-1)$.
Now let $\si:[0,1] \to M$ be a length minimising geodesic between $x$
and $y$ with respect to $g(s)$.
If $\si$ doesn't leave ${{}^{-1} B}_{\frac{
    R}{10}}(p)$, then $d(x,y,-1) \leq  L_{-1}(\si) \leq  2L_s(\si) =
2d(x,y,s)$, and hence $d(x,y,s) \geq \frac{1}{2} d(x,y,-1)$ in
this case. 
If $\si$ leaves ${{}^{-1} B}_{\frac{
    R}{10}}(p)$, then let $m$ be the first point at which it does so:
$\si(m) \in \boundary ({{}^{-1} B}_{\frac{
    R}{10}}(p))$, $\si(r) \in {{}^{-1} B}_{\frac{
    R}{10}}(p)$ for all $r <m$, and consider $\al = \si|_{[0,m]}$.
Then $d(x,y,s) = L_s(\si) \geq L_s(\al) \geq
\frac{1}{2}L_{-1}(\al)\geq \frac{1}{100 }  R = \frac{1}{2}\frac{2
   R}{100 }
\geq \frac{1}{2} d(x,y,-1)$.
Hence $d(x,y,s) \geq \frac{1}{2} d(x,y,-1)$ in
this case as well.\\
{\bf End of the proof of Corollary \ref{uniformdcoro}} $\Box.$
\section{Behaviour of the flow near singular points}\label{singularsection}
In this section we examine the behaviour of the flow near singular
points $p$.
We consider a sequence of good times $t_i \upto T$.
We will show that the singular set $\sing(M)$ can be covered by $L$ 
balls $({ {}^{t_i}B}_{R(t_i)}(^ip_j))_{j=1}^L$ ($L$ being independent
of $t_i$) of radius $R(t_i) = C\sqrt{T-t_i}$ ($C$ a large 
fixed constant, which is determined in the proof of Theorem
\ref{singularregion} below)  at time $t_i$, where $t_i$ are good times close
enough to $T$  , and that the balls ${{}^{t}B}_{R(t_i)}(^ip_j)$ with $t
\in (t_i,T)$ also cover $\sing(M)$.  We say nothing at this stage about
the topology of these regions, or how they geometrically look.
In the next sections we give more information on how singular
regions look like in the limit (as $t \upto T$).

The results of this section are used at the end of this section to
show that distance is uniformly continuous in the following sense: 
For all $\ep >0$ there
exists a $\de(\ep) >0$ such that 
$|d(x,y,t) - d(x,y,s)|  \leq \ep$
for all $x,y \in M$ for all $ t,s \in [0,T)$ with $|t-s| \leq \de$.
The singular set and the regular set were defined in the previous
section:
$\reg(M) := \{ p \in M \ | \ p$ is regular $ \}$ was defined in
Definition \ref{defreg} and  
$\reg_t(M) $ was defined in Definition \ref{defregII}.
$\sing(M) := \{ p \in M \ | \ p$ is { \bf not } regular $ \}$. 
The theorem that we prove in this section is
\begin{theo}\label{singularregion}
Let $(M,g(t))_{t\in[0,T)}$ be a solution to Ricci flow satisfying the
basic assumptions. Then there exist (large) constants $0<J_0,J_1,J_2<
\infty$, a (small) constant $0<w < \infty,$  and a constant $L \in \N$ such that for all
good times $s<T$ with $|s-T|\leq w$, there exist $p_1(s),
\ldots, p_L(s) \in M$ such that
\begin{eqnarray}
\sing(M)  &=& (\reg(M))^c \cr
 & \subseteq & (\reg_s(M))^c \cr
& \subseteq &  \cup_{j=1}^L{{}^t B}_{J_0 \sqrt{ T-s}}(p_j(s))
 \cr 
&\subseteq&
  \cup_{j=1}^L{{}^{s} B}_{J_1 \sqrt{ T-s} } (p_j(s))\cr
 &\subseteq&  \cup_{j=1}^L{{}^r B}_{J_2 \sqrt{ T-s}}(p_j(s)) \label{Gest}
\end{eqnarray}
for all $ s\leq t,r <T$.
\end{theo}
\begin{remark}
Notice that for fixed $s$, the sets ${{}^{s} B}_{J_1 \sqrt{ T-s} } (p_j(s))$ 
in the statement of the theorem don't depend on $t$ or $r$ ($s\leq t,r <T$), but
${{}^t B}_{J_0 \sqrt{ T-s}}(p_j(s))$ and ${{}^r B}_{J_2 \sqrt{ T-s}}(p_j(s))$ do.
\end{remark}
\begin{remark}
Using the estimates of the previous section and this
covering, we will obtain as a corollary, that the distance function is uniformly
continuous in time (see Theorem \ref{contdist}).
\end{remark}
\begin{proof}

Let $(M,h)$ be a Riemannian manifold with $\int_M |\Riem(h)|^2 \leq K_0 < \infty.$
Let $R>0$ be given fixed.
Assume there is some point $p_1$ with 
$\int_{ B_{R}(p_1)} |\Riem|^2 \geq \ep_0$. Then we look for a ball
$B_{R}(p_2)$ which is disjoint from $B_{R}(p_1)$ and satisfies $\int_{ B_R(p_2)} |\Riem|^2 \geq \ep_0$.
We continue in this way as long as it is possible to do so.  This leads
to a family of pairwise disjoint balls 
$ (B_{R}(p_j))_{j\in \{ 1, \ldots, L\}  } $ such that  $\int_{ B_{R}(p_j)} |\Riem|^2\geq 
\ep_0$
for all $j \in \{1, \ldots, L\}$. 
We define \begin{eqnarray}
  \BB_R &:=&  \BB_R(h) := \cup_{j=1}^L B_{2R}(p_j) \cr
\Omega_R &:=& \Omega_R(h):= M \backslash  \BB_R(h) = M \backslash \cup_{j=1}^L B_{2R}(p_j).
\end{eqnarray}
From the definition of $\Omega_R$ it follows that  $\int_{B_R(x)} |\Riem|^2 \leq \ep_0$ for all $x
\in \Omega_R$.

Using $\int_M |\Riem|^2(h) d\mu_h  \leq K_0$, we see
that we have an upper bound $L \leq \frac{K_0}{\ep_0}$ for the number
of balls constructed in this way.

Notice that for fixed $R$ this construction is not unique: by  choosing  the balls
in the construction differently we obtain a different
$\BB_R$, respectively $\Omega_R$.

If $(M,g(t))_{t \in I}$ is a solution to Ricci flow, $I$ an interval, 
then $ {}^t\Omega_R$ will denote
$\Omega_R(g(t))$ and $ {}^t\BB_R$  will denote  $ \BB_R(g(t))$ for any $t \in I$.
Take a sequence of good times $t_i \upto T$, and assume  we have scaled as in the proof Theorem \ref{regularregion} above,
to obtain a solution $(M,g(t))_{t \in (-A_i,0)}$.
Using the characterisation of the regular set given in Theorem \ref{regularregion}, and using
the $R$ appearing there, we see
that 
${{}^{-1}\Omega}_R \subseteq \reg(M)$ and hence
$\sing(M) =  M \backslash \reg(M) \subseteq M \backslash {}^{-1}\Omega_R$.

We wish to show that distance is not changing too rapidly near  and in
$ \BB_R(g(t=-1))$.

In order to explain this statement more precisely, and to explain the
argument which proves the statement, we assume for the moment that there is only one
ball ${}^{-1} B_{2R}({ {}^ip}_1) $ coming from the above construction of $
\BB_R(g(-1))$ and we call this ball
${}^{-1} B_{2R}(p) $.  Note that for each $i$, we may obtain a
different point $^ip_1$ depending on $i$. For the moment we drop the
$i$ and the $1$ from our notation and simply denote the point $^ip_1$ by $p$.

We define $G := {{}^{-1}B}_{2J}(p)$, for some large $J>>R$  fixed, and
$H :=  {{}^{-1}B}_{J}(p)$. It follows, that 
$H^c \subseteq   ({{}^{-1}B}_{2R}(p))^c  = (M \backslash
{{}^{-1}B}_{2R}(p) )  \subseteq \reg_{-1}(M) \subseteq \reg(M).$
Hence, using \eqref{metricestproof}, we have $\frac 1 8 g(x,t) \leq
g(x,-1) \leq 8 g(x,t)$ for all $x \in H^c \cap G$ for all $t \in [-1,0)$.

We may assume that $\dist(g(t=-1))(\boundary G,
\boundary H) = J$ 
(we are scaling a connected solution to Ricci flow with
diameter larger than $\frac 1 {d_0}>0$ (see 
\eqref{toppingresult}) by constants $c_i$ which go to infinity,  and hence the diameter of the
resulting solution is as large as we like at all times).

Note by construction $  G \cap H^c \neq \emptyset$ as the diameter of
the solutions we
are considering  is  as large as we like, as we just noted.
We have $J 8 \geq \dist(g(t))(\boundary G, \boundary H)
\geq \frac{J}{8}$
for all $t \in (-1,0)$ as we explain now. Any smooth regular curve $\ga:[0,1] \to M$  ($\ga'(s) \neq 0$ for all
$s \in [0,1]$) going from
$\boundary H$ to $\boundary G$ which lies completely in 
$\overline{H^c}\cap \bar G$ and satisfies $L_{g(t)}(\ga)\leq  \dist(g(t))(\boundary G,
\boundary H) + \ep $ must satisfy the following:
\begin{eqnarray}
L_{g(t)}(\ga) &=& \int_0^1 \sqrt{g(\ga(r),t)(\ga'(r),\ga'(r)) }dr\cr
&\geq& \frac 1 8 \int_0^1 \sqrt{g(\ga(r),t=-1)(\ga'(r),\ga'(r))} dr
\cr
&\geq& \frac 1 8 \dist(g(t=-1))(\boundary G, \boundary H) =
\frac{J}{8}
\end{eqnarray} 
in view of the definition of $G$ and $H$.
Hence 
\begin{eqnarray}
\dist(g(t))(\boundary G, \boundary H) \geq \frac{J}{8} \label{Jlowerbound}
\end{eqnarray} 
 for all $t\in[-1,0)$.
Notice that this means 
\begin{eqnarray}
^tB_{J/8}(p) \subseteq G ,\label{Gbound}
\end{eqnarray}
since $p \in H$ implies that 
$\dist(g(t))(p, \boundary G) \geq
\dist(g(t))(H, \boundary G) = \dist(g(t))(\boundary H, \boundary G)$
which is larger than or equal to $J/8$ in view of equation \eqref{Jlowerbound}.
Similarly, for $z    \in H^c \cap G $, let $\ga:[0,1] \to M$ be the radial
geodesic with respect to the metric at time $t = -1$ coming out of
$p$, starting at $z$ and stopping at $\boundary G$. We have $\ga([0,1])
\subseteq H^c \cap \bar G$ and hence 
\begin{eqnarray}
\dist(g(t))(z, \boundary G) &\leq& L_{g(t)}(\ga) \cr
&=& \int_0^1 \sqrt{g(\ga(r),t)(\ga'(r),\ga'(r)) }dr \cr
&\leq&  8  \int_0^1 \sqrt{g(\ga(r),t=-1)(\ga'(r),\ga'(r))} dr
\cr
&\leq& 8 L_{g(-1)}(\ga) \leq 8J
\end{eqnarray} 
That is,
\begin{eqnarray}
&&\dist(g(t))(z, \boundary G) \leq 8J  \mbox{ for all } z \in
 H^c \cap G, t \in [-1,0) \ \mbox{ and } \cr
&& \dist(g(t))(\boundary G, \boundary H) 
\leq  8J  \mbox{ for all } t \in [-1,0) \label{Jupperbound}
\end{eqnarray}

We wish to show that $\dist(p, \boundary G,t)$ is bounded by a
constant independent of time. 
\hfill\break\noindent {\bf Claim: $\dist(p, \boundary G,t)  \leq J^5$ }
\hfill\break\noindent 
Assume that there is some time $t \in (-1,0)$ with $\dist(p, \boundary
G,t) = N \geq  J^5$. Choose $q \in \boundary
G$ such that $d(p,q,t) = N$. 
This  part of the  argument was inspired by the argument
given in the proof of Claim 5.1 in the paper \cite{Topping}.
Take a distance minimising geodesic $\ga:[0,N] \to M$ from $p$ to $q$,
at time $t$, 
which is parameterised by arclength. 
Consider points
$$z_0 := \ga(0),z_1 := \ga(1), z_2 := \ga(2), \ldots, z_{N} :=
\ga(N) = q.$$
Without loss of generality $J \in \N$.
From the above, we see that the first $N- {16J}$ 
points $z_0, \ldots, z_{N-16J}$  must lie
in $H$, as we now explain. If not, then let $z_i =\ga(i)$ be the first
point with $i \leq N-16J$ such that $z_i \notin H$.
Then we could join the point $z_{i-1} = \ga(i-1)$ 
to $\boundary G$ by a geodesic whose length w.r.t to $g(t)$ is less than $8J+2$, 
in view of \eqref{Jupperbound}. This
would result in a path from $p$ to $\boundary G$ at time $t$ whose length
is  less than $N$ which is a contradiction. 
Also, using   equation \eqref{Jlowerbound}, we see that 
$   {{}^t B}_1(z_i) \subseteq G$ for
all $0 \leq i \leq  N-16J-1$ ($z_i \in H$ for such $i$, so to reach $\boundary G$ we
must first reach $\boundary H$ and then reach $\boundary G$: any such
path must have length larger than $J/8>>1$).

For $i \in \{1, \ldots , N-1\}$, the ball ${{}^tB}_{1}(z_i)$ is disjoint from all other balls
${{}^tB}_{1}(z_j)$, for all  $j \in  \{0, \ldots , N\}$ except for its two immediate neighbours
${{}^tB}_{1}(z_{i-1})$ and ${{}^tB}_{1}(z_{i+1})$, since $\ga$ is distance
minimising implies $\ga|_I$ is distance minimising for all intervals
$I \subseteq [0,N]$. Hence: for $i \neq 0$  we have
${{}^tB}_{1}(z_i) \cap {{}^tB}_{1}(z_j) = \emptyset$ as long as $j \neq i-1$ and
$j \neq i + 1$, where $i \in 1, \ldots, N-16J$.

Using the non collapsing estimate we see that 
\begin{eqnarray}
\vol(G,g(t))  &\geq&
\vol(\cup_{i=1}^{N-16J-1}  \ \ ({{}^tB_1}(z_i)) \ \ ) \cr
&\geq& \vol(\cup_{i=1}^{(N-16J-1)/2} \ \ ({{}^tB}_1(z_{2i})) \ \ )\cr
 &=& \sum_{i=1}^{ (N-16J-1)/2} \vol({}^tB_1(z_{2i})) \cr
&\geq& \sum_{i=1}^{(N-16J-1)/2} \si_0 = \si_0(N-16J-1)/2
\end{eqnarray}
On the other hand, $\vol(G,g(t)) \leq e^2 \vol(G,g(-1)) = e^2 \vol(
{{}^{-1}B}_{2J}(p)) \leq
e^2\si_1 64 J^4$ in view of the non-expanding estimate and the fact that
$G$ is defined independently of time (here we used the fact that
$\partt d\mu_{g(t)} \leq d\mu_{g(t)}$).
This leads to a contradiction since, $N = J^5 >16J + \frac{e^2 \si_1 128 J^4}{\si_0}
$  if $J$ is large enough and
$t_i$ is  close enough to time $T$ before scaling: we need $t_i$ close
to $T$ to guarantee
that the non-expanding and non-collapsing estimates hold for balls
(after scaling) of
radius $0 \leq r \leq 2J$.
\hfill\break\noindent
{\bf This finishes the proof of the claim}.
\hfill\break\noindent
Note, that this estimate and  \eqref{Gbound} imply that 
\begin{eqnarray} 
 {{}^t B}_{J/8}(p) \subseteq G = {{}^{t=-1} B}_{2J}(p) \subseteq {{}^r
     B}_{ \frac {J^5}{10^4} }(p) \label{Gestinitial}
\end{eqnarray}
for all $t,r \in [-1,0)$. Repeating the argument for $\frac{J^5}{10^4}$ instead of
$J$, we get
\begin{eqnarray} 
 {{}^t B}_{J/8}(p) \subseteq G && = {{}^{t=-1} B}_{2J}(p) \subseteq {{}^t
     B}_{\frac 1 {10^5}   J^5 }(p) \cr
&& \subseteq \ti G := {{}^{t=-1} B}_{2J^5}(p) \subseteq {{}^r
     B}_{  J^{25} }(p),
 \label{Gestinproof2}
\end{eqnarray}
for all $t,r \in [-1,0)$,
and $\sing(M) \subseteq G$.
This implies
\begin{eqnarray} 
 \sing(M) && = (\reg(M))^c  \subseteq (\reg_{-1}(M))^c\cr
&& \subseteq {{}^{t=-1} B}_{2J}(p)  \subseteq {{}^t B}_{\frac 1 {10^5}   J^5 }(p) \cr
&& \subseteq \ti G = {{}^{t=-1} B}_{2J^5}(p)   \subseteq {{}^rB}_{  J^{25} }(p)\label{Gestproofonepoint}
\end{eqnarray}
for all $t,r \in [-1,0)$.

The general case is as follows.
We wish to cluster those points ${{}^ip}_k$ (the centre points of the balls
appearing in the construction of $\BB_R(g(-1))$) together if they  satisfy
the condition:
$\dist(g(t=-1))({{}^ip}_k,{{}^ip}_l) $ remains bounded as $i \to
\infty$.
We assume that for each $t_i$, we obtained $L$ balls (independent of $i$) in the
construction of $\BB_R(g(-1))$ : if not, pass to a subsequence so that
this is the case.
Remember, that the solutions $(M,g(t)= g_i(t))_{t\in(-A_i,0)}$ are obtained by
translating (in time) and 
scaling the original solution $(M,g(t))_{t\in [0,T)}$ at good times
$t_i \upto T$ by
$g_i(\cdot,\ti t):= (T-t_i)g(\cdot,t_i + \frac{\ti t}{T-t_i})$ . So
the points in the construction of  $\BB_R(g(-1))$ could depend on $i$. 
We can guarantee, after taking a subsequence in $i$ if necessary,  that there
are  $\ti L \leq L$ sets, (clusters of points) ${^iT}_j$, $j \in
\{1, \ldots, \ti L\}$ and some large constant $\Lambda< \infty$ such
that: for all $i$ large enough, for all $k,l \in \{1,\ldots, L\}$, 
exactly one of the following two statements is true:
\begin{itemize}
\item$\dist((g(t=-1))( {{}^ip}_k,{{}^ip}_l) \leq \Lambda $  if $
  {{}^ip}_k,{{}^ip}_l \in {^iT}_s$ for some $s \in \{1, \ldots, \ti
  L\}$,  or
\item $\dist (g(t=-1))({{}^ip}_k,{{}^ip}_l) \to \infty$ as $i \to \infty$ if ${{}^ip}_k \in {^iT}_s$ and ${{}^ip_l}
  \in {^iT}_v$  and $s \neq v, s,v  \in \{1, \ldots, \ti L\}$.
\end{itemize}

We explain now how the sets  ${^iT}_s$ are constructed.
Fix $k,l \in \{1, \ldots, L\}$. If there 
is a subsequence in $i$ such that after taking this subsequence
$\dist((g(t=-1))(^ip_k,^ip_l) \to \infty$ as $i \to \infty$, then take
this subsequence. Do this for all $k,l \in \{1, \ldots, L\}$. As the
index set $\{1, \ldots, L\}$ is finite,  after taking finitely many
subsequences, we will arrive at the following situation: there exists a constant
$\Lambda < \infty$ such that for all $k,l \in \{1, \ldots, L\}$ one of
the following two statements is true:
\begin{itemize}
\item $\dist((g(t=-1))(^ip_k,^ip_l) \to \infty$  for all $i$ {\it or}
\item $\dist((g(t=-1))(^ip_k,^ip_l) \leq \Lambda$ for all $i$
\end{itemize}

Now we define $^iT_1$ as the set of all $^ip_k$, $ k \in
\{1, \ldots L\}$, such that 
$\dist((g(t=-1))(^ip_k,^ip_1) \leq \Lambda$ for all $i$.
$^iT_2$ is the set of all $^ip_k$, $ k \in
\{1, \ldots L\}$, such that 
$\dist((g(t=-1))(^ip_k,^ip_2) \leq \Lambda$ for all $i$.
And so on.
This gives us sets $^iT_1, \ldots, ^i T_{L}$. Each set contains finitely
many points, and for arbitrary $k,l \in \{1, \ldots,L\}$
 either $^iT_k \cap {{}^iT}_l = \emptyset$  for all $i\in \N$ or 
$^iT_k = {{}^iT}_l$ for all $i \in \N$. For fixed $i \in \N$:  if a set appears more than once, we
throw away all copies of the set except for one. This completes the construction of the sets $T_1, \ldots, T_{\ti L}$ (we drop the index $i$ again for the moment).

 Take one of these sets, for example $T_1$.
$\BBB_1$ will denote the union of the balls $B_{2R}(z)$ where
$z \in  T_1$.
Let $^ip_1  \in \BBB_1$ be arbitrary : we are rechoosing the points
$^ip_j$ (we choose exactly one point $^ip_j$, arbitrarily,  with
$^ip_j  \in \BBB_j$, and we do this for each $j \in \{1,
\ldots, \ti L\}$).
Define $G_1:= { ^{t=-1}B}_{2J}(^ip_1)$, $H_1:={{}^{t=-1}B}_{J}(^ip_1)$  where $J>>
\max(\Lambda,R)$ is large but fixed (independent of $i$).
Arguing as in the case of one point as above, we see that (for $i$ large enough) 
\begin{eqnarray} 
&&G_1 := {{}^{t=-1} B}_{2J}(p)  \subseteq {{}^t B}_{\frac 1 {10^5}   J^5 }(p) \cr
&& \subseteq \ti G_1 = {{}^{t=-1} B}_{2J^5}(p) \subseteq {{}^rB}_{  J^{25} }(p)
\label{Gestproofmorepoints}
\end{eqnarray}
for all $p \in \BBB_1$ (the choice
of $^ip_1 \in \BBB_1$ was
arbitrary), for all $t,r \in [-1,0)$. Note that we need $i$ large enough here, to guarantee that all other
sets $T_2, \ldots, T_{\ti L}$ do not interfere with the arguments
presented above: that is, we can guarantee that $ \overline{H_1^c
\cap G_1} \subseteq \reg_{-1}(M)$ and $ \overline{\ti H_1^c
\cap \ti G_1} \subseteq \reg_{-1}(M).$ 
Now do the same for the other sets $\BBB_j$, $j \in \{1, \ldots, \ti L\}$.

We call the constant $\ti L$ once again  $L$.
 Hence,  \begin{eqnarray} 
\sing(M) && \subseteq (\reg_{-1}(M))^c \cr
&&\subseteq \cup_{j=1}^L({{}^t B}_{J_0 }(^i p_j))
 \cr 
&& \subseteq \ti G =
  \cup_{j=1}^L({{}^{-1} B}_{J_1 } (^i p_j))\cr
&& \subseteq   \cup_{j=1}^L({{}^r B}_{J_2 }(^ip_j))
 \label{Gestproof}
\end{eqnarray}
for all $t,r \in [-1,0)$, where $J_0:= \frac 1 {10^5}   J^5 $, $J_1:=
2J^5$,$J_2:= J^{25}$.

Note, that by construction we have $d(-1)({{}^{-1} B}_{J_1 } (^i p_j),
 {{}^{-1} B}_{J_1 } (^i p_k)) \to \infty$ as
$i\to \infty$ for $j \neq k$.

Scaling and translating back to the original solution, we get
\begin{eqnarray} 
\sing(M) && \subseteq (\reg_t(M))^c \cr
&&\subseteq \cup_{j=1}^L({{}^t B}_{J_0 \sqrt{ T-t_i}}(^i p_j))
 \cr 
&& \subseteq G =
  \cup_{j=1}^L({{}^{t=t_i} B}_{J_1\sqrt{ T-t_i} } (^i p_j))\cr
&& \subseteq   \cup_{j=1}^L({{}^r B}_{J_2 \sqrt{ T-t_i}}(^ip_j)) \label{Gestproofsc}
\end{eqnarray}
for all $t,r\in [t_i,T)$.

The proof of the claim of the theorem is as follows.
Assume the conclusion of the theorem is false. Then for any constants $J_0, J_1,
J_2$, we can find  good times $t_i \in (T-w_i,T)$, where $w_i \to 0$, such that we cannot find
points $p_1(t_i), \ldots, p_L(t_i)$, with $L \leq \frac{K_0}{\ep_0}$ for which
\eqref{Gest} holds. Taking a subsequence, as above,  and choosing
$p_1(t_i) = {{}^ip}_1, \ldots, p_L(t_i) =   {{}^ip}_L$ leads to a
contradiction if $i$ is large enough. Note at first that it could be
that $L=L(s) \leq L  \frac{K_0}{\ep_0}$ depends on $s$. But by adding regular points
$p_{L(s) +1}(s), \ldots p_{\frac{K_0}{\ep_0}}(s)$ which are in
$\reg_s(M)$, and satisfy $\dist(s)(p_{i}(s),p_{j}(s)) \geq \si_0 >0,$
for all $i \geq L(s)+1$, for all $j \in \{1, \ldots,
\frac{K_0}{\ep_0} \}$, 
the conclusion of the theorem is still correct, and the comments which
follow this proof are still valid.

\end{proof}
\begin{remark}\label{distanceremark}
Note, that in the  construction above,  $d(-1)({{}^{-1} B}_{J_1 } (^i p_j),
 {{}^{-1} B}_{J_1 } (^i p_k)) \to \infty$ as
$i\to \infty$ for all $j \neq k$ (before scaling back).
Hence, any smooth curve $\ga:[0,1] \to M$ 
which lies in $(\cup_{j=1}^L({{}^{-1} B}_{J_1 } (^i p_j)))^c$ and has
$\ga(0) \in \boundary( {{}^{-1} B}_{J_1 } (^i p_j)) $ and $\ga(1) \in
\boundary( {{}^{-1} B}_{J_1 } (^i p_k) )$ must have $L_t(\ga) \geq
N(i)$ for all $t \in
[-1,0)$
with $N(i) \to \infty$ as $i \to \infty$ (in $(\cup_{j=1}^L({{}^{-1} B}_{J_1 } (^i p_j)))^c$ 
we have $\frac{1}{10}g(t) \leq g(-1) \leq 10g(t)$ for all $t \in
[-1,0)$.
Hence 
$d(t)( {{}^{-1} B}_{J_1 } (^i p_j),
 {{}^{-1} B}_{J_1 } (^i p_k) ) \geq N(i)$ for all $t \in
[-1,0)$ with $N(i) \to \infty$ as $i \to \infty$. Hence, without loss
of generality we can assume that the $p_j(s)$ in the statement of the Theorem satisfy
\begin{eqnarray}
d(t)( {{}^{s} B}_{J_1\sqrt{T-s} } (p_j(s)),
 {{}^{s} B}_{J_1 \sqrt{T-s}} (p_k(s))) \geq N(s)\sqrt{T-s} \label{distancepipj}
\end{eqnarray} for all $t \in
(s,T)$ for all $j \neq k$, where  $N(s) \to \infty$ as $s \upto T$. That is: the new claim is the
claim of the Theorem \ref{singularregion}, but with the extra claim  
\ref{distancepipj}. The proof is:  repeat the contradiction
argument at the end of the proof above for this new claim, using the information mentioned at
the beginning of this remark.
\end{remark}
\begin{remark}\label{ballsremark}
Note that in the conclusion of the theorem, we may also assume, that
\begin{eqnarray} 
&&{{}^{t} B}_{J^5 \sqrt{T-s} }(p_j(s)) \subseteq
 {{}^{s} B}_{16 J^5 \sqrt{T-s}}(p_j(s)) \cr
&&\subseteq {{}^rB}_{  J^{25}\sqrt{T-s} }(p_j(s)) \label{ballsestimate1}
\end{eqnarray}
for all $r,t \in [s,T)$, for all $j \in \{1, \ldots, L\}$ holds (not just
for the union of the balls). Repeating this part of the proof for larger $J$, but
keeping the same $p_j(s)$,  we see that 
in fact the following is also true:
\begin{eqnarray} 
&&{{}^t B}_{ K \sqrt{T-s} }(p_j(s)) \subseteq  {{}^{s} B}_{16 K \sqrt{T-s}}(p_j(s)) \subseteq {{}^rB}_{ K^{5} \sqrt{T-s}_2 }(p_j(s))
\label{ballsestimate2}
\end{eqnarray}
for all $r,t \in [s,T)$, for all $j \in \{1, \ldots, L\}$ for all $K
\geq J^5 \in \R^+$ as long as $|T-s|\leq w(K)$ is small enough, for
all good times $s$, in
view of Remark \ref{distanceremark} from above.
\end{remark}


As a corollary we obtain that the distance is uniformly continuous in
time. We explain this in the following.

Let $x,y \in M$ and $t \in [t_i,T)$ and $\ga:[0,1] \to M$ be a
distance minimising curve with respect to  $g(t)$ from $x$ to $y$,
$d_t(x,y) = L_t(\ga)$, $t_i$ a good time close to $T$. We use the notation $L_t(\si)= L_{g(t)}(\si)$ here, to
denote the length of a curve $\si$ with respect to $g(t)$.

We modify the curve $\ga$ to obtain  a new curve $\ti \ga:[0,1] \to M$ in the following
way: if $\ga$ reaches the closure of the  ball  ${{}^t B}_{J_0 \sqrt{ T-t_i} }(^i
p_k)$ (here, $ {{}^i p}_k = p_k(t_i)$, $k \in \{1, \ldots, L\}$ and $J_0,J_1,J_2$ are from the above construction) at a first point $\ga(r)$ then let
$\ga(\ti r)$ be the last point which is in the closure of the  ball  ${{}^t B}_{J_0 \sqrt{ T-t_i} }(^i
p_k)$ (it could go out and come in a number of times). Remove
$\ga|_{(r,\ti r)}$ from the curve $\gamma$. In doing this we obtain the finite
union of at most $L+4$ curves  $\ti \ga_j$. 
Call this finite union
$\ti \ga$ and consider it as a curve with finitely many discontinuities.

The new $\ti \ga$ has 
\begin{eqnarray}
L_t(\ti \ga) \leq L_t(\ga) =
d(x,y,t) \label{firstdist}
\end{eqnarray}

Now $(\cup_{k=1}^L {{}^{t} B}_{J_0 \sqrt{T-t_i} }(p^i_k))^c  \subseteq \reg_{t_i}(M)$ ($J_0$ coming
from \eqref{Gest} above), as we saw above, and   the Riemannian metric is
uniformly continuous (in time) on $\reg_{t_i}(M)$ for good times $t_i$.
 That is,
for all $\ep>0$ there exists a $\de(\ep,t_i)>0$ such that
\begin{eqnarray}
(1-\ep)g(y,t) \leq g(y,s) \leq (1+\ep)g(y,t) \label{inbetween}
\end{eqnarray} 
for all $ y \in (\cup_{k=1}^L {{}^{t} B}_{J_0 \sqrt{T-t_i}}(p^i_k))^c $ 
for all $ t_i \leq t,s  \leq T$, $|t-s| \leq \de$ in view of \eqref{metricreg} and the fact that  $y \in  (\cup_{k=1}^L {{}^{t} B}_{J_0
  \sqrt{T-t_i}}(p_k))^c \subseteq \reg_{t_i}(M)$.
Hence $L_t(\ti \ga) \geq L_s(\ti \ga) -c\ep$ for all $ T-\de
\leq t,s  \leq T$ in view of the fact that
the diameter of the manifold is bounded: more precisely,
$L_t(\ti \ga) \geq \frac{1}{1+\ep}L_s(\ti \ga) = (1- \frac{\ep}{1+ \ep}
)L_s(\ti \ga) \geq L_s(\ti \ga) - \ep L_s(\ti \ga)$,
and  $L_s(\ti \ga) \leq (1+ \ep)L_t(\ti \ga)  \leq (1+\ep)d(x,y,t)
\leq 2D $, in view of \eqref{firstdist} and
\eqref{inbetween}, and hence

\begin{eqnarray}
L_t(\ti \ga) \geq  L_s(\ti \ga) - 2D \ep  \label{seconddist}
\end{eqnarray}
 for all $ T-\de
\leq t,s  \leq T$ as claimed.
Putting \eqref{firstdist} and \eqref{seconddist} together we get 

\begin{eqnarray}
d(x,y,t)  && \geq L_t(\ti \ga) \cr
 && \geq  L_s(\ti \ga) -2D\ep \cr
&& \geq d(x,y,s) - 2L\ep -2D\ep.
\end{eqnarray}
The last inequality can be seen as follows: 
when $\ti \ga$ reaches a ball ${{}^t B}_{J_0 \sqrt{ T-t_i} }(^i
p_k)$, it must also be in ${{}^s B}_{J_2 \sqrt{ T-t_i} }(^i
p_k)$, by estimate 
\eqref{ballsestimate1}. So the two points of discontinuity on $\ti \ga$ may be joined smoothly
by a curve with length (with respect to $g(s)$) at most $2J_2 \sqrt{
  T-t_i}$, which is without loss of generality less than $\ep$. Doing this with all of the points of discontinuity (that is
with all the balls), we obtain a new continuous curve $\hat \ga$ from $x$ to $y$
with length $L_s(\hat \ga) \leq 2L\ep + L_s(\ti \ga)$, which implies
$L_s(\ti \ga) \geq L_s(\hat \ga) - 2L\ep \geq d(x,y,s) -2L\ep$ as claimed.
 
Swapping $s$ and $t$ in this argument gives us 
\begin{eqnarray}
|d(x,y,t) - d(x,y,s)|  \leq C\ep
\end{eqnarray}
for all $ T-\de  \leq t,s  \leq T$, where $x,y \in M$ are
arbitrary, and the constant $C$ appearing here does not depend on the
choice of $x,y \in M$. Smoothness of the flow (and bounded diameter of $M$) for $t <T$
implies that: 
\begin{theo}\label{contdist}
Let $(M,g(t))_{t \in [0,T)}$ be a smooth  solution on a compact
manifold satisfying the basic assumptions.
For all $\ep >0$ there
exists a $\de >0$ such that 
\begin{eqnarray}
|d(x,y,t) - d(x,y,s)|  \leq \ep
\end{eqnarray}
for all $x,y \in M$ for all $ t,s \in [0,T)$ with $|t-s| \leq \de$.
\end{theo}
\section{Convergence to a length space}

The results of the previous sections imply that $(M,d(g(t))) \to
(X,d_X)$ in the Gromov-Hausdorff sense
as $t \upto T$, where
$(X,d_X)$ is a metric space, and that 
away from at most finitely many points $x_1, \ldots,x_L \in X$  we have 
that $X\backslash\{x_1, \ldots x_L \} $ is a smooth Riemannian
manifold with a natural metric and that
the convergence is in the $C^{k}$ Cheeger-Gromov
sense. Furthermore,  $(X,d_X)$ is a length space (we explain all of
this below).

In the  paper \cite{BZ}, the authors also
showed independently, with the help of estimates
proved in their paper, a similar result to the result mentioned above
(see Corollary 1.11 of their paper).

The previous sections of this paper give us lots of information on how well the
limit $(X,d)$ will be  achieved and what the limit looks like, geometrically and topologically, near singular points.
We will use the results of the previous sections combined with a
method of G. Tian (in \cite{Tian}) to show somewhat more than the result mentioned
at the start of this section: namely, we will show  that  $(X,d_X)$ is a {\bf $C^0$
Riemannian  orbifold}, smooth away from its singular points (this is
shown in Section \ref{orbifoldsection}).
In the last section of this paper we explain
how it is possible to flow $C^0$
Riemannian  orbifolds of this type  using the orbifold Ricci flow and results from
the paper \cite{SimC0}.

We construct the limit space $(X,d_X)$ directly using the following
Lemma, which relies on the uniform continuity of the distance function
(in the sense of Theorem \ref{contdist}).

\begin{lemma}\label{Xdeflemm}
Let $(M,g(t))_{t\in [0,T)}$ be a solution to Ricci flow satisfying the
standard assumptions. Then
\begin{eqnarray}
X &:=&  \{ [x] \ | \ x \in M \} \mbox{ where } [x] = [y] \mbox{ if and
  only if }\cr 
&& \ \ \  d(x,y,t) \to 0  \mbox{ as } t\upto T. \label{Xdefi}
\end{eqnarray}
$X$ is well defined. Furthermore, the function
$d_X: X \times X \to \R^+_0$, 
\begin{eqnarray}
d_X([x],[y]) & := & \lim_{t \upto T} d(x,y,t)
\end{eqnarray}
is well defined and defines a metric on $X$.
\end{lemma}
\begin{proof}
If $d(x,y,t_i) \to 0$ for some sequence $t_i \upto T$, then
$d(x,y,s_i) \to 0$ for all sequences  $s_i \upto T$, in view of Theorem \ref{contdist}. This means that $[x]$ is well defined, and hence $X$ is
well defined.
Define $d_X([x],[y]) = \lim_{i \to \infty} d(x,y,t_i)$ where $t_i
\upto T$ is any sequence of  times approaching $T$.
The limit on the right hand side is well defined in view of the theorem on the uniform
continuity of distance (Theorem \ref{contdist}) and 
$d_X$ is then also well defined,
due to the theorem on the uniform continuity of distance (Theorem \ref{contdist}) and the
triangle inequality on $d(\cdot,\cdot,t)$.

From the definition, we see that  $ d_X([x],[y]) =0$ if and only if
$[x] =  [y]$.
The triangle inequality of, and symmetry of  $d_X$ follows from the triangle
inequality of, and symmetry of $d(\cdot,\cdot,t)$. 
\end{proof}
This $(X,d_X)$ is the limiting metric space of $(M,d(g(t)))_{t\in
  [0,T)}$ in view of the theorem on the uniform continuity of distance,
as we now show.

\begin{lemma}\label{fdeflemm}
Let everything be as in Lemma \ref{Xdeflemm} above.
The function $f:M \to X$ is defined by \begin{eqnarray}
f(x) &:=& [x].\label{fdef}
\end{eqnarray}
$f:(M,g(t)) \to (X,d_X)$ is a Gromov-Hausdorff approximation in the
sense that
\begin{eqnarray}
&&|d_X(f(x),f(y)) - d(g(t))(x,y)| \leq \ep(|T-t|) \cr
&&X:= f(M),
\end{eqnarray}
where $\ep(r) \to 0$ as $r  \downto 0$.
$f$ is continuous and surjective and hence $(X,d_X)$ is compact, precompact, connected and
complete.
In particular $(M,d(g(t))) \to (X,d_X)$ as $t\upto T$.
\end{lemma}
\begin{proof}
The first claim of the theorem follows immediately from the theorem on
the uniform continuity of distance and the definition of $X$.
Now we show that $f$ is continuous.
Let $U$ be open in $X$ and ${^{d_X}B}_\ep(p) \subseteq U$.
Due to the uniform continuity of the distance function, we know the
following: for all  $\ep>0$ there exists a $\de >0$ such that
$f({{}^{d(t)}B}_{\ep/2}(q)) \subseteq {{}^{d_X}B}_\ep(p)$ for all
$|T-t| < \de$  where $q$ is an arbitrary point with $f(q) = p$ (there could
  be lots of such points).
Hence ${{}^{d(t)}B}_{\ep/2}(q) \subseteq f^{-1}(U)$. Since $ p\in U $ was
  arbitrary, and $q$ with $f(q) = p$ was arbitrary, we have shown the following:
 for any  point $q \in f^{-1}(U)$ there exists an $\ep(q)$ and a
 $t_{\ep,q}< T$ such that
${{}^{d(t_{\ep,q})}B}_{\ep(q)}(q) \subseteq f^{-1}(U)$.
So we can write $$f^{-1}(U) = \cup_{q \in f^{-1}(U)}
{{}^{d(t_{\ep,q})}B}_{\ep(q)}(q) .$$
Each of the sets contained in the union is an open set in $(M,d(t))$
for {\bf any} $ t < T$ and hence, $f$ is continuous.
Hence $(X,d_X)$ is compact, being the continuous image
of a compact space, and hence complete and
precompact as it is a metric space.
\end{proof}
We have shown $f:M \to X$ is a continuous surjective map, where the topology on $X$
comes from $d_X$ and that on $M$ is the initial topology of the
manifold $M$, which agrees with that coming from the metric space
$(M,d(g(t)))$ for any $t<T$. 
Note that the map $f:M \to X$ is {\bf not } necessarily injective: it could be that
a  set  $\Omega$ containing more than two points is all mapped onto
one point in $X$  by $f$.

Let $^ip_1, \ldots, {{}^ip}_L$ be the points constructed in the previous
section (in the
possibly singular region) for large $i$.
Taking a subsequence we can assume that $f(^i p_k) \to x_k$ as $i \to
\infty$ for all $k \in\{1,\ldots,L\}$ for some fixed $x_1,
\cdots,x_L$.
We do not rule out the case  $x_j = x_k$ for  $j\neq k$. 
After  renumbering the $x_j's$ we have
finitely many (we use the symbol $L$ again)
distinct points $x_1, \ldots, x_L$, and $x_i \neq x_j$ for all $i \neq
j, i, j \in \{ 1, \ldots, L\}$. 

\begin{defi}
Let $[x] \in X$, $x \in M$. We say $[x]$  is a  {\it regular point in $X$}  if  $[x]$ contains
only one point, and $[x]$ is a {\it singular
  point in $X$}, if $[x]$ contains more than one point. 
\end{defi} 
\begin{remark}
Notice that the notion of {\it singular point} and {\it  regular
  point} differs depending on whether the point is in $X$ or $M$.
The following theorem gathers together properties that we have already
proved and shows that there is a connection between the different
notions of {\it singular} and {\it regular}.
\end{remark}
\begin{theo}\label{manifoldstructure}
Let $(M,g(t))_{t\in[0,T)}$ be a solution to Ricci flow satisfying the
basic assumptions, and $(X,d_X)$, 
$x_1,x_2,\ldots, x_L \in X$ as above.
Then \begin{itemize}
\item[(i)]  $X\backslash \{x_1, \ldots, x_L\} \subseteq f(\reg(M))
  \subseteq \reg(X)$ and $f(\sing(M)) \subseteq \{x_1, \ldots,
  x_L\}$. 
\item[(ii)]  $V:= f^{-1}( X\backslash \{x_1, \ldots, x_L\} ) \subseteq
  \reg(M)$, $V$  is open
  and $f|_V:V  \to  X$ is an {\it open,
  continuous, bijective} map,
and hence $f|_V:V \to  f(V) := X \backslash \{x_1, \ldots, x_L\}$ is a homeomorphism.
\end{itemize}

\end{theo}
\begin{proof}
(i) 
Take a point $x \notin \{ x_1, \ldots, x_L\}$. Then $d_X(x,x_j)  \geq \ep >0$ for
all $j \in \{1, \ldots, L \}$ for some $\ep >0$.
Let $[z] = x$. Remembering that $[ {{}^ip}_j] \to  x_j$ as $i \to \infty$, 
we see that $d_X([z],[ {{}^ip}_j])  \geq \frac{\ep}{2} >0$ for
all $j \in \{1, \ldots, L \}$  if $i $ is large enough. Fix $i$ large.
Then we can find a $\ti t(i)$ such that 
$d(z, {{}^ip}_j,t) \geq \ep/4$ for all $\ti t(i) \leq t <T $ near enough to
$T$, for some $\ti t(i) \geq t_i$, for
all $j \in \{1, \ldots, L \}$ in view of the definition of $d_X$. Scaling as in the
proof of Theorem \ref{singularregion} (and using the notation of the
proof), we see that $d(z, {{}^ip}_j, -\hat s_i ) \to \infty$ for all
$j \in \{1, \ldots, L \},$ for some $0\leq \hat s_i \leq 1$  as $i \to \infty$, and hence $z \in
\reg_{-1}(M) \subseteq \reg(M) $ due
to \eqref{Gestproof}, and hence $x= [z] = f(z) \subseteq f(\reg(M))$. 
This shows that $X\backslash \{x_1, \ldots, x_L\} \subseteq
f(\reg(M))$ and hence we have shown the first inclusion of (i).
Let $z \in \reg(M)$ be arbitrary. Then $z \in \reg_t(M)$ for $t$ close
enough to $T$ by definition. Choose a good time $t$ and scale the
solution by $\frac{1}{T-t}$ and
translate the the solution in time (as in the proof  of Theorem
\ref{regularregion} above). Then
$z \in \reg_{-1}(M)$.
Hence $d(z,y,t) \geq \frac{1}{10} d(z,y,-1)$ for all $y \in \overline{
  {{}^{-1}B}_{\frac{R}{200} }(z)}$ for all $t\in (-1,0)$
in view of \eqref{uniformd}, and $d(z,p,t) \geq
\inf_{y \in  \boundary ({{}^{-1}B}_{\frac{R}{200} }(z))  }  d(z,y,t) \geq \ep_0 >0$ for all $t\in (-1,0)$,
for all $ p\in ( {{}^{-1}B}_{R/200}(z)  )^c$ for the same reason.
That is $f(z)=[z]$ is not singular, since $\lim_{t\upto 0} d(z,y,t) >0$ for
all $y\neq z$, $y \in M$.  That is $f(\reg(M)) \subseteq \reg(X)$. 
This shows the second inclusion of (i). Now we prove the last
statement of (i). Let $p \in \sing(M)$. Assume $f(p)
\in X \backslash \{x_1, \ldots, x_L\}$. 
Then we know that there exists a $x\in \reg(M)$ such that $f(x) =f(p)$ in view of the
set inclusions just proved. But then $[x] = [p]$ and $x \neq p$ (since
$\reg(M)$ and $\sing(M)$ are disjoint). Furthermore $[x] \in \reg(X)$
due to the set inclusions just shown. This contradicts the definition
of $\reg(X)$. Hence, we must have $[p] = f(p) \in \{x_1, \ldots,
x_L\}$.
This finishes the proof of (i).

(ii) Let $z \in f^{-1}( X\backslash \{x_1, \ldots, x_L\} )$.
Then $f(z) \in  X\backslash \{x_1, \ldots, x_L\}.$ If $z \in
\sing(M)$ were the case,  then we would have $f(z) \in  \{x_1, \ldots,
x_L\}$ from (i), which is a contradiction.
Hence $ z \in \reg(M)$. That is $V:= f^{-1}(X\backslash \{x_1, \ldots,
x_L\}) \subseteq \reg(M).$ $V$ is open, since $f$ is
continuous, und $X\backslash \{x_1, \ldots, x_L\} $ is open.
From the above, $f|_V:V \to X$ is injective: assume there exists $x,y
\in V$ with $f(x) =[x]= [y] = f(y)$.
$x \in V$ implies $[x]
\in X \backslash \{x_1, \ldots, x_L\}$ and hence $[x] \in \reg(X)$
from (i). Combining this with $[x] = [y]$, we see that $x =y$ in view
of the definition of $\reg(X)$, that is $f|_V:V \to X\backslash \{x_1, \ldots, x_L\}$ is injective.
Let $(f|_V)^{-1}: X \backslash \{x_1, \ldots, x_L\} \to V$ be the inverse
of $f|_V:V \to X\backslash \{x_1, \ldots, x_L\}$.
Then $ (f|_V)^{-1}:N:= X\backslash \{x_1,\ldots,x_L\}  \to M$ is continuous as we now show. Assume $[z_k] \to
[z]$ in $f(V) = X \backslash \{x_1, \ldots, x_L\}$  as $k \to \infty$.  
Using the fact that
$f|_V:V \to X$ is injective, we see that there are unique points
$z_k,z \in V$  such that $f(z_k) = [z_k]$ and $f(z) = [z]$. 
Furthermore, $z_k, z \in \reg(M)$: if $z_k$ respectively $z$ were in $\sing(M)$,
then we would have $f(z_k)$ respectively $f(z) \in \{x_1, \ldots,
x_L\}$ which is a contradiction.

Assume $z_k$ does not converge to $z$. $z \in \reg(M)$ and hence we can
find a good time $t_i$ near $T$ such that $z \in
\reg_{t_i}(M)$. Fix this $t_i$.
$z_k $ doesn't converge to $z$ means:  we can find a an $\ep(i) =
\ep(t_i) >0$ and a subsequence $(z_{k,i})_{k\in \N}$ of $(z_k)_{k\in \N}$ (depending possibly on $i$),
such that $d(t_i)(z_{k,i},z) \geq \ep(i)$ for all $k \in \N$.
Scale at a time $t_i$ and translate as above to $t=-1$ (as in the
proof of Theorem \ref{regularregion}).
Then we have $ z \in \reg_{-1}(M)$ and
$d(-1)(z_{k,i},z) \geq \ti \ep(i)>0$ for all $k \in \N$.

Hence (arguing as above) $d(z,z_{k,i} ,s) \geq \frac{1}{10}
d(z,z_{k,i},-1) \geq \ep(i) >0 $ for all $z_{k,i} \in \overline{
  {{}^{-1}B}_{\frac{R}{200} }(z)}$ for all $s\in (-1,0)$
in view of \eqref{uniformd},\\ and $d(z,z_{k,i},s) \geq
\inf_{y \in  \boundary( {{}^{-1}B}_{\frac{R}{200} }(z) )}  d(z,y,s)
\geq \ep_0 >0$ for all $s \in (-1,0)$,
for all $ z_{i,k} \in ( {{}^{-1}B}_{\frac{R}{200}}(z)  )^c$ for the same reason.

Taking a limit
$s\upto 0$, we see $d_X([z_{k,i]},[z]) \geq \hat
\ep(i) >0$  for all $k \in \N$ , which contradicts the fact that $[z_k]
\to [z]$ as $k \to  \infty$.

\end{proof}
These facts allows us to give $X\backslash \{x_1,\ldots,x_L\}$ a natural manifold
structure, as we now explain.
\begin{prop}\label{rmanifoldstructure}
Let everything be as in Lemma \ref{manifoldstructure} above. $N = X\backslash
\{x_1,\ldots,x_L\}$ has a natural manifold structure and with this
structure $f|_V:V \to N$ is a diffeomorphism, $V:= f^{-1}(N)$.
There is a natural Riemannian metric $l$ on $N$ defined by
$l:= \lim_{t\upto T}f_*g(t)$.
\end{prop}
\begin{proof}
For $x \in N$,  let $\ti x \in V \subseteq M$ be the unique point in $V$
with $f(\ti x) =x$.  Let  
$\psi: \ti U \subsub V \subseteq M \to \R^4$ be
a smooth chart on $M$ with $\ti x \in \ti U$, and let $U:= f(\ti
U)$. $U$ is open from the above.
Define a coordinate chart $\phi:U \subseteq N \to \R^4$ 
by $\phi = \psi \of (f|_V)^{-1}$.
Clearly these maps define a $C^{\infty}$ atlas on $N$ (the topology
induced by  $f:V \to N$ on $N$ is the same as that induced by $d_X$ on $N$).
Using this atlas  on $N$, $f|_V:V \to N$ is then a smooth diffeomorphism
per definition.

Also, we can define a limit metric $l$ on $N$ in a natural way:
let $l:= \lim_{t\upto T} f_* (g(t))$. This metric is well
defined. Let $[z]  \in N$ and $z$ be the corresponding point in $V$.
$z \in \reg(M)$ because of (ii) above.
Hence $z \in \reg_{t}M$ for all good times $t$ near enough $T$ and hence,
after rescaling as in the proof of Theorem \ref{regularregion},
$z \in \reg_{-1}(M)$. Fix coordinates $\psi: \ti U \subsub V \to \hat U
\subseteq \R^4$ with 
$\ti U \subseteq  {{}^{-1} B}_{R/2}(z)$. Let
$g_{ij}(\cdot,t)$ refer to the metric $g(\cdot,t)$ with respect to the
coordinates $\psi$. Then $g_{ij}(t) \to l_{ij}$ as $ t\upto 0$ for
some smooth metric $l$ on $ \psi(\ti U)$, in view of the estimates
in the statement of Theorem \ref{regularregion} (see for example
the arguments in Section 8 of \cite{HaForm}). 
Noting that ${f}_{*}(g(t))_{ij}(\cdot,t) = g_{ij}(\cdot,t)$ in the
coordinates $\phi = \psi \of (f|_V)^{-1}: U \to \R^4$, we see that this
limit is well defined. 
\end{proof}

Notice that for each $x,y \in X$ we can find a $z$ with $d_X(x,z) =
d_X(z,y) = \frac{1}{2}d_X(x,y)$: this follows by using the
Gromov-Hausdorff approximation $f$, and the fact that this is true for
$d(g(t_i))$ (for a sequence of times $t_i \upto T$), and using the compactness of $M$ and $X$.
Hence, since $(X,d_X)$ is complete, we have that $(X,d_X)$ is also a
length space. We include the statement of this fact and others, some
of which appeared already in this section,  in the following theorem.

\begin{theo}\label{lengththm}
Let everything be as in Proposition \ref{rmanifoldstructure}, and let
$p_1, \ldots, p_L \in M$ be arbitrary points with $f(p_j) = x_j$ for all $j \in \{1,
\ldots, L\}$. 
 $(X,d_X)$ is a compact length space, with length function $L_X$, and $(N,l)= (X\backslash
\{x_1, \ldots,x_L\},l)$ is a smooth Riemannian manifold with
\begin{itemize}
\item[(a)]
$\sup_{x,y \in M} |d(g(t))(x,y)- d_X(f(x),f(y)| \to 0$ as $t \upto T$,
and hence
\begin{eqnarray}
\sup_{r \in [0,D]}d_{GH}({ {}^{d(g(t))}B}_r(x_i),{ {}^{d_X}B}_r(p_i))
\to 0 \mbox{ as } t \upto T,
\end{eqnarray} for arbitrary $p_j \in M$ with $f(p_j) =x_j$.
\item[(b)] 
Let $\hat N$ be a component of $N$ and $d_{\hat N,l}$ the metric
induced by $(\hat N,l)$ on $\hat N$. 
Then, for all $x \in N$, there exists an open set  $U \subsub N$ with $x \in U$,
such that
$d_{X}|_{U} = d_{\hat N,l}|_U$ and 
$\vol_l(E\cap U) = d\mu_X(E \cap U)$ for all measurable $E
\subseteq N$, where $d\mu_X$ refers to $n$-dimensional
Hausdorff-measure with respect to the metric space $(X,d_X)$, and
$\vol_l$ is the volume form coming from $l$ on $N$.
Hence, $d\mu_X|_N = \vol_l$ if we restrict to measurable sets in $N$.
\item[(c)]
 $L_X(\ga) = L_l(\ga)$,
in the case where $\ga$ is a piecewise smooth curve which 
lies completely in $N = X\backslash
\{x_1, \ldots,x_L\}$. 
\end{itemize}
\end{theo}
\begin{proof}
(a) follows directly from Lemmata \ref{fdef}, \ref{manifoldstructure} and \ref{rmanifoldstructure}.
As we mentioned above, for each $x,y \in X$ we can find a $z$ with $d_X(x,z) =
d_X(z,y) = \frac{1}{2}d_X(x,y)$: this follows by using the
Gromov-Hausdorff approximation $f$, and the fact that this is true for
$d(g(t_i))$, and using the compactness of $M$ and $X$.
Hence, since $(X,d_X)$ is complete, we have that $(X,d_X)$ is also a
length space, see Chapter 2 and in particular Theorem 2.4.16 of
\cite{BBI}: in the proof of Theorem 2.4.6 in \cite{BBI}, it is shown, that one
can construct a continuous curve $\ga:[0,l:= d_X(x,y)] \to X$ such
that $d_X(\ga(s),\ga(t)) = |t-s|$
for all $0<s,t \leq l$,  and hence
$d_X(x,y) = L_X(\ga)$  where, for $\si:[a,b] \to \R$ a continuous curve, $L_X(\si)$ 
is the supremum of the sums
$\Sigma(Y) =  \sum_{i=1}^N d_X(\si(y_{i-1}),\si(y_i))$ over all
finite partitions $Y = \{y_1, \ldots, y_N\}$, $N\in \N$ of $[a,b]$.
Hence all points $x,y$ can be joined by a continuous geodesic curve $\ga:[0,s]
\to X$ such that $L_X(\ga|_{[a,b]})=d_X(\ga(b),\ga(a))$ for all $ 0\leq
a,b \leq s$. We are using the notation of \cite{BBI}: a {\it geodesic}
in a  length space is a continuous curve whose length  realises the distance.

Let $q \in N = X\backslash
\{x_1, \ldots,x_L\}$.  Then $q \in \hat N$, the unique connected
component of $N$ containing $q$.
For the proof of (b), $d_l$ will refer to $d_{l,\hat N}$ the distance
function associated to $(\hat N,l)$.

From the above (Lemmata \ref{manifoldstructure}
and \ref{rmanifoldstructure}),  there exists a unique $\hat q \in M$
such that $f(\hat q) = q$, and 
we can find a neighbourhood $ Z \subsub U \subsub N$ and coordinates $\phi:U \to
\R^4$, with $x \in U$, $\ti U:= \phi(U)$, $\ti Z:= \phi(Z)$, $\phi(q)
= p$.
By choosing $\ep>0$ small enough, we can
guarantee
that ${{}^{ d_l}}B_{100 \ep}(q)$ and ${{}^{ d_X}}B_{100 \ep}(q)$ are compactly
contained in $Z$.
Using the fact that $g_{ij}(t) \to l_{ij}$ in the $C^k$ norm on $\ti
U = \phi(U)$,
we see that every smooth, regular curve $\ga:I \to \ti U$ with $\ga(0) \in
\phi({{}^{d_l }B}_{2\ep}(q) \cap {{}^{d_X}}B_{2\ep}(q)   )$
which leaves $\ti Z$ must have length larger than $10\ep$ with respect
to $g_{ij}(t)$ if $|t-T|\leq \de$ (and with respect to $l_{ij}$), in view of the
fact that $(1-\ti \ep)l_{ij} \leq g_{ij}(t) \leq (1+\ti \ep)l_{ij}$ in $\ti
U$ if $|t-T|\leq \de$, $\de$ small enough. Hence, for $x,y \in
{{}^{d_l}}B_{2\ep}(q) \cap {{}^{d_X}}B_{2\ep}(q) $, we have 
$d_l(x,y) = d_{\ti l,\ti U}(\phi(x),\phi(y))$, where
$\ti U = \phi(U)$, $\ti l = \phi_*(l) = (l_{ij})_{i,j\in \{1,\ldots,
  n\}}$ and $d_{\ti l,\ti U}$ is the
distance on the Riemannian manifold $(\ti U, \ti l)$. 
Similarly, $d_{g(t) }(f^{-1}(x),f^{-1}(y)) =  d_{\ti g(t),\ti
  U}(\phi(x),\phi(y))$, $\ti g(t) = \psi_*(g(t)) =
(g_{ij}(t))_{i,j\in \{1,\ldots, n\}}$  if $|T-t|\leq \ti \ep$,
where we are using the coordinates $\phi = \psi \of f^{-1}$,
introduced in Proposition \ref{rmanifoldstructure} [Explanation. Without loss of
generality, $|d_{g(t)}(f^{-1}(x),f^{-1}(y))-d_X(x,y)|\leq \ep$ for $|T-t|\leq \ti
\ep$, and hence $d_{g(t)}(f^{-1}(x),f^{-1}(y)) \leq 3\ep$. If $\ga$ is
any curve in $M$ between $f^{-1}(x)$ and $f^{-1}(y)$ whose length is
less than $4\ep$, then $\ga$ must lie in $f^{-1}(Z)$: otherwise,
pushing down to $\ti U$ with the coordinates $\psi$, we would obtain a
part of the curve having length larger than $10\ep$, which is a
contradiction. End of the explanation].
This shows us  \\
$d_X(x,y) = \lim_{t \upto T} d_{g(t)}(f^{-1}(x),f^{-1}(y))  = \lim_{t \upto T} d_{\ti g(t),\ti
  U}(\phi(x),\phi(y))$ \\
$=  d_{\ti l,\ti
  U}(\phi(y),\phi(y)) = d_l(x,y)$, as claimed.
Furthermore, since $l$ is smooth, we can assume that $\ep>0$ is so
small, that $\vol_l|_{ {{}^{d_l} B}_{\ep}(q) } = \curlH_{d_l}^n|_{
  {{}^{d_l}B}_{\ep}(q) },$ and hence
$\vol_l|_{ {{}^lB}_{\ep}(q) } = \curlH_{d_X}^n|_{
  {{}^{d_l}B}_{\ep}(q) }$, since $d_X = d_l$ on $  {{}^{d_l}B}_{\ep}(q) $,
where here $\curlH_{d_l}^n$ is Hausdorff-measure on $(\hat N, d_l)$.
This finishes the proof of (b).\\
It follows, that  $L_X(\si) = L_l(\si)$ for any piecewise
smooth $\si:[0,1] \to X \backslash \{x_1, \ldots, x_L\}$ curve: we cover the
image by small  balls for which on each of the balls $d_X = d_l$, and
use the fact that locally, $L_l(\si)$
is the supremum of the sums
$\Sigma(Y) =  \sum_{i=1}^Nd_l(\si(y_{i-1}),\si(y_i)) = \sum_{i=1}^Nd_X(\si(y_{i-1}),\si(y_i)) $ over all
finite partitions $Y$,$Y = \{y_1, \ldots, y_N\}$,$N \in \N$ of $[a,b]$ (without loss of generality,
$\si[y_i,y_{i+1}]$ lies in a small ball on which $d_X = d_l$).
This is (c).
\end{proof}

\section{Curvature estimates on and near the limit space}\label{curvestsec}

Let $d\mu_X$ denote Hausdorff-measure on the metric space
$(X,d_X)$. This is an outer measure and defined for all sets in
$X$. See for example  Chapter 2 of \cite{AT}. Let $d\mu_l =\vol_l$ refer to the  measure on $N= X \backslash
\{x_1, \ldots, x_L\}$ coming from the Riemannian metric $l$. From (b)
in Theorem \ref{lengththm}
above, we saw that $d\mu_l = (d\mu_X)|_N$ when we restrict to
measurable sets in $N$.
Hence for any measurable set $E $ in $N$, we have
\begin{itemize}
\item[(i)] $d\mu_l(E) = d\mu_X(E) = \lim_{\ep \downto 0}
  d\mu_X(E\backslash {{}^{d_X}B}_{\ep}(p)) = \lim_{\ep \downto 0}
  d\mu_l(E\backslash {{}^{d_X}B}_{\ep}(p))$
\item[(ii)]
 By construction $l$ is the
limit of the pull back of the metrics $g(t)$ by $f^{-1}$, and hence,
$c_0r^4 \leq d\mu_l({{}^{d_X}B}_{r/2,r}(x_i)) \leq c_1r^4$ for all $r \leq
diam(X)$, where $c_0,c_1$ are fixed constants. This can be seen as
follows. Let $U:= {{}^{d_X}B}_{r/2,r}(x_j)$. Then
$  {{}^{t}B}_{r/4}(p_j) \subseteq  \hat U := f^{-1}(U) \subseteq {{}^{t}B}_{2r}(p_j)$
for all $t$ with $|T-t| \leq \de$ small enough, in view of the
definition of $f$, and the uniform continuity in time of the distance function
(here $p_j$ is an arbitrary point with $f(p_j) = x_j$), and hence
$c_0r^4 \leq \vol_{g(t)}(\hat U) \leq c_1r^4$. Letting $t\upto T$ and
using $\vol_{g(t)}(\hat U) = \vol_{f_*(g(t))}(U) \to
\vol_l(U)=d\mu_l(U)$ implies the claimed estimate.
\item[(iii)] Hence the non-collapsing/non-expanding estimates $\ti \si_0 r^4 \leq
d\mu_l({^{d_X}B}_r(z)) \leq \ti \si_1r^4$ must also hold on $X$ for
some constants $0<\ti \si_0 ,\ti \si_1< \infty$, for all $r \leq
\diam(X)$ .We denote the constants $0<\ti \si_0 ,\ti \si_1< \infty$
once again by $0<\si_0 ,\si_1< \infty$.
That is, the non-collapsing / non-expanding estimates
survive into the limit.
\end{itemize}

 In view of the results of the previous sections we have
\begin{theo}\label{flatnessthm}
Let everything be as in the previous section ($X$, $x_1, \ldots, x_L$
are  defined in Lemma
\ref{manifoldstructure} and $l$ is defined in Lemma \ref{rmanifoldstructure}).
Then, \\
(i)
\begin{eqnarray}
\int_{X}  |\Riem(l)(x)|^2 d\mu_{X} \leq\label{standardint}
K_0:=c_0(g_0,T)
\end{eqnarray}
where $c_0(g_0,T)$ is the constant appearing in \eqref{riemint1}, and we define $|\Riem(l)(x)|= 0$ for $x \in \{ x_1,
\ldots,x_L\}$ (this is  a measurable function, since $d\mu_X(S) =
0$ for any finite set $S \subseteq X$).\hfill\break
(ii) The following {\it flatness estimates} are also true.\hfill\break
Let $(a_i)_{i\in\N}$ be any sequence with $a_i \upto \infty$, and let
$l_i =a^2_i l$, $d_i =\sqrt{a_i} d_X$.
Then for all $0<\si< N < \infty$, $K \in N$, we have
\begin{eqnarray}
|\grad^k\Riem(l_i)(x)| && \leq \ep(i,\si,N,K)  \label{flatness}
\mbox{ on } 
{{}^{d_i}B}_{\si,  N}(x_j)
\end{eqnarray}
where $\ep(i,\si,N,K) \to 0$ as $i \to \infty$ for fixed $N,\si,K$,
and $j\in \{1,\ldots, L\}$.

\end{theo}

\begin{remark}
Note that we obtain the result \eqref{flatness} {\it for all}
sequences. It is not necessary to pass to a subsequence in order to
obtain the result.
\end{remark}
\begin{remark}
Compare the estimates
with those stated in Corollary 1.11 in \cite{BZ}, which were
obtained independently.
\end{remark}
\begin{proof}
(i)\hfill\break
Using the theorem on monotone convergence  (see for example
Theorem 2 Section 1.3 in \cite{EG}) and the fact that $d\mu_X(\cup_{i=1}^L B_{\ep}(x_i))
\to 0$ as $\ep \downto 0$, we see that 
\begin{eqnarray}
&& \int_{X}  |\Riem(l)(x)|^2 d\mu_{X}(x) \cr 
&& \ \  = \lim_{\ep \downto 0} \int_{X
  \backslash (\cup_{i=1}^L B_{\ep}(x_i))} |\Riem(l)(x)|^2 d\mu_{X} \cr
&& \ \  = \lim_{\ep \downto 0} \lim_{t\upto T} \int_{f^{-1}(X
  \backslash (\cup_{i=1}^L B_{\ep}(x_i)))} | \Riem(g(t))(x)|^2
  d\mu_t \leq K_0
\end{eqnarray}
This finishes the proof of (i).\hfill\break
(ii)\hfill\break
Let $c_i :=  \frac{1}{T-t_i}$ where $t_i$ is a sequence of good times.
Scale and translate in time $(M,g(t))_{t\in
  [0,T) }$ as in Theorem \ref{regularregion}, we call the resulting
solution also $(M,g(t))_{t \in (-A_i,0)}$, and scale $d_X$ by $d_i = \sqrt{c_i}d_X$.
Notice that $d_i(x_k,x_l) \to \infty$ as $i \to \infty$ and we will
only be concerned with these blow ups near one point $x_k$: without
loss of generality $x_k = x_1$. Assume $x_1 \in f(\sing(M)),$ and let
$p_1 \in \sing(M)$ be a point with $f(p_1) = x_1$.
If $x_1 \in f(\reg(M))$, then the theorem follows by blowing up the region
around $x_1$, which has a Riemannian manifold structure.
From the estimates of Theorem \ref{singularregion}, we have $\sing(M) \subseteq (\reg_{-1}(M))^c \subseteq \cup_{k=1}^L
(   { {}^{-1}B}_{J_1} ({{}^ip}_k)    ) $ and hence $p_1 \in   {
  {}^{-1}B}_{J_1} ({{}^ip}_k)   $ for some $k\in \{1,\ldots, L\}$:
renaming the $({{}^ip}_k)'s$ we can assume $p_1 \in   {
  {}^{-1}B}_{J_1} ({{}^ip}_1)$ and hence 
\begin{eqnarray}
|\grad^j \Riem (g_i(\ti t))| \leq C_j \mbox{ on } (\cup_{k=1}^L
({ {}^{-1}B}_{2J_1} ({{}^ip}_k))\ \ )^c \mbox{  if } \ti t \in (-\frac 1
2,0), \end{eqnarray}
in view of the estimates \eqref{firstcurv}, where ${{}^ip}_1=p_1$.
From Remark \ref{ballsremark}, we see that, without loss of generality,
${{}^t B}_{2J_1}(p_1)  \subseteq  {{}^{-1} B}_{32 J_1}(p_1)  \subseteq {{}^t
  B}_{2^5J_1^5}(p_1) $ for all $t \in (-1,0]$  and
${{}^t B}_{N}(p_1)  \subseteq  {{}^{-1} B}_{16N}(p_1)  \subseteq {{}^t
  B}_{N^5}(p_1) $ for all $t \in (-1,0]$ 
if $i$ is large enough.
Hence, using the fact that $d(-1)({{}^ip}_j,{{}^ip}_k) \to \infty$ as $i \to
\infty$ (see Remark \ref{distanceremark}) for all $j \neq k$, we see
that $ {{}^t B}_{N}(p_1) \cap    ({{}^t B}_{2^5J_1^5}(p_1))^c \subseteq
{{}^{-1} B}_{16N}(p_1) \cap  ({{}^{-1} B}_{32J_1}(p_1))^c $ and
\begin{eqnarray}
&& |\grad^j \Riem (g_i(\ti t))| \leq C_j \cr 
&& \ \ \ \ \ \mbox{ on }  {{}^{-1}
  B}_{16N}(p_1) \cap  ({{}^{-1} B}_{32J}(p_1))^c \supseteq 
  {{}^t B}_{N}(p_1) \cap    ({{}^t B}_{2^5J_1^5}(p_1))^c  \cr
&& \ \ \ \  \ \mbox{  if } \ti t \in (-\frac 1
2,0),  \label{gradientriem}\end{eqnarray}
and hence,
taking a  limit $t \upto T$, we see that 
\begin{eqnarray}
&& |\grad^j \Riem (l_i)| \leq C_j \mbox{ on }    {{}^{d_i} B}_{N}(x_1) \cap    ({{}^{d_i}
  B}_{J_4}(x_1))^c,
\end{eqnarray}
where $l_i =c_il$, $J_4 := 2^5J_1^5$.
Using that 
$\int_{   { {}^{d_X}B}_r(x_1)   }    |\Riem(l)(x)|^2 d\mu_X(x)  \to 0$ 
 as $r \to 0$, we see that
\begin{eqnarray}
\int_{ {  {}^{d_i}B}_{J_4,N  }(x_1)}    |\Riem(l_i)(x)|^2 d\mu(i)_X(x)
\to 0 \mbox{ as } i \to \infty, \label{flatty}
\end{eqnarray} 
where $d\mu(i)_X$ is Hausdorff-measure on $(X,d_i)$,
and hence 
\begin{eqnarray*}
|\Riem(l_i)(x)| && \leq \ep(i) \to 0   \mbox{ as } i \to \infty
\mbox{ on } 
{{}^{d_i} B}_{J_4+1, N-1}(x_1)
\end{eqnarray*} 
in view of the fact that $|\grad^j \Riem (l_i)| \leq C_j$for all $j \leq K$ on the
same set ($C_j$ not depending on $i$).
In fact we may assume smallness for all gradients up to a fixed order. This
can be seen as follows. Introduce geodesic coordinates at a point $m_i
\in {{}^{d_i} B}_{J^5+1, N-1}(x_1)$. The injectivity  radius at
$m_i$  is larger that $\be>0$ for
all metrics independent of $i$ in view of the injectivity radius
estimate of Cheeger-Gromov-Taylor, Theorem 4.3 in \cite{CGT}, and the
non-collapsing/non-inflating estimates. Now using Theorem 4.11 of
\cite{HaComp}, and writing $l_i$ in these geodesic coordinates, we get
$|D^kl_i|_{B_{\be}(0} \leq C(K)$ for all $k \in \{1, \ldots, K\}$.
Hence taking a subsequence, we get a limit metric in
$C^{k-1}(B_{\be}(0))$, which is equal to $\de$, by Theorem 4.10 of \cite{HaComp}.

That is, without loss of generality, 
\begin{eqnarray}
|\grad^k\Riem(l_i)(x)| && \leq \ep(i) \to 0 \mbox{ as } i \to \infty
\mbox{ on } 
{{}^{d_i}B}_{J_4+1, N-1}(x_1) \label{flatnessinit}
\end{eqnarray} 
for all $k  \leq K \in \N_0$, where $K$ is fixed but  as large as we
like, for $l_i = c_i l$,  $c_i = \frac{1}{(T-t_i)}$, where $t_i \upto T$  is a sequence of good
times, where we took various subsequences to achieve this.
In fact the equation \eqref{flatnessinit} is true {\it for any sequence} $c_i
\upto \infty$: it is not necessary to take a subsequence, and it is
not necessary that $c_i$ has the form $c_i = \frac{1}{(T-t_i)}$, where
$t_i$ are good times. We explain this now.
First, the statement is true for any  sequence of the form $c_i =
\frac{1}{(T-t_i)}$: if not, then take a sequence for which it
fails. Taking a subsequence, if necessary,  in the proof above, we arrive  
at a contradiction.

Now let $c_i \to \infty$ be arbitrary.
We can always write $c_i = \frac {\al_i}{(T-t_i)}$ for some sequence of
good times $t_i \upto T$ and $\al_i \in (1/4,4)$, in view of Lemma
\ref{goodtimelemma}. 
Now \eqref{flatnessinit} holds for the metrics $\ti l_i = \frac{1}{(T-t_i)}
l$, as we have just shown,
and hence, for  $l_i = \frac{\al_i}{(T-t_i)}l = \al_i \ti l_i$, we get
\begin{eqnarray}
|\grad^k\Riem(l_i)(x)| && \leq \ep(i) \to 0 \mbox{ as } i \to \infty
\mbox{ on } 
{{}^{d_i}B}_{2 (J_4+1), \frac{1}{2}(N-1)}(x_1)\cr \label{flatnessinitmiddle}
\end{eqnarray} 
Now let $a_i$ be an arbitrary sequence going to infinity, and $\ti l_i
= a_i l$.
Writing $l_i =c_il$ with $c_i= \frac{4 (J_4+1)}{\si^2} a_i $, we see that 
$\ti l_i = \frac{\si^2}{4(J_4+1)} l_i$, and hence, using the fact that $N$ was arbitrary
(but large), we get, 
\begin{eqnarray}
|\grad^k\Riem(l_i)(x)| && \leq \ep(i) \to 0 \mbox{ as } i \to \infty
\mbox{ on } 
{{}^{d_i}B}_{\si, N}(x_1)\cr \label{flatnessinit2}
\end{eqnarray} 
\end{proof}

So we see that the manifold is becoming very flat away from singular
points, in the sense just described, after scaling.
Using these flatness estimates  
we will show that $X$ is a {\it generalised  $C^0$ Riemannian orbifold}.
We wish also to show that at each possible
orbifold point there is only one component: that is, that $X$
is actually a {\it $C^0$ Riemannian orbifold} with only finitely many orbifold
points. To do this, it will be necessary to obtain approximations of
the  blow ups $({{}^{d_i}B}_{\si,N}(x_1))$ (constructed in the proof
above) by Riemannian manifolds which have certain nice properties. 

This is the content of the next theorem.
\begin{theo}{\rm (Approximation Theorem)} \label{approximationthm}\\
Let $l$ and $X$, $x_1, \ldots, x_L$,  be as in Lemma
\ref{manifoldstructure} and Lemma \ref{rmanifoldstructure}. 
There exist smooth metrics $g_i$ on $M$, and points $p_j \in M$  for $j
\in \{1, \ldots, L\}$ such that
\begin{eqnarray}
&& d_{GH}({}^{g_i}B_{N(i)}(p_j),{}^{d_i}B_{N(i)}(x_j))
\leq \al_i  \cr
&& |\grad^k \Riem(l_i)|^2 \leq \al_i  \mbox{ on }
{{}^{d_i}B}_{\si(i),N(i)}, \mbox{ and }
\cr
&& ( {{}^{d_i}B}_{\si(i),N(i)}(x_j),l_i), \mbox{ is } \al_i  \mbox{ close to }
({{}^{g_i} B}_{\si(i),N(i)}(p_j) ,g_i)  \\
&& \mbox { in the } C^k \mbox{ sense, and  } \cr
&& \int_M |\Ricci(g_i)|^4d\mu_{g_i} \to 0 , \mbox{ as } i\to \infty,
\end{eqnarray}
where $d_i (\cdot,\cdot) = a_id_X(\cdot,\cdot)$, $l_i =a_i^2l$ and
$a_i,\si(i),a_i,N(i) \in \R^+$ are numbers satisfying
$0<\al_i, \si(i) \to 0$ as $i \to \infty$, $a_i,N(i) \upto \infty$ as
$i \to \infty$.

The condition {\rm $\ep$ close in
  the $C^k$ sense}, is made precise in the proof of the theorem, 
and the approximations are always achieved with $f$.
\end{theo}

\begin{proof}
Let $x_j$ be fixed. If $x_j \in f(\reg(M))$ then the theorem follows
directly using the definition of $C^k$ close and Theorem
\ref{regularregion} (see below). So assume $x_j  = x_1 \notin f(\reg(M))$ and
let $f(p_1) =x_1$. 

Let $t_i$ be a sequence of good times and scale by $a_i:= \frac{1}{T-t_i}$
and translate as in the proof of (ii)
Theorem \ref{flatnessthm}.

First we use a similar argument to that given at the end of Section \ref{singularsection}
to show that $|d_t(\cdot,\cdot)- d_s(\cdot,\cdot)| \leq C(J_1)$ for all
$t,s \in (-\de(N,J_1),0]$ for all $x,y \in {{}^t B}_{N/4}(p_1)$.
We use the notation from the proof of (ii)
Theorem \ref{flatnessthm} in this argument, and we take various subsequences when necessary.

Let $x,y \in {{}^t B}_{N/4}(p_1)$ be arbitrary in, and $\ga$ a
distance minimising curve between these two points w.r.t to $g(t)$ ($\ga$ must lie in
${{}^t B}_{N}(p_1)$ and we have $L_t(\ga) \leq N$).

We modify the curve $\ga$ to obtain  a new curve $\ti \ga:[0,1] \to M$ in the following
way: if $\ga$ reaches the closure of the  ball  ${{}^{t} B}_{2J_1 }(
p_1)$ at a first point $\ga(r)$ then let
$\ga(\ti r)$ be the last point which is in the closure of the  ball  ${{}^{t} B}_{2J_1 } (
p_1)$ (it could go out and come in a number of times). Remove
$\ga|_{(r,\ti r)}$ from the curve $\gamma$. In doing this we obtain the finite
union of at most $2$ curves  $\ti \ga_1$ and $\ti \ga_2$. 
Call this finite union
$\ti \ga$ and consider it as a curve with finitely many discontinuities.

The new $\ti \ga$ has 
\begin{eqnarray}
L_t(\ti \ga) \leq L_t(\ga) =
d(x,y,t) \leq N \label{firstdistII}
\end{eqnarray}

From equation \eqref{gradientriem} in the proof above, we see that 
for all $\ep>0$ there exists a $\de(\ep)>0$ such that
\begin{eqnarray}
(1-\ep)g(y,t) \leq g(y,s) \leq (1+\ep)g(y,t) \label{inbetweenII}
\end{eqnarray} 
for all $ y \in  {{}^t B}_{J_4,N}(p_1)  =  {{}^t B}_{2^5J_1^5,N}(p_1)  \subseteq 
{{}^{-1}  B}_{32J,   16N}(p_1) $
for all $ t,s  \in (-\de,0]$ ($\de$ independent of $i$:
use \eqref{gradientriem} and the evolution equation
$\partt g= -2\Ricci(g)$).
Hence $L_t(\ti \ga) \geq L_s(\ti \ga) -\ep L_t(\ti \ga) \geq L_s(\ti \ga)
-\ep N$ for all $ t,s  \in (-\de,0],$ which, when combined with
\eqref{firstdistII}, gives us 
\begin{eqnarray}
d(x,y,t)  && \geq L_t(\ti \ga) \cr
 && \geq  L_s(\ti \ga) -\ep N \cr
&& \geq d(x,y,s) -\ep N -J_4^{6}.
\end{eqnarray}
The last inequality can be seen as follows: 
when $\ti \ga$ reaches the ball ${{}^t B}_{J_4 }(
p_1)$, it must also be in ${{}^s B}_{J_4^5  }(
p_1)$ in view of Remark \ref{ballsremark}. So the two points of discontinuity on $\ti \ga$ may be joined smoothly
by a curve with length (with respect to $g(s)$) at most
$2J_4^{5}$. Call this curve $\hat \ga$. 
Hence $L_s(\hat \ga) \leq 2J_4^{5}+ L_s(\ti \ga)$, which implies
$L_s(\ti \ga) \geq L_s(\hat \ga) - 2J_4^{5} \geq d(x,y,s) -J_4^{6}$ as claimed.
So we have $ d(x,y,t)  \geq d(x,y,s) - J_4^{7}$ if we choose $\ep =
\frac{1}{N}$. Swapping $s$ and $t$, we see that 
\begin{eqnarray}
|d(x,y,t)-d(x,y,s)| \leq J_4^{7}
\end{eqnarray}
if $t,s \in (-\de,0]$, and $x,y  \in {{}^t B}_{N/4}(p_1)$ where
$\de=\de(N,J)$ and may depend on the solution, but does not depend on
$i$, as long as $i$ is large enough.

In  particular, 
${{}^t B}_{J^{100},\frac{N}{8}}(p_1) \subseteq {{}^s B}_{J^{50}, \frac{N}{4} }(p_1)  \subseteq {{}^t
  B}_{J^{5}, N }(p_1) $ for all $N> J^{100}$ and $i$ large enough,
for all $t,s \in (-\de,0]$ where $\de=\de(N,J)$ and may depend on the
solution, but does not depend on $i$. Notice, by taking a limit $
s\upto 0$, we see $f({{}^t B}_{J^{100},\frac{N}{8}}(p_1) ) \subseteq   {{}^{d_i} B}_{J^{50}, \frac{N}{4} }(x_1)$ (*).

Using \eqref{gradientriem}, and the evolution equation for the
curvature as in Section 8 of \cite{HaForm}, we see that 
\begin{eqnarray}
|f_*(g(t))- l_i|_{C^k({{}^{d_i}  B}_{J^{50}, N}(x_1),  g(t)) }
\leq \hat \ep  \label{gtllemm}
\end{eqnarray}
 if $t \in (-\de(k,\hat \ep), 0]$. We explain now why Inequality \eqref{gtllemm}
 is true. To see this, work with fixed geodesic coordinates $\phi:B_{i_0}(z)\to
 B_{i_0}(0)$ of radius larger
  $i_0>0$ at any point in $z \in {{}^{d_i}  B}_{J^{50}, N}(p_1)$ (these exist
 because of the curvature estimates of the previous theorem, Theorem
 \ref{flatnessthm}, and the non-collapsing estimates).
Writing $l_i$ in these coordinates, (we drop the $i$ in these
coordinates and call $f_*(g(t))$ also $g(t)$ in the coordinates)
 we have $\frac{1}{C} \de_{ij} \leq l_{ij}(\cdot)
\leq  C\de_{ij}$, $\sum_{j=0}^K |D^j l|^2(\cdot) \leq C$ on $B_{i_0}(0)$ for some $C$
not depending on $i$, where here $D$ is the standard euclidean
derivative 
(the manifolds are non-collapsed and satisfy the
curvature bounds of Theorem \ref{flatnessthm}: see Corollary 4.12 in
\cite{HaComp} for example).
Using the evolution equation for $g(t)$ and the curvature bounds, and
the fact that $g(0)-l=0$ we
see, using arguments similar to those of Section 8 in \cite{HaForm}, 
$e^{-C|t|} l \leq g(t) \leq e^{C|t|}l$, $|D (g(t)-l)|
\leq C|t|$, $|D^2(g(t)-l)| \leq Ct$ and so on. This implies
$\sum_{j=0}^k |{}^{g(t)}\grad^j (l -g(t))|_{g(t)}^2  + |{}^{l}\grad^j
(l -g(t))|_l^2 \leq \ep$ on $B_{i_0}(z)$ if $|t| \leq \de(C,\ep)$,
where $\de$ is chosen near enough to $0$. This finishes the
explanation of why Inequality \eqref{gtllemm} is true.
For a tensor $T$ and a metric $l$ defined on $U$, we have used the
following notation: 
$$ |T|^2_{C^k(U,l)} := \sum_{j=0}^k \sup_{x\in U} | {{}^l\grad}^j T|_{l}^2(x), $$ where 
${{}^l\grad}^j $ refers to the $j$th covariant derivative with
respect to $l$, if $j\in \N$, and ${{}^l\grad}^0 T := T$.
Note that this $\de$
doesn't depend on $i$.
In fact, what we have shown, is $\sum_{j=0}^k |{}^{g(t)}\grad^j
(f^{*}(l_i) -g(t))|_{g_i(t)} ^2(f^{-1}(z)) 
+ |{}^{l_i}\grad^j (l_i -f_*g(t))|_{l_i}^2(z) \leq \ep$  for all $t\in (-\de,0)$ if $z  \in
{{}^{d_i}  B}_{J^{50}, N}(p_1)$. Hence, using (*), we have also shown
$\sum_{j=0}^k |{}^{g(t)}\grad^j
(f^{*}(l_i) -g(t))|^2_{g(t)}(w) 
+ |{}^{d_i}\grad^j (l_i -f_*g(t))|^2_{l_i}(f(w)) \leq \ep$  for all $t\in
(-\de,0)$ if $w  \in {{}^t B}_{J^{100},\frac{N}{8}}(p_1)$.

Scaling the solution by $(\frac{\si}{J^{100}})^2$, and assuming $N =
\frac{8\ti N J^{100}}{\si}$ we see that
$|g(t)- f^* (l_i)|_{C^k({{}^{t}  B}_{\si, \ti N}(p_1),  g(t)) }
\leq \si$ 
and $|f_*(g(t))-  l_i|_{C^k({{}^{d_i}  B}_{\si, \ti N}(p_1),  \ti l_i) }
\leq \si$ 
 if $t \in (-\hat \de,0]$ (the original $\hat \ep$ is as small as we
 like) and 
$|d(x,y,t)-d(x,y,s)| \leq \si$
if $t,s \in (-\hat \de,0]$ and $x,y \in {{}^t B}_{\ti N}(p_1)$.
Choosing $i$ large enough, and a time $t_1 \in (-\hat \de, -\frac{\hat \de}{4})$
which corresponds to a good time of the original solution, we see that
we may assume without loss of generality, that $g_1:=g(t_1)$ satisfies
$\int_M |\Ricci(g_1)|^4 d\mu_{g_1} \leq \si$.
$g_1$ is our first metric.
It satisfies
$|g_1- f^*(l_i)|_{C^k({{}^{g_1}  B}_{\si, \ti N}(p_1),   l_i) } \leq
\al_1$, $|f_*(g_1) -l_i|_{C^k({{}^{d_i}  B}_{\si, \ti N}(x_1),
  l_i) } \leq \al_1$
and 
$|d_i(f(x),f(y))-d_{g_1}(x,y)| \leq \al_1$
on $ { {}^{g_1}  B}_{\ti N}(p_1)$,
where $\al_1 = \si$,
and $\int_M |\Ricci(g_1)|^4 d\mu_{g_1} \leq \al_1$.

Repeating the procedure, but scaling by $(\frac{(\si)^2}{J^{100}})^2$,
at the end, with 
$N =\frac{2 \times 8 \ti N J^{100}}{\si^2}$ 
leads to our second metric $g_2$, and $g_2$ satisfies (for a new
larger $i$)\\
$|g_2- f^*(l_i)|_{C^k({{}^{g_2}  B}_{\si^2, 2\ti N}(p_1),   g_2) } \leq
\al_2$, $|f_*(g_2) -l_i|_{C^k({{}^{d_i}  B}_{\si^2, 2\ti N}(x_1),
  l_i) } \leq \al_2$\\
and 
$|d_i(f(x),f(y))-d_{g_2}(x,y)| \leq \al_2$
on $ { {}^{g_2}  B}_{2\ti N}(p_1) $,
where $\al_2 = \si^2$,
and \\ $\int_M |\Ricci(g_2)|^4 d\mu_{g_2} \leq \al_2$.

And so on.
Choosing $\si_i$ to be an arbitrary  sequence with
$\si_i >> \si^i$ and $\si_i \to 0$ as $i \to \infty$  completes the proof.

\end{proof}

For convenience we introduce some notation which will help us describe
the phenomenon of metric annulli being $C^k$-close, as described in the
theorem above.
This phenomenon occurs at a number of points in the rest of the paper.
\begin{definition}\label{ckconvdef}
Let $(X,d_X)$, $(Y,d_Y)$ be complete, connected metric spaces. We
assume also that these spaces have a given 
Riemannian structure with at most finitely many (possible) singularities in the following sense: 
$N := X \backslash \{x_1, \ldots, x_L\}$ and 
$V:= Y \backslash \{y_1, \ldots, y_L\}$ are smooth manifolds, and
$l$ is a Riemannian metric on $N$ and $v$ on $V$.
For $0 <r<R\leq \infty$, $E \subseteq X$ an open set($E = X$ is allowed), and $x_0 \in \{1, \ldots,
x_L\}$, $y_0 \in \{y_1, \ldots, y_L\}$  we say that 
\begin{eqnarray}
d_{C^k}( E \cap { {}^{d_X} B}_{r,R}(x_0),  { {}^{d_Y} B}_{r,R}(y_0) ) \leq \ep
\end{eqnarray}
(we always assume $\ep <<\min(r,R-r)$)
if
\begin{itemize}
\item[(i)]  $E \cap { {}^{d_X} B}_{r,R}(x_0) \subseteq N$ and $ { {}^{d_Y}
    B}_{r,R}(y_0) \subseteq V$, and
\item[(ii)]  there exists a $C^{k+1}$ map  $f: E\cap{ {}^{d_X} B}_{r,R}(x_0) \to  V$,
such that $f$ is a $C^{k+1}$ diffeomorphism onto its image, ${
  {}^{d_Y} B}_{r+\ep,R-\ep}(y_0)  \subseteq  f( E\cap  { {}^{d_X} B}_{r,R}(x_0))$
\item[(iii)] $|d_X(w,x_0) - d_Y(f(w),y_0)| \leq \ep$ for all $w \in E
  \cap {
    {}^{d_X} B}_{r,R}(x_0)$: in particular
$ { {}^{d_Y} B}_{s+\ep,m-\ep}(y_0) \subseteq f(E \cap  {{}^{d_X}
  B}_{s,m}(x_0)) \subseteq { {}^{d_Y}
  B}_{s-\ep,m+\ep}(y_0)$ for all $0<r\leq s<m\leq R$ with $s+ \ep <
m-\ep$.
\item[(iv)]
$|f^*(v) -l|^2_{C^k(  E\cap{{}^{d_X} B}_{r,R}(x_0) ,l  )    } :=
\sum_{j=0}^k \sup_{ x \in {
    E\cap {}^{d_X} B}_{r,R}(x_0)  } | {{}^{l} \grad}^j (f^*(v)-l)|_l^2(x)
\leq \ep$ and $|v -f_* l|^2_{C^k(  { {}^{d_Y} B}_{r+\ep,R-\ep}(y_0) ,v
  )    } \leq \ep$.

\end{itemize}

\end{definition}
\begin{remark}
Note that in the Approximation Theorem above, Theorem \ref{approximationthm}, the $f$ that occurs there is also a
Gromov-Hausdorff approximation when considered as a map  on the balls
being considered (and $E = M$). Here we only
require condition (iii), which is weaker.
\end{remark}

\begin{remark}
The definition of $C^k$ close is coordinate free. This allows us to
compare elements of sequences of Annuli in a coordinate invariant way.
\end{remark}
\begin{remark}
From the definition we see, that if $(X,d_X)$, $(Y,d_Y)$,  are
metric spaces of the type occurring in the theorem then
$d_{C^k}( { {}^{d_X} B}_{r,R}(x_0),  { {}^{d_Y} B}_{r,R}(y_0) ) \leq
 \ep$  implies   \\
$d_{C^k}( { {}^{d_Y} B}_{r+4\ep,R-4\ep}(y_0),  { {}^{d_X} B}_{r+4\ep,R-4\ep}(x_0) ) \leq
 \ep$ ({\it almost symmetry})
\end{remark}

\section{Orbifold structure of the limit space}\label{orbifoldsection}

The flatness estimate \eqref{flatness} of the previous section, along
with the non-collapsing and non-expanding estimates (which survive
into the limit, as explained in (iii) at the beginning of Section
\ref{curvestsec}) guarantee that $X$ is actually a so called {\it generalised
   $C^0$ Riemannian orbifold } with only finitely
many isolated {\it orbifold points} : points $q \in X$ for which there
exists a neighbourhood $q \in U \subseteq X$ and a smooth diffeomorphism
$\phi : U \to \R^4$ are called {\it manifold points}, all other points
in $X$ are called {\it orbifold points}. These objects have been studied in
\cite{Tian}, \cite{And1}, \cite{BKN} . In the papers \cite{HM,HM2},
the authors also used 
generalised Riemannian orbifolds (they refer to them as {\it multifolds}: see
section 3 of \cite{HM2}) to prove an orbifold compactness result for solitons.
They were introduced and used  in the static (for example the
Einstein) setting by M. Anderson \cite{And1} (see also \cite{BKN}), to describe non-collapsing limits of Einstein manifolds.
The estimates required to show that $X$ is a generalised $C^0$ Riemannian orbifold are
contained in the previous section. {\it Generalised Riemannian orbifolds} can have a number of components
at each orbifold type point. In our case we will see that there is
exactly one component at each singular point. Before showing this,  
we state the general result which follows from the argument for
example in \cite{Tian} (see also \cite{And1}, \cite{BKN}).

We use the following notation
in the statement of the theorem and in the rest of the paper:
$D_{r,R} \subseteq \R^4$ is the standard
open annulus of inner radius $r\geq 0$ and outer radius $R\leq
\infty$, ($r<R$) centred at $0$: $D_{r,R} = \{ x \in \R^4 \ | \ |x| >r, |x| <R \}$.
$D_r$ represents the open disc of radius $r$ centred at $0$: $D_r:=
\{ x \in \R^4 \ | \ |x| <r \}$. Note $D_{0,R} = \{ x \in \R^4 \ | \
|x| >0, |x| <R \} = D_R \backslash \{0\}$.
\begin{theo}
$X$ is a {\it generalised
   $C^0$ Riemannian orbifold} in the following sense.
\begin{itemize}
\item[(i)] $X\backslash \{x_1, \ldots, x_L\}$ is a manifold, with the
  structure explained above in Lemmata
\ref{manifoldstructure} and \ref{rmanifoldstructure}.
\item[(ii)]There exists an
$r_0 >0$ small, and an $N< \infty$ such that the following is true.
Let $x_i \in X$ be one of the singular points. Then the number of connected
components $(E_{i,j}(r))_{j\in \{1, \ldots , \ti N_i\}}$ of ${{}^{d_X}B}_{r}(x_i)\backslash \{x_i\}$
in $X\backslash \{x_1,\ldots,x_L\}$ is finite and bounded by $N$ (that is $\ti N_i
\leq N$) for $r \leq r_0$, where $N = N(\si_0,\si_1) < \infty$.
\item[(iii)] Fix $i \in \{1, \ldots, L\}, j \in \{1, 2, \ldots, \ti N_i\}$, and let $E = E_{i,j}(r_0)$ be one of the
  components from $(ii)$. Then there exists a
  $0<\ti r \leq  r_0 $ and a
  diffeomorphism $k:  D_{0,\ti r } \to k(D_{0,\ti r}) \subseteq \ti E$ where $\ti E$ is the
  universal covering space of $E \cap(\cup_{j=1}^{N_i} E_{i,j}(\ti r(1+\ep))$,
  such that the covering map $\pi_E: \ti E \to E$ is finite and for $r
  \leq \ti r$ we have
\begin{eqnarray}
\sup_{ D_{0,r} } |(\pi_E \of k)^*l - \de|_{C^0(D_{0,r})} \leq \ep_1(r)
\end{eqnarray}
where $ \ep_1(r)\geq 0$ is a decreasing function with $\lim_{r \downto 0}
\ep_1(r) = 0$, $\de$ is the standard euclidean metric on
$\R^4$ or subsets thereof, $|\cdot|_{C^0(L)}$ is the standard
euclidean norm on two tensors, $|v|^2_{C^0(L)}:= \sup_{x \in L}
\sum_{i,j =1}^n |v_{ij}(x)|^2$ for any set $L \subseteq \R^4$
and any two tensor $v =v_{ij}dx^i dx^j$. 
\end{itemize}
\end{theo}

\begin{proof}
(i) was shown above.
(ii) follows from the non-expanding and non-collapsing estimates,
exactly as in the proof of Lemma 3.4 in \cite{Tian}. 

(iii) Follows as in the proof of Lemma 3.6 in \cite{Tian} using
the {\it flatness estimates}, \eqref{flatness}, and the
non-collapsing and non-expanding estimates.
\end{proof}
\begin{remark}
Some of the proofs of the Lemmata mentioned here (Lemma 3.6 and Lemma 3.4 of
\cite{Tian}) can be simplified at certain points, by using that
$\inj(B_r(p)) \geq c_0 r$ for all balls $B_r(p)$ which are compactly
contained in $(D,h)$ where $(D,h)$ is any smooth, open flat ($\Riem(h) =
0$) non-collapsed, non-inflated (on all scales) manifold without
boundary: this follows from the injectivity radius estimate of Cheeger-Gromov-Taylor, Theorem 4.3 in \cite{CGT}, whose
proof is local.
\end{remark}

The construction of this $k$ in \cite{Tian} (see Lemma 3.6 in \cite{Tian}) is achieved by pasting
together maps $\phi_i:D_{\frac{1}{2^{i+2}}, \frac{1}{2^i} }\to
  \pi^{-1}(B_{\frac{1-\ep}{2^{i+2}}, \frac{1+\ep}{2^{i}}})$ where
  $i\in \N$. That is
  $\phi_1, \phi_2, \ldots$ are first constructed, and then $\phi_1$ is
  pasted to $\phi_2$ and $ \phi_2$ to $\phi_3$ and so on. This leads
  to a map $k$ with the properties given in the theorem above: see
 the proof of Lemma 3.6 in \cite{Tian}.
We construct a $\phi$ here using the method described in the proof of
Lemma 3.6 in \cite{Tian}  with some minor modifications: the explicit
construction will be used in later sections.

As shown in the proof of Lemma 3.6 of \cite{Tian}: 
if we scale, $l_i = (2^{i+2})^2l$, $d_i = 2^{i+2}d_X$, then
 \begin{eqnarray*}
d_{C^k}(     ({^{g(i)} B}_{1/2,4}(0),g(i) )  ,  ({^{d_i}
  B}_{1/2,4}(x_1)\cap E,l_i) ) \leq \ep(i) \to 0 \ \mbox{ as } i \to
\infty, 
\end{eqnarray*}
where
$({^{g(i)} B}_{1/4,4}(0),g(i)) \subseteq( (\R^4
\backslash \{0\}) / \Gamma(i),g(i))$, and $g(i)$ is the standard metric on $ (\R^4 \backslash \{0\}) /
\Gamma(i)$, and $\Gamma(i)$ is some finite subgroup of $O(4)$ with finitely
many elements (less than or equal to $N$ elements, $N$ independent
of $i$).
Hence there exists a diffeomorphism 
\begin{eqnarray}
&&v_i: ({{}^{g(i)}
  B}_{1/2 , 4}(0) ,g(i)) \to 
({^{d_i} B}_{1/2 - \ep(i) ,4+\ep(i) }(0) \cap E ,l_i) \subseteq
(E,l_i), \label{videfn}
\end{eqnarray}
such that
\begin{eqnarray*}
&&|v_i^*l_i -g(i)|_{C^k(   {{}^{g(i)}
  B}_{1/2 , 4}(0) ,g(i))} + |(v_i)_*g(i) - l_i|_{C^k(  { {}^{d_i} B}_{1/2
  +\ep(i) ,4- \ep(i) }(0) \cap E ,l_i)} \leq \ep(i)
\end{eqnarray*}

In the following, $\ep(i)>0$ will refer to positive numbers with the
property that $\ep(i) \to 0$ as $i \to \infty$. As
the notation suggests, in fact this $\Gamma(i)$ (and hence $g(i)$) could depend on the sequence we take,
and could depend on $i \in \N$. However, 
$\inj( (\R^4 \backslash \{0\}) / \Gamma(i),g(i)) (x) \geq |x|i_0$ for some
fixed $i_0 >0$, where $|x| =|\ti x|_{\R^4}$ is the euclidean norm of the point
$x$  lifted to $\ti x \in \R^4$ (any such $\ti x$ has the same
euclidean distance from the origin, regardless of which $\ti x$, covering  $x$, we choose)
[Explanation 1: this follows in view of the construction: 
for any ball $ {{}^{d_i} B}_r(x)
\subseteq {^{d_i} B}_{1,4}(x_1) \cap E$ we have $r^4 \si_1 \geq
\vol({{}^{d_i} B}_r(x)) \geq r^4\si_0$, and the norm of the
curvature tensor on  ${^{d_i} B}_{1,4}(x_1)$ goes to zero as
$i \to \infty$. 
Hence $\inj({{}^{d_i} B}_{1/100}(0),l_i)(x) \geq i_0$ for any $x \in
{^{d_i} B}_{\frac{5}{4},3}(x_1)$, for some
$i_0>0$, if $i$ is large enough, 
in view of the injectivity radius estimate  of Cheeger-Gromov-Taylor
contained in  Theorem 4.3 in \cite{CGT}). Hence, using 
$d_{C^k}(     ({^{g(i)} B}_{1/2,4}(0),g(i) )  ,  ({^{d_i}
  B}_{1/2,4}(x_1)\cap E,l_i) ) \leq \ep(i)$, we see that we have 
$\inj({{}^{g(i)} B}_{1/100}(x),g(i))(x) \geq i_0/2$ for some $i_0>0$
for any $x \in 
({{}^{g(i)} B}_{\frac{3}{2},2}(0),g(i))$, if $i$ is large enough. ]

Let $\pi_i: \R^4 \backslash \{0\} \to (\R^4 \backslash \{0\}) /
\Gamma(i)$ be the standard projection,  and $x \in (\R^4 \backslash \{0\})/
\Gamma(i)$,
$(\pi_i)^{-1}(x) =  \{x_1, \ldots, x_N\}$. $\pi_i$ is a covering map
and a local isometry, and
using the fact that $\inj(\R^4 \backslash \{0\} / \Gamma(i),g(i)) (x)
\geq |x|i_0,$ we see that $d_{\R^4}(x_k,x_l) \geq (i_0|x|)/20 >0$ in $\R^4$
for $x_k, x_l \in (\pi_i)^{-1}(x)$, $k\neq l$.


Let $\psi_i:D_{1,4} \to E$ be the
natural map $\psi_i = v_i \of \pi_i|_{D_{1,4}}$ where \\$v_i: ({{}^{g(i)}
  B}_{1/2 , 4}(0) ,g(i)) \to 
({^{d_i} B}_{1/2 - \ep(i) ,4+\ep(i) }(0) \cap E ,l_i) \subseteq (E,l_i) $ is
the map defioned in \eqref{videfn} above,  and $\pi_i: \R^4 \backslash \{0\}
\to  (\R^4 \backslash \{0\} )/ \Gamma(i)$, the standard projection, is
as above.
Define $\phi_i(x) = \psi_i(2^{i+2} x)$ for $x \in D_{\frac{1}{2^{i+2}}, \frac{1}{2^i} }$: this is the unscaled version
of $\psi_i$.
Later we will paste the $\phi_i$'s together. To do this, it is
conveniant to work at the scaled level.
We will require that neighbours $\phi_i$ and $\phi_{i+1}$ are close to one
another for all $i\in \N$, in a $C^k$ sense (to be described) on their common domain of definition, at least at the scaled
level.
To show this, we have to compare neighbours $\phi_i$ and $\phi_{i+1}$,
for all $i\in N$, on their
common domain of definition $D_{\frac{1}{2^{i+2}}, \frac{1}{2^{i+1}} }$.  We do this at the scaled level: $\psi_i:
D_{1,4  } \to E$ is as defined above, $\psi_i(x) =
\phi_i(\frac{1}{2^{i+2}}x)$, and we define
$\eta_{i+1}:D_{\frac{1}{2},2} \to E$ by 
$\eta_{i+1} (x) =
  \phi_{i+1}(\frac{x}{2^{i+2}}) = \psi_{i+1}(2x)$

Notice that in defining the $\psi_i's$, we have the freedom to change
the coverings $\pi_i$ by a deck transformation, that is by an element
$A \in O(4)$.
Also, in view of the definitions, and the notion of convergence
introduced in Definition \ref{ckconvdef}, we have
$(\psi_i)^*(l_i)$ is $C^k$ close to $\de$ on 
$D_{1+\ep(i),4-\ep(i)}$ and  $(\eta_{i+1})^*(l_i)$ is $C^k$ close to
$\de$ on  $D_{1/2+\ep(i),2-\ep(i)},$ in view of the fact that
$(\eta_{i+1})^*(l_i)(x) = (\psi_{i+1})^*(l_{i+1})(2x)$

\hfill\break\noindent
{\bf Step 1.} For all $i \geq N \in \N$ the following is true: 
By changing  the map $\pi_{i+1}$ by an
element $A \in O(4)$, if necessary,  we can assume
that the pair $\psi_i$ and $\eta_{i+1}$ are, for sufficiently large
$i \in \N$, $C^k$ close to one another on their common domain of definition,
in a sense which we now describe: 
take any arbitrary ball ${{}^{\de}B}_s(y) \subseteq
{D}_{1+\de/2,2-\de/2}$ with some fixed $s>0$ 
$s \leq \frac {i_0}{10}, s \leq \frac{\de}{ 10}$ where $y \in 
{D}_{1+\de,2-\de}$, is  in the common domain of definition
of $\psi_{i}$ and $\eta_{i+1}$, where $\de>0$ is some fixed small number.
Then $d_{i}(\psi_i(x),\eta_{i+1}(x)) \leq \ep(i)$ for all $x \in 
{D}_{1+\de/2,2-\de/2}$ and, 
$\psi_i(B_s(y)) \cup \eta_{i+1}(B_s(y)) \subseteq {{}^{d_i}B}_{2s}(\ti y),$
$|\theta \of \psi_i - \theta \of \eta_{i+1} |_{C^k(B_s(y),\R^4)} \leq \ep(i),$
where $\theta: {{}^{d_i} B}_{2s}(\ti y)  \to {{}^{\de} B}_s(0) \subseteq \R^4$ are geodesic
coordinates on$(M,l_i)$ centred at the point $\ti y  = \eta_{i+1}(y)$
(note these coordinates exist, in view of the fact that
$d_{C^k}( ({{}^{d_i} B}_{1,4}(x_1) \cap E,l_i), ({^{g(i)}
  B}_{1,4}(0),g(i)) ) \leq \ep(i) $).

\hfill \break \noindent
{\bf Proof of Step 1.}
\hfill \break \noindent
Assume this is not the case. Then we find a sequence for which this is not true. Taking a
subsequence (we denote the subsequence of the pairs
$\psi_i,\eta_{i+1}$ also by $\psi_i,\eta_{i+1}$), we see that 
$({{}^{g(i)}B}_{1, 2}(0),g(i)) $,
and $({{}^{g(i+1)}B}_{1, 2 }(0),g(i+1)) $ converge to the
same limit space, $(B_{ 1, 2   }(0),g) \subseteq (\R^4
\backslash \{0\} ,\de)/ \Gamma$ (in the sense
of $C^k$ convergence described above in Definition \ref{ckconvdef}), where $\Gamma $ is a
finite subgroup of $O(4)$ with finitely many (bounded by $N$) elements
: the argument in the beginning of the proof of Lemma 3.6 in
\cite{Tian}, for example, gives us this fact.

Let us denote by $Z_i: ({{}^{g(i)}B}_{1, 2}(0),g(i) )  \to
(B_{1-\ep(i), 2+\ep(i)}(0),g)$\\
and $Z_{i+1}:  ({{}^{g(i+1)}B}_{1, 2}(0),g(i+1) )  \to
(B_{1-\ep(i), 2 +\ep(i) }(0),g)$ the natural maps which are
diffeomorphisms and almost $C^k$ local  isometries  onto their images: these must exist in view of
this convergence.
 
Let us denote by  $R_i: (E\cap {{}^{d_i}B}_{ 1, 2 }(x_1),l_i) \to
(B_{1-\ep(i), 2+\ep(i) }(0),g)$
the natural map, which is also a diffeomorphism onto its image and
almost an local isometry, that
arises in this way: $R_i =  Z_i \of (v_i)^{-1}$ (if $\ep(i)$ changes
in the proof, but the new constant $\ti \ep(i) \to 0$ as $i\to
\infty$, 
then we denote $\ti \ep(i)$ by $\ep(i)$ again).
Then $R_i \of \psi_i$ converges (after taking a subsequence) to a map $\hat \pi:D_{1,2} \to
(B_{1, 2}(0),g) \subseteq ( (\R^4 \backslash \{0\},\de)) / \Gamma$ which is a covering map, with
$(\hat \pi)^{*}g = \de$ and $R_{i} \of \eta_{i+1}$ converges (after taking a subsequence) to a map $\ti \pi:D_{1,2} \to
(B_{1, 2 }(0),g)$ which is a covering map with $\ti
\pi^* g =\de$, and the convergence is in the usual $C^k$ sense of
convergence of maps between fixed smooth Riemannian manifolds
[Explanation: 
$R_i \of \psi_i, R_i \of \eta_{i+1}:D_{1+\ep(i),2 - \ep(i)} \to (B_{1-2\ep(i), 2+2\ep(i)}(0),g),$ have $ (R_i \of
\psi_i)^*g $ and $(R_i \of \eta_{i+1})^*(g)$ are  $\ep(i)$ close in the $C^k$ norm to $\de$, and hence, taking a
subsequence, we obtain maps $\hat \pi, \ti \pi:D_{1,2} \to
(B_{1, 2}(0),g)$ with $\hat \pi^*(g) = \ti \pi^*(g) = \de$. We work now
with $\hat \pi$: the same argument works for $\ti \pi$.
For any $x \in D_{1,2} $ we can find a small
neighbourhood $ U \subsub D_{1,2} $ with $x \in U$ such that 
$\hat \pi(U) \subseteq {{}^g B}_s(p)$ where  ${{}^g B}_s(p)\subseteq ({{}^g B}_{1, 2}(0),g) $
is a geodesic ball and there exist geodesic coordinates 
 $\beta:{{}^g B}_s(p) \to {{}^{\de}B}_s(0)$  ($s$ small enough).
Then $\beta \of \hat \pi: U \to \R^4$ is well defined, and has
$\det(D (\beta \of \hat \pi)) = 1$ and hence
$\hat \pi:D_{1,2} \to {{}^gB}_{1,2}(0)$ is a local diffeomorphism. The map is, per construction, surjective
(here the definition of the convergence of annuli from Definition \ref{ckconvdef}
is used). It is also proper, since by construction, $D_{r,s} \subseteq
D_{1,2}$ is mapped onto $(B_{r,s}(0),g) \subseteq (B_{1,2}(0),g(0))$
(here the definition of the convergence of annuli from Definition \ref{ckconvdef}
is used). Hence, $\hat \pi$ is a covering map (see, for example, Proposition 2.19 in \cite{Lee}).  End of the Explanation].
Hence the two maps differ only by a deck
transformation, which is an element $A$  in $O(4)$: 
$\ti \pi  = \hat \pi \of A$. 
Before taking a limit, we can change $\eta_{i+1}$ by this element, 
$\hat \eta_{i+1}:= \eta_{i+1} \of A$.  Remembering the definitions
of $\eta_{i+1}$ and $\psi_{i+1}$, we see that we have $\hat \eta_{i+1}(x) = (\eta_{i+1} \of A)(x) = \psi_{i+1} (
A(2x)) =
(\psi_{i+1} \of A)(2x) = ((v_{i+1}) \of (\pi_{i+1}) \of A)(2x) $.
That is we change the covering map
$\pi_{i+1}$ to the covering map $\hat \pi_{i+1} =  \pi_{i+1} \of A,$ and then
define $\hat \eta_{i+1}:= (v_{i+1}) \of \hat \pi_{i+1}(2x)$:
we have this freedom in the choice of our $\pi_{i+1}$'s.
Now both $R_{i}
\of \psi_{i}$ and $R_{i} \of \hat \eta_{i+1} = R_{i} \of
\eta_{i+1} \of A$
converge to  $\ti \pi$ in the
sense explained above. In
particular, returning to $ ({{}^{d_i}B}_{1, 2}(x_1),l_i)$ with
$(R_i)^{-1}$ and writing things in geodesic coordinates, we see that
$\hat \eta_{i+1}$ is arbitrarily close to $\psi_i$, which leads to a
contradiction. Here we used the following fact. In geodesic
coordinates $\beta: B_s(p) \subseteq (B_{1+\ep(i),2-\ep(i)}(0),g) \to B_s(0)
\subseteq \R^4$, the metric is
$\de$. Hence for geodesic coordinates $\ga: {{}^{d_i}B}_{s/2}(z) \subseteq
{{}^{d_i}B}_{1,2}(x_1) \to B_{s/2}(0) \subseteq \R^4 $ with $R_i(z) = p$,
we see
$\beta \of R_i \of \ga^{-1}:B_{s/2}(0) \to \R^4$ is $C^k$ close to an
element in $O(4)$, in view of, for example, Corollary 4.12 in \cite{HaComp}.
End of the Explanation]. 
We assume in the following, that we have made the necessary modifications to the $\phi_i's$
(note, that in changing $\pi_{i}$ by a deck transformation, we are also
changing the $\phi_i's$ and hence the $\psi_i$'s), 
so that the above $C^k$ {\it closeness} of neighbours
$\psi_i,\eta_{i+1}$ for all $i\in \N$ large enough is
guaranteed. These modifications are made inductively: for $i \in \N$
sufficiently large, first change
$\pi_{i+1}$ by a deck transformation if necessary, then $\pi_{i+2}$ by
a deck transformation if necessary, then $\pi_{i+3}$ by a deck
transformation if necessary, and so on. \hfill\break\noindent
{\bf End of Step 1.}
\hfill\break\noindent
Now, {\bf Step 2}, we explain how to join $\phi_i$ and $\phi_{i+1}$,
assuming we have made the necessary modifications to the $\phi_i's$,
as explained in Step 1. The resulting map, at the unscaled level will
be $\phi$.
\hfill\break\noindent

For large $i\in \N$, we know that $(v_{i}^{-1} \of \psi_i):D_{1+\ep(i),4-\ep(i)}
\to (B_{1,4}(0),g(i))  $ and $(v_{i}^{-1} \of \eta_{i+1}):D_{1/2+\ep(i),2-\ep(i)}
\to (B_{1,4}(0), g(i) ) $ are well defined smooth maps which
are $C^k$ close to one another on  the common domain of definition
$D_{1+\ep(i),2-\ep(i)}$ and $C^k$ close to 
$\pi_i:D_{1+\ep(i),2-\ep(i)} \to
(B_{1 ,2}(0),g(i))$ on $D_{1+\ep(i),2-\ep(i)}$ 
 in the sense just described. Lifting these maps  to $D_{0,4}(0) \subseteq \R^4$ 
with respect to the covering $\pi_i:D_{0,4}(0) \to 
(B_{0,4}(0),g(i))  $, we see that we obtain maps
$\ti \psi_i: D_{1+\ep(i),4-\ep(i)} \to D_{1,4}(0)$
and 
$\ti \eta_{i+1}: D_{1/2+\ep(i),2-\ep(i)} \to D_{1,2}(0) $ (these
maps are lifts with respect to $\pi_i$, that is $\pi_i
\of \ti \psi_i = (v_{i}^{-1} \of \psi_i)$, $\pi_i \of \ti \eta_{i+1}
= (v_{i}^{-1} \of \eta_{i+1})$, and these lifts exist, since the
domain of the maps we are lifting are simply connected: see
Corollary 11.19
in \cite{Lee2}) 
which are $C^k$ close to the same element in $O(4)$  on
$D_{1+\ep(i),2-\ep(i)},$ 
 which is without loss of
generality the identity (transform the lifts $\ti \psi_i ,\ti \eta_{i+1}$
by the inverse of this element in the target space: the resulting maps
are still lifts).
Also $\ti \psi_i^*(\de)$ and $\ti \eta_{i+1}^*(\de)$ are $C^k$ close
to $\de$, on their domains of definition, and hence 
$\ti \psi_i$ is $C^k$ close to an element in $O(4)$ on
$D_{1+\ep(i),4-\ep(i)} $ and 
$\ti \eta_{i+1}$ is $C^k$ close to an element in $O(4)$ on
$D_{1/2+\ep(i),2-\ep(i)} $, and using the information in the previous
line, this element is the identity in each case.

Defining 
\begin{eqnarray}
&&\ti \phi_i:D_{1/2 + \ep(i), 4- \ep(i)} \to  { {}^{d_i}
  B}_{1/2,4}(x_1),  \cr
&& \ti \phi_i := v_i \of \pi_i \of (\eta \ti \psi_i +
(1-\eta) \ti \eta_{i+1}) \label{tiphiidef}
\end{eqnarray}
where $\eta:\R^4 \to \R^+_0$ is a smooth cutoff function, with
$0\leq \eta \leq 1$,  $\eta = 0$ on $D_{0,2- 2\de} $, $\eta =
1$ on $D_{2-\de,\infty}$, (**) we 
obtain a smooth map,  which is  equal to $\eta_{i+1}$ on $D_{1/2 + \ep(i),2-2\de} $ and equal
to $\psi_i$ on $D_{2-\de,4-\ep(i)}$, and for which $(v_i)^{-1}
\of \ti \phi_i:D_{1/2 +2 \ep(i), 4-2 \ep(i)} \to   { {}^{g(i)}
  B}_{1/2,4}(0)$ is $C^k$ close to $\pi_i$.
The map $\ti \phi_i$ satisfies
\begin{eqnarray}
(1-\ep(i))|x| \leq  d_{i}( \ti \phi_i(x),x_1)  \leq (1+\ep(i))|x|  \label{disttphii}
\end{eqnarray}
on $D_{1/2 +\ep(i), 4-\ep(i)}$,
by construction. We can now define $\phi:D_{0,\ep} \to X \backslash \{x_1\}$. For $x \in
[\frac{1-7\de}{2^{i+1}}, \frac{1-4\de}{2^i}]$ and $i \in \N$ large, we define
$\phi(x):= (\ti \phi_i)( 2^{i+2} x).$ 
This map is smooth and well defined: fix $i \in \N$, and let $x \in 
[\frac{1-7\de}{2^{i+1}},\frac{1-2\de}{2^{i+1}}].$
Then $\phi(x) = \ti \phi_i(2^{i+2} x) = \eta_{i+1}(2^{i+2} x) =
\phi_{i+1}(x)$, and if $x  \in [\frac{1}{2^{i+1}},\frac{1-4\de}{2^i}]$, then 
$\phi(x) = \ti \phi_i(2^{i+2} x) = \psi_i(2^{i+2} x) = \phi_i(x)$.
{\bf This finishes Step 2}.

We examine, in the following, various properties of $\phi$.

By construction, $\phi: D_{0,\ep} \to X$ satisfies: $|d_X(\phi(x),x_1) -|x||
\leq \ep(|x|)|x|$, where $\ep(|x|) \to 0$ as $|x| \to 0$: this follows from  \eqref{disttphii} and the
definition of $\phi$.
We consider $\ti V:=
  \phi^{-1}(\phi(D_{0,\ep} ))$ and $V:= \phi(D_{0,\ep} )$.
We claim that $\phi|_{\ti V}: \ti V \to V$ is a covering map if $\ep>0$ is
small enough. Note: we do {\bf not} claim that $V$ or $\ti V$ have smooth boundary.
We first note, that the cardinality of $(\phi|_{\ti V})^{-1}(x)$ for $x \in
\ti V$ is bounded if $\ep$ is small enough.
Assume there are points $z_1, \ldots , z_K$, $z_s \neq z_j$ for all
$s\neq j \in \{1, \ldots K\}$, with 
$\phi(z_1) = \phi(z_2) = \ldots \phi(z_K) = m$.
We can always find an $i\in\N$ with
 $z_1 \in [\frac{1-5\de}{2^{i+1}} ,  \frac{1-5\de}{2^{i}}],$ and hence
$  (1-\ep(i)) |z_1| \leq  d_X(m,x_1) \leq  (1+\ep(i)) |z_1|$ implies
$ (1-\ep(i)    \frac{1-5\de}{2^{i+1}} \leq d_X(m,x_1) \leq (1+\ep(i)) \frac{1-5\de}{2^{i}}$
and hence $\frac{(1-\ep(i))}{(1+\ep(i))}  \frac{(1-5\de)}{2^{i+1}}\leq     |z_j| \leq
 \frac{(1-5\de)}{2^{i}}\frac{(1+\ep(i))}{(1-\ep(i))}$
for $j = 1,\ldots, K$. Hence, after scaling by $2^{i+2}$, 
we have $ \ti z_1, \ldots, \ti z_K \in  [2-11\de,4-19\de]$ with $\ti
\phi_{i}(\ti z_1) = \ti \phi_{i}(\ti z_2) = \ldots \ti \phi_{i}(\ti z_K)$.
At the scaled level, we know that, $(v_i)^{-1} \of \ti \phi_i: D_{1/2 +2 \ep(i), 4-2 \ep(i)}
\to (B_{1/2,4},g_i(0))$ is $C^k$ close to $\pi_i$, the standard
projection, and the pull back of $g(i)$ with this map is $C^k$ close
to $\de$ on $D_{1/2 +2 \ep(i), 4-2 \ep(i)}$. In fact $(v_i)^{-1} \of
\ti \phi_i = \pi_i \of h_i$ where $h_i:D_{1/2+\ep(i),4-\ep(i)} \to \R^4$ is $C^k$
close to the identity.
In particular, $(v_i)^{-1} \of \ti \phi_i(B_{s/2}(z)) \subseteq
B_{s}((v_i)^{-1} \of \ti \phi_i(z))$
for any $z \in [2-11\de,4-19\de]$ for $0<s\leq \frac{i_0}{10}$ fixed and small.
Let $\psi: {{}^{g(i)}B}_s(\hat z_j) \to B_s(0) \subseteq \R^4$ be geodesic coordinates in 
$(B_{1/2,4},g_i(0))$, where $\hat z_j = (v_i)^{-1} \of \ti \phi_i(\ti z_j)$. The map $\psi \of (v_i)^{-1} \of \ti
\phi_i:B_{s/2}(\ti z_j) \to \R^4$ is $C^k$
close to an isometry $B(i,j)= A(i,j) + \tau_{\ti z_j}$of $\R^4$, where
$A(i,j) \in O(4)$ and $\tau_{\ti z_j}$ is $\tau_{\ti z_j}(x) = x-\ti z_j$, and hence after a rotation
in the geodesic coordinates and a translation, $C^k$
close to the identity. In particular, this map is a
diffeomorphism when restricted to
$B_{s/2}(\ti z_j)$, and hence $\ti z_i \notin B_{s/2}(\ti z_j) $ for all $j \neq
i$. Hence, 
$\vol(D_{1/2,4}) \geq \sum_{j=1}^K \vol(B_{s/2}(\ti z_j)) \geq
K \omega_4 (s/2)^4$ which leads to a contradiction if $K$ is too large.

If we scale the map $\phi:D_{[\frac{1-7\de}{2^{i+1}} ,  \frac{1-4\de}{2^{i}}]}
  \to X$ by $2^{i+2}$, that is let $\hat \phi:D_{ [2-14\de ,  4-16\de]}
  \to X$ be defined by $\hat \phi(x) = \phi(\frac{x}{2^{i+2}} )$,
  then we obtain  the map $\ti \phi_i$: $\hat \phi = \ti \phi_i|_{[2-14\de ,  4-16\de]}$.
The argument above, shows that $(v_i)^{-1} \of \ti \phi_i|_{B_{s/2}(z)}: 
B_{s/2}(z) \to \R^4$ is a diffeomorphism for all $|z| \in [2-10\de,4-18\de]$
if $s<< \de$,$s<\frac{i_0}{100}$, $i$ sufficiently large. That is $\hat \phi|_{B_{s/2}(z)} = \ti \phi_i|_{B_{s/2}(z)}: 
B_{s/2}(z) \to X$ is a diffeomorphism for all $|z| \in
[2-10\de,4-18\de]$ and hence 
$\hat \phi|_{D_{(2-10\de,4-18\de)}}$ is a local diffeomorphism, which
tells us, scaling back, that 
$\phi: D_{(\frac{1-5\de}{2^{i+1} },  \frac{1-(9/2)\de}{2^{i}})}
  \to X$ is a local diffeomorphism.

That is, $\phi:D_{0,\ep} \to X$ is a local diffeomorphism, if
$\ep>0$ is small enough.

Hence $V:= \phi(D_{0,\ep})$
  is open if $\ep>0$ is small enough (this corresponds to $i$ being
  sufficiently large), and $\phi:\ti V:= \phi^{-1}( V) \to V$ is a
  local diffeomorphism and an open map. $V$ is connected, as it is the
  image under a continuous map of a connected region. 
In fact $\ti V$ is also connected: this will be shown below.

$\phi: \ti V \to  V$ is proper: Let $(x_k)_{k \in \N}$ be a sequence in
$K \subseteq V$, where $K$ is compact in $ V$. This means, there is a
subsequence of $x_i$ (also denoted $x_i$) such that $x_i \to x \in K
\subseteq V,$ $x =\phi(m)$ for some $m \in D_{0,\ep}.$
Let $z_1, z_2, \ldots, z_N \in \phi^{-1}(x)$ be the finitely many points in
$\ti V$ with $\phi(z_j) =m$.
We can choose a small neighbourhood $U_j$ of each one, such that $U_i
\subsub \ti V$ and $\phi|_{U_j}:U_j \to \phi(U_j)$ is a diffeomorphism,
and without loss of generality $\phi(U_j) = U \subsub  V $ for all
$j$, and
$\phi(m) \in U$.
Hence {\it any} sequence $y_k \in  \phi^{-1}(K)$ with $\phi(y_k) =x_k$ has a
convergent subsequence, $y_k \to z_i$ as $k \to \infty$ for some $z_i
\in \{z_1, \ldots z_N\}$. Hence $(\phi|_{ \ti V})^{-1}(K)$ is sequentially compact in
$\ti V$. That is $\phi:\ti V \to V$ is proper. That is, $\phi: \ti V \to
V$ is a proper, surjective, local diffeomorphism. In particular lifts
$\ti \ga:I \to \ti V$ of
curves $\ga:I \to  V$,  $I = [a,b] \subseteq \R$, always exist and
are uniquely determined by their
starting points $\ti \ga(0)$ which is an arbitrary point in $\phi^{-1}(\ga(0))$.\\
$\ti V$ is also connected. Let $\hat x$ and $\hat y$ be points in $\ti
V$ and
$x= \phi(\hat x) \in V$, $y = \phi(\hat y) \in  V$.
$x = \phi(\hat x) \in \phi(D_{0,\ep})$ implies $x = \phi(x_0)$ for an
$x_0 \in D_{0,\ep}$. Let $x_1$ be the point $x_1 = x_0/4$. 
Then $x_1 \in D_{0,\ep/4}$ and $\phi^{-1}(\phi(x_1)) \in D_{0,\ep/3}$ if $i$
is sufficiently large, in view of the construction of $\phi$ (see the above).

Joining $x_0 $ to $x_1$  with a ray $\al: I \to D_{0,\ep}$ (w.r.t to
the euclidean metric) which points into $0$ and pushing this down to
$V$ again with $\phi$, we obtain a continuous map 
 $\si =  \phi \of \al: I \to V$ with $\si(0) = \phi(x_0) =
 \phi(\hat x)$ and $ \si(1) =
 \phi(x_1)$.
Taking the lift of this map, and using the starting point $\hat x$, we obtain a continuous curve $\ti \si:I
\to \ti V$ with 
$\ti \si(0) = \hat x$ and 
$\ti \si(1) \in \phi^{-1}(\phi(x_1)) \in D_{0,\ep/3}.$ We may perform
the same procedure with $y$ to get a continuous curve $\ti \be:I \to
\ti V$ with
$\ti \be(0)= \hat y$ and $\ti \be(1) \in D_{0,\ep/3}$. We may join
$\ti \be (1) $ to $\ti \si(1)$ in $D_{0,\ep/3}\subseteq \ti V$ with a curve $T:I \to D_{0,\ep/3}$, as this space is
 connected.
Hence, following the curve $\ti \si$ from $\ti \si(0) = \hat x$  to
$\ti \si(1)$ in $\ti V$ and then from $\ti \si (1) $ to $\ti \be(1)$ with
$T$ and then from $\ti \be (1)$ to $\ti \be(0)$ by  going backwards
along the curve $\ti \be$, we see that we have constructed a
continuous curve in $\ti V$ from $\hat x$ to $\hat y$ as required. \\
Hence $\ti V$ is also connected.
\\
That is, $\phi: \ti V \to V$
is a proper, surjective, local diffeomorphism, between two path
connected spaces, and hence $\phi: \ti V \to V$ is  a covering
map (see Proposition 2.19 in \cite{Lee}).\\


In fact, $\ti V$ is simply connected if $\ep>0$ is sufficiently small,
and hence $\ti V$ is the universal
covering space of $V$.  We explain this now.
Let $i$ be sufficiently large, and we consider the map
$\ti \phi_i: D_{1/2 + \ep(i), 4- \ep(i)} \to X $ from above.
$v_i \of \ti \phi_i: D_{1/2 + \ep(i), 4- \ep(i)} \to (B_{1/2,
  4},g(i))$ is $C^k$ close to $\pi(i)$ as shown above.
In particular, $\ti \phi_i|_{B_s(x)}:B_s(x) \to X$ is a  
diffeomorphism  onto its image and $\ti \phi_i(B_{s/8}(x)) \subseteq
B_{s/4}(z)$ and $B_{s/4}(z) \subseteq  \ti \phi_i(B_{s}(x))$ for all
$x \in D_{2,5/2}(0)$ for all $z$ with $z = \ti \phi_i(x)$, for a fixed
$s>0$, $s$ independent of $i$, and $(\ti \phi_i)^*(l_i)$ is $C^k$
close to $\de$, as shown above. Let $p \in  D_{2,5/2}(0)$ and let $p =
p_1, p_2, \ldots, p_N \in D_{2-\ep(i),(5/2) + \ep(i)}$ be
the distinct points with $\ti \phi_i(p_j) = \ti \phi_i(p)$ for all $j =
1, \ldots , N$. $\theta_j:= (\ti \phi_i)_{B_s(p)}^{-1} \of \ti \phi_i:
B_{s/8}(p_j) \to B_{s}(p)$ is $C^k$ close to an element in $O(4)$, and
has $\theta_j(p_j) = p$.
$\theta_j$ is $C^k$ close to an element in $O(4)$ means $\theta_j(x) =
A_{j} \cdot x +\beta_{i,j}(x)$ for all $x
\in B_{s/8}(p_j)$, where $|\beta_{i,j}|_{C^1( B_{s/8}(p_j) ) } \leq \ep(i)$    and $A_{j} \in O(4),$ and hence $A_{j}(p_j) = p + \beta_{ij}(p_j)$
where $|\beta_{ij}(p_j) |  \leq \ep(i)$.
In particular, 
\begin{eqnarray}
\partial_r ( \theta_j((1-r)p_j) ) && = -D\theta_j( (1-r)p_j) \cdot p_j \cr
 && = -A_j \cdot p_j - D\beta_{i,j}((1-r)p_j)\cdot p_j \cr
&& = -p + v_{j}(r)
\end{eqnarray}
where $|v_j(r)| \leq \ep(i)$.
That is, using $\theta_j(p_j) = p \in  D_{2,5/2}(0)$, we see that 
$\theta_j((1-r)p_j)  \in D_{2-s,5/2}(0)$ for all $r \in [0,s/100]$.\\
That is $(1-r)p_j \in (\theta_j)^{-1}(D_{2-s,5/2}(0)) \subseteq (\ti
\phi_i)^{-1}(\ti \phi_i(D_{2-s, 5/2}))$ for all $r \in
[0,s/100]$: \\
$\theta_j((1-r_0)p_j) -p=  \theta_j((1-r_0)p_j)  -  \theta_j(p_j)  
=\int_0^{r_0} \partial_r ( \theta_j((1-r)p_j) ) dr =
-r_0p + r_0 \ti v_j,$ with $|\ti v_j| \leq \ep(i)$ implies\\
$\theta_j((1-r_0)p_j) = (1-r_0)p  +r_0\ti v_j$ and hence
$|\theta_j((1-r)p_j)| = |(1-r_0)p  +r_0\ti v_j| \leq (5/2)(1-r_0) +
\ep(i)r_0 <(5/2)$ (respectively $\geq (1-r_0)2 -r_0\ep(i) \geq 2- s$)\\

As $p \in D_{2,5/2}$ was arbitrary, we see 
$(1-r)q \in (\ti \phi_i)^{-1}(\ti \phi_i(D_{2-s,5/2}))$ for all $r \in
[0,s/100]$ for all $q \in (\ti \phi_i)^{-1}(\ti \phi_i(D_{2,5/2}))$ for
large enough $i$. Furthermore $(1-s/100)q \in D_{2-s,(5/2)-(s/200)} \subseteq D_{2-s,5/2} $ for large enough $i$.
We assume that this $i$ corresponds to $\ep$: that is $\ep=
(\frac{5}{2} ) \cdot (\frac{1}{2^{i+2}})$. Then, we have just shown
that $(1-r)\ti q \in
\phi^{-1}(\phi(D_{0,\ep}) )$ for all $r \in [0,s/100]$, for all 
$\ti q \in \phi^{-1}(\phi(D_{0,\ep}) )$, and we also
know that $(1- s/100)q \in D_{0,\ep} \subseteq
\ti V:=\phi^{-1}(\phi(D_{0,\ep}) )$.  That is, there exists a smooth map
$c:[0,s/100] \times \ti V \to \ti V $, $c(r,x) = (1-r)x\eta(x)
+(1-\eta(x))x$, where $\eta$ is a rotationally symmetric cut off
function on $D_{0,\ep}$ with, $0 \leq \eta \leq 1$,
$\eta = 1$ on $D_{\ep/2,\ep}$ and $\eta = 0$ on $D_{0,\ep/4}$, such that 
$c(0,\cdot) = Id$ and $(c(s/100,\cdot))(\ti V) \subseteq D_{0,\ep},$ und
$D_{0,\ep}$ is a simply connected space. 
Hence $\ti V$ is itself simply connected.\\


Notice also, that $E \cap B_r(x_1)$  is contained in $V$ for $r$ small
enough (***). We
explain this now.  There
is some $x \in E\cap B_r(x_1)$ with $x \in V$ by construction. Let
$\ga:[0,1] \to  {{}^l B}_{r}(x_1) \backslash
\{x_1\}$ be a smooth path of finite length in ${{}^l B}_{r}(x_1) \backslash
\{x_1\}$ with $\ga(0) = x \in E \cap  V \cap B_r(x_1)$, and $|\ga'(\cdot)|_{l}\leq C$. Let $s$ be a value for
which $\ga(t) \in E \cap V$ for all $t <s$ and  $\ga(s) \in E \cap
(V)^c$.
Lifting $\ga:[0,s) \to X$ with $\phi$, we get a curve $\ti \ga:[0,s) \to D_{0,\frac{3}{2}r}$
with $r<< \ep$. Clearly, $\ti \ga(t ) \to d \in D_{0,2r}$ as $t \upto
s$. Hence $\ga(t) = \phi (\ti \ga(t)) \to \phi(d) \in V.$ On the other
hand, $\ga(t) \to \ga(s)$ as $t \to s$. Hence $\ga(s) = \phi(d)\in
V$, per definition of $V$, which is
a contradiction. Hence $\ga$ is also a curve in $V$, that is $E \cap
B_r(x_1)$  is contained in $V$.

Also, $V \subseteq E$ if $\ep>0$ is small enough in the
definition of $V:= \phi(D_{0,\ep})$ [Explanation. As we noted
above, $V$ is connected. Furthermore,
 $V \cap E \neq \emptyset$ by definition of $V$, and $E$ is a connected component of $B_{r_0}(x_1)
\backslash \{x_1\}$, and, without loss of generality,  $\ep <<r_0$. This
means that  we have:  $ V $ is connected, $ V \subseteq B_{r_0}(x_1)
\backslash \{x_1\}$,  and 
$E$ is a connected component of $B_{r_0}(x_1)
\backslash \{x_1\}$, and $V \cap E \neq \emptyset$. Hence $ V$ is contained in $E$].
\hfill\break

We will see that for $r_0$ small enough in the above theorem, that in
fact\\
 ${^{d_X} B}_{r}(x_i) \backslash \{x_i\}
\subseteq X$ has exactly one  component for all $r\leq r_0$. This will
follow by considering the manifolds $(M,g_i,p_1)$, which
approximate a blow up $ (X,d_i:= \sqrt{c_i}d_X,x_1)$ in the sense
explained above in the Approximation Theorem, Theorem \ref{approximationthm}.

The approximations and the blow ups of $X$ itself will converge to a metric cone of the form  
 $\R^4\backslash\{0\}/
\Gamma$ for some $\Gamma $, where $\Gamma$ is a finite subgroup of $O(4)$, and the
number of elements in $\Gamma$ is bounded by $C(\si_0,\si_1) <
\infty$. That is, each blow up near a singular point consists of
exactly {\it one} cone. This will show us that
for each $i$, $B_{r_0}(x_i) \backslash \{x_i\}
\subseteq X$ has exactly one  component.
These facts are collected in the following theorem

\begin{theo}\label{orbifoldstructure}
$X$ is a {\it   $C^0$ Riemannian orbifold} in the following sense.
\begin{itemize}
\item[(i)] $X\backslash \{x_1, \ldots, x_L\}$ is a manifold, with the
  structure explained above in Lemmata \ref{manifoldstructure} and  \ref{rmanifoldstructure}.
\item[(ii)]There exists an
$r_0 >0$ small such that the following is true.
Let $x_i \in X$ be one of the singular points. Then
$B_{r}(x_i)\backslash \{x_i\}$ is connected for all $r \leq r_0$.
\item[(iii)] There exists a
  $0<\ti r \leq  r_0 $ and a
  smooth map $\phi: D_{0,\ti r} \to X \backslash \{x_1, \ldots, x_L\}$ 
  such that $\phi:\ti V \to V$ is
 a covering map, $V$ and $\ti V$ are connected sets, $\ti V$ is simply connected, and, for all $r \leq \ti r$,  we have\\
$\phi(S_r^3(0)) \subseteq {{}^{d_X} B}_{r(1 - \ep_1(r)), r(1+
    \ep_1(r))}$, and 
\begin{eqnarray}
\sup_{ D_{0,r}} |(\phi)^*l - \de|_{\de} \leq \ep_1(r)
\end{eqnarray}
where $ \ep_1(r) \leq \frac{\ti r}{100}$ is a decreasing function with $\lim_{r \downto 0}
\ep_1(r) = 0$, and $V:= \phi(D_{0,\frac{\ti r }{2} })$, $\ti V:=
  \phi^{-1}(V) \subseteq D_{0,\ti r}$,
and $S_r^3(0):= \{ x \in \R^4 \ | \ |x| =r \}$, and here $\de$ is the standard euclidean metric on
$\R^4$ or subsets thereof. 
\end{itemize}
\end{theo}
\begin{remark}
Using the facts *** mentioned at the end of the construction of $\phi$,
we see that $B_{r}(x_1) \subseteq V \cup \{x_1\}$ for all $r\leq r_0$
small enough, and hence $V \cup\{x_1\}$ is an open neighbourhood of
$x_1$ in $X$.
\end{remark}
\begin{proof}
Fix $x_1 \in \{x_1, \ldots,x_L\}$ and assume that
$B_{r_0}(x_1) \backslash \{x_1\}$, $r_0$ as above, contains more than one
component: $B_{r_0}(x_1) \backslash \{x_1\} = \cup_{i=1}^N E_i $ with
$E_i \cap E_j = \emptyset $ for all $i,j \in \{1,\ldots, N\}$, $ i
\neq j$, and $N\geq 2$.
Let $E,G$ denote two distinct components, $E:= E_1 \neq E_2 =: G$.
We use the following notation: for $p \in E \cap B_{r_0/4}(x_1)$ and $q
\in G \cap B_{r_0/4}(x_1)$, $\hat q$, $\hat p$ will denote the unique
points in $M$ with 
$f(\hat q)=
q$, $f(\hat p) = p$: these points are unique  since $p, q$ are not
singular in $X$. 

Our proof is essentially a modified version of the {\it Neck Lemma},
Lemma  1.2 of \cite{AnCh2}, of  
M. Anderson  and J. Cheeger adapted to our situation.
Note that we do not have $\Ricci $ bounded from below (as they do) for our
approximating sequences, but we do know that
they all satisfy $\int_M |\Rc|^4(g_i)d\mu_{g_i} \to 0$ as $i \to \infty$.
Hence we can use the volume estimates of P. Petersen G.-F. Wei,
\cite{PeWe}, in place of the
Bishop-Gromov volume estimates. The estimates
we require do not  appear in \cite{PeWe}, although they  follow after
making minor modifications to the proof of their estimates. We have
included the estimates and a proof thereof in Appendix \ref{PeWeapp}.

Let $(M,g_i)$ and $(X,d_i)$ be as in the Approximation Theorem, Theorem \ref{approximationthm}.
Let $E$ be as above,
and let $z_i \in E \cap  {{}^{d_i}B}_{1/4,10}(x_1)$ satisfy
$d_i(x_1,z_i) = 1$ and $v_i \in T_{z_i} E$ be
a vector such that there is a length minimising geodesic $\ga_i:[0,1] \to X$ 
on $(X,d_i)$ with $\ga_i(0) = z_i$, $\ga_i(1) = x_1$,  $\ga_i'(0) = v_i$, and $|\ga_i'(t) |_{l_i} = 1$ for all
$t \in [0,1)$ ($\ga_i'$ makes sense on $E$, since $x_i \notin E$ for
all $i \in \{x_1, \ldots, x_L\}$, and $(X \backslash \{x_1, \ldots,
x_L\},l_i)$ is  a smooth Riemannian manifold). 
We define
$ \hat z_i  := f^{-1}(z_i) \in M$, the corresponding point in $M$, and
$\hat v_i:= f^{*}v_i$, the corresponding vector in $T_{\hat z_i}
M$, where $(M,g_i)$ are as in the Approximation Theorem.(TT)

We remember, that 
$\inj(b) >i_0/1000$ for all $b \in E \cap  {}^{d_i}B_{1/4,10}(x_1)$
due to the injectivity radius estimate of Cheeger-Gromov-Taylor (
Theorem 4.3 in \cite{CGT}) and the non-inflating/non-collapsing estimates.
For any $i$,  $(T_{\hat z_i} M,g_i(\hat z_i) = g(i))$ is isometric (as
a vector space) to $(\R^4,\de)$. We will make this identification in the following, sometimes without
further mention.
%

Let $S_i \subseteq S^3_1(0)$ denote the set of vectors $\hat w$ in $
S^3_1(0)\subseteq (\R^4,\de) =(T_{\hat z_i} M,g(\hat z_i))$
(using the isometry above) which satisfy $\angle(\hat v_i, \hat w) \leq \al$ with respect to the euclidean
metric, where $\al>0$ is a small but
positive angle. We claim \hfill\break\noindent
 { \bf Claim 1}: there exists a small $\ti \ep(\al) >0$ such
that  any geodesic $\exp(g_i)_{\hat z_i}( \cdot \ m): [0,100] \to
M$ does not go through ${{}^{g_i}B}_{\ep}(p_1)$ if  $m
\in  S^3_1(0) \cap (S_i)^c$ and $0<\ep \leq  \ti \ep(\al)$ is small
enough, and $i\geq N$ large enough.

\hfill\break\noindent
{\bf Proof of Claim 1.}
\hfill\break\noindent
Let $\al >0$ be fixed.
We assume we can find $\hat w_i \in (S_i)^c \cap S^3_1(0)\subseteq \R^4 =T_{\hat
  z_i} M$ and $r_i \in (0,100]$ 
such that $g_i(\hat z_i)(\hat w_i,\hat v_i) > \al$, and
$\exp(g_i)(r_i\hat w_i) \in \boundary({ {}^{g_i} B}_{\ep}(p_1))$
but 
$\exp(g_i)(s\hat w_i) \in  ({{}^{g_i}B}_{\ep}(p_1))^c$  for all $0\leq
   s<r_i$ for $i$ arbitrarily large. We shall see, that this leads to a contradiction, if
   $\ep\leq \ti \ep(\al)$ is chosen small enough.
Let $w_i $ be the push
forward back to $X$, $w_i:=  f_* (\hat w_i)$.

For any $\de>0$, we know that
 $E \cap  {}^{d_i}B_{\de,1/\de }(x_1)$ converges, after taking a
 subsequence if necessary,  to 
$({{}^g B}_{\de,1/\de}(0),g) \subseteq \R^4 \backslash \{0\}) /
\Gamma$ in the sense of convergence given in Definition
\ref{ckconvdef}, in view of  Lemma 3.6 of \cite{Tian}: there exist diffeomorphisms $F_i:
{}^{d_i}B_{\de,1/\de }(x_1)\cap E  \to (\R^4 \backslash \{0\}) / \Gamma$
for $i$ large enough, such that $(F_i)_* l_i \to g$ in the $C^k$ sense,
where $g$ is the
Riemannian metric on $(\R^4 \backslash \{0\}) /\Gamma$ {\bf and} 
$|d_i(F_i^{-1}([x]),x_1) -|x|| \leq \ep(i) \to 0$ as $i \to \infty$ for
all $[x] \in  B_{2\de,\frac{1}{2\de} }(0) /\Gamma$, where $|x| =d([x],[0])$ 
here refers to
the  standard norm in $\R^4$ of $x$, and $[x] = \{ \Gamma_i(x) \ | \
\Gamma_i \in \Gamma \}$. In particular, the curves $F_i \of f \of \exp(\cdot \hat v_i):[0,1-
3 \de] \to (\R^4 \backslash \{0\}) /\Gamma$ and 
$F_i \of f \of \exp(\cdot \hat w_i):[0,r_i] \to (\R^4 \backslash
\{0\}) / \Gamma$, are well defined and converge smoothly to geodesic curves
$\ga: [0,1-3 \de] \to (\R^4 \backslash \{0\})/ \Gamma$ respectively
$\ti \ga: [0,r] \to (\R^4 \backslash \{0\}) /\Gamma$, $r \leq 100$, with 
$\ga(0) = \ti \ga(0) = z$ with $d(z,[0]) =1$, and
$g(\ga'(0), \ti \ga'(0)) \geq \al$ and 
$\ga(1-3\de) \in { {}^gB}_{0,3\de}(0)$ and $\ti \ga(r) \in {{}^g B}_{0,3\ep}(0)$. Here, we can choose $\de>0$
arbitrarily small. By considering the lift of the curve $\ga$ to 
to $\R^4 \backslash \{0\}$ (which must be  a straight line in
$\R^4\backslash \{0\}$), and using that $\de$ is arbitrarily small, we
see that $\ga:[0,1-3\de] \to (\R^4 \backslash \{0\})/ \Gamma$ is
arbitrarily close to the projection of a ray coming out of 
$0$ (in $\R^4$)  on $[0,1-\si]$ for $\si>0$ as small as we
like (choose $\de << \si$). Now lifting $\ti \ga$ to a curve in $\R^4
\backslash \{0\}$ (which is also a straight line in $\R^4 \backslash \{0\}$),  
and using the fact that $g(\ga'(0), \ti
\ga'(0)) \geq \al$ (which is also true for the lift), we see that $\ti
\ga(r)  \in
(B_{0,\ti \ep(\al)}(0))^c$ , for some $\ep(\al) >0$.
This leads to a contradiction to the fact that $\ti \ga(r) \in {{}^g B}_{0,3\ep}(0)$ if $\ep>0$ is chosen smaller than
say $\frac{\ti \ep(\al)}{6}$.
\hfill\break\noindent

{\bf This finishes the proof of Claim 1.}

\hfill\break\noindent

{\bf Claim 2}: For all $z \in   f^{-1}(E \cap
{}^{d_i}B_{\frac{1}{4},2}(x_1))$ and $w \in   f^{-1}(G \cap
{}^{d_i}B_{\frac{1}{4},2}(x_1))$, any length minimising geodesic
from $z$ to $w$ must go through ${{}^{g_i}B}_{\ep(i)}(p_1)$, where $\ep(i) \to
0$ as $i \to \infty$.

\hfill\break\noindent
{\bf Proof of Claim 2.}
\hfill\break\noindent
Assume we can find $i$ arbitrarily large, and  points $\hat z_i \in    f^{-1}(E \cap
{{}^{d_i}B}_{\frac{1}{4},2}(x_1))$  and $\hat w_i \in   f^{-1}(G \cap
{{}^{d_i}B}_{\frac{1}{4}, 2}(x_1))$  and a length minimising geodesic
$\hat \ga_i:[0,r_i] \to M$ (w.r.t. $g_i$),
parameterised by arclength, such that $\hat \ga_i(0) = \hat z_i$ and
$\hat \ga_i(r_i) = \hat w_i$, for  which $\hat \ga_i$ doesn't go through
${}^{d_i}B_{\si}(p_1)$ for some $\si >0$ (*) .  

Note, the Approximation Theorem, Theorem  \ref{approximationthm}, guarantees that \\
$\hat z_i , \hat w_i \in f^{-1}({{}^{d_i}B}_{\frac{1}{4},2}(x_1))
\subseteq {{}^{g_i}B}_{10}(p_1)).$
Hence $\hat \ga_i ([0,r_i]) \subseteq {{}^{g_i}B}_{40}(p_1)),$ and
hence, once again using the Approximation Theorem, 
 $f(\hat \ga_i ([0,r_i])) \subseteq {{}^{d_i}B}_{41}(x_1)$.

Let $\ga_i:[0,r_i] \to X$
be the curve $\ga_i:= f \of \hat \ga_i$. 
The Approximation Theorem guarantees that $\ga_i([0,r_i]) \subseteq
{{}^{d_i}B}_{41}(x_1)$ as we just noted (*).

There must be a first value
$r_0(i) \in [0,r_i]$ with $\ga_i(r_0(i))  =x_1$: the curve is continuous and goes from
$E$ to $G$, and so there must be some point $r_0(i)$ with $\ga(r_0(i))
\in \boundary E$. $\ga(r_0(i))$ must be equal to $x_1$, since $d_i(\boundary E
\backslash \{x_1\},x_1) \to \infty$ as $i \to \infty$ and  $\ga_i([0,r_i]) \subseteq
{{}^{d_i}B}_{41}(x_1)$.

By assumption, $\hat \ga_i(r) \notin
{{}^{g_i}B}_{\si}(p_1)$ for all $r \in [0,r_i]$.
But then, once again by the Approximation Theorem, 
$f \of \hat \ga_i ([0,r_i]) \cap {{}^{d_i}B}_{\si/2}(x_1) =
\emptyset,$ which contradicts the fact that $f \of \hat \ga_i(r_0(i)) =x_1$.

\hfill\break\noindent
{\bf End of the proof of Claim 2.}

\hfill\break\noindent

 Let $z_i  \in E \cap ( {{}^{d_i}B}_{\frac{1}{4}, 2}(x_1)),$ $\hat z_i$, $S_i$
$v_i$ $\hat v_i$ be as above (see (TT) above):
$S_i \subseteq S^3_1(0)$ denotes the set of vectors $\hat w$ in $
S^3_1(0)\subseteq \R^4 =T_{\hat z_i} M$
which satisfy $\angle(\hat v_i,\hat w) \leq \al,$
where we have identified vectors $T_{\hat z_i} M$ and vectors in
$\R^4$ using the isometry between $(T_{\hat z_i} M,g_i(0))$ and
$(\R^4,\de)$ explained above.

Let 
 $W_r:= \{ \exp(g_i)_{\hat z_i}(t \hat w) \ | \ t \in [0,r], \hat w
\in S_i\}$, $V_r:= \{\exp(g_i)_{\hat z_i}(s \hat w) \ | \ s \in [0,r], \hat w
\in S_i $ and $\exp(g_i)_{\hat z_i}(\cdot \hat w):[0,s] \to M$ is a
minimising geodesic $\}.$ $E_r$ is the set in Euclidean space which
corresponds to $W_r$: $E_r: = \{ t \be \ | \ \be \in S_i , \angle(\be, e_1) \leq
\al, t\leq r \}$
\hfill\break\noindent

{\bf Claim 3}: 
Let $\hat Z: = f^{-1}(G \cap
{{}^{d_i}B}_{1/2,1}(x_1)) $. Then
$\hat Z \subseteq V_3$, if $i$ is large enough.

\hfill\break\noindent

{\bf Proof of Claim 3}.
Let $\ga(\cdot):=\exp(g_i)_{\hat z_i}(\cdot m_i):[0,r_i] \to M$ be a
length minimising geodesic from $\hat z_i$ to a point $\hat a_i \in f^{-1}(G \cap
{{}^{d_i}B}_{1/2,1}(x_1)) $ parameterised by arclength.
Using the Approximation Theorem, Theorem \ref{approximationthm}, we must have $\hat a_i, \hat z_i \in
{{}^{g_i} B}_{1 +\ep(i)}(p_1)$, since $d_i(z_i,x_1) = 1$ and hence we must have
$r_i = d(g_i)(\hat a_i, \hat z_i)  \leq 5/2$. Assume $m_i \in (S_i)^c$. Claim 1  tells us
that the curve does not go through $B_{\ep(\al)}(p_1)$ for some 
$\ep(\al) >0$ if $i$ is large enough. But this contradicts Claim 2, if
$i$ is large enough. Hence $m_i \in S_i$ and hence $\hat Z \subseteq
V_3$ in view of the definition of these two sets.
\hfill\break\noindent

{\bf End of the proof of Claim 3}.
\hfill\break\noindent

Note for later, that $\vol(g_i)(\hat Z) \geq \theta >0$ for $i$ large
  enough, where this $\theta$ is independent of $\al,i$, and
  independent of which subsequence we
  take, in view of the fact that $(\hat Z,g_i)$ converges to $
  (B_{1,1/2}(0) /\Gamma)$ in the sense of $C^k$ manifold convergence
  given in Definition \ref{ckconvdef} (this follows from the Approximation
  Theorem \ref{approximationthm} and Lemma 3.6 of \cite{Tian}),
  and we have bounds on the number of elements of $\Gamma$, and this
  gives as a non-collapsing estimate.

The volume comparison of
Peterson/Wei shows (see Appendix \ref{PeWeapp}) that 
\begin{eqnarray}
(\frac{ \vol V_3 }{ \vol(E_3) })^{1/8} -   (\frac{ \vol V_{1/2} }{ \vol(E_{1/2}) })^{1/8} 
\leq  \frac{c}{ \al^{3/8}}(\int_M |\Rc|^4)^{1/8}
\end{eqnarray}
where we have used that the volume of $S_i\subseteq S^3_1(0)$  on  the
sphere $S^3_1(0)$ with respect to the the metric on the sphere
$d\theta$ is $\al^3c$ where $c$ is a universal constant.
Multiplying everything by $ \vol(E_3)^{1/8}  (\leq (\omega_4 3^4)^{1/8})$ we get
\begin{eqnarray}
&&(\vol V_3)^{1/8} - (\vol V_{1/2})^{1/8} \Big(  \frac{\vol(E_3) }{ \vol(E_{1/2}) }\Big)^{1/8} 
\cr
&& \ \ \leq \frac{c}{\al^{3/8}}(\vol(E_3)^{1/8})  (\int_M |\Rc|^4)^{1/8}. \label{volest}
\end{eqnarray}
The quantities $\vol(E_3) $ and $\vol(E_{1/2})$ are fixed and
positive and
depend on $\al$ (they are uniformly bounded above by the volume of
$B_3(0)$ for every $\al$). The quantity $\frac{\vol(E_3) }{
  \vol(E_{1/2}) }$ is a fixed 
positive constant which don't depend on $\al$, so we may write $c_* = (
\frac{\vol(E_3) }{ \vol(E_{1/2}) })^{1/8} $, where $c_*$ is
independent of $\al$ and $i$.
Using this in equation \eqref{volest} we get
\begin{eqnarray}
(\vol V_3)^{1/8} \leq c_*(\vol V_{1/2})^{1/8} 
 + \frac{c}{\al^{3/8}}  (\int_M |\Rc|^4)^{1/8}. \label{volest2}
\end{eqnarray}

From Claim 3 above, we see that 
$$\vol(V_3) \geq \vol(\hat Z)  \geq \theta $$ for some fixed $\theta >0$ 
since on each component the metric approaches the euclidean metric
divided out by a finite subgroup of $O(4)$.
Recall that ${{}^{d_i} B}_{\de,\frac{1}{\de}} \cap E $ converges to 
$({{}^g B}_{\de,\frac{1}{\de}},g) \subseteq (\R^4 \backslash \{0
  \})/\Gamma$ in the sense of Definition \ref{ckconvdef} using a map
  $F_i: {{}^{d_i} B}_{\de,\frac{1}{\de}} \cap E \to {{}^g B}_{\de,\frac{1}{\de}}$ , and 
${{}^{g_i} B}_{\de,\frac{1}{\de}}(p_1)$ is $\ep(i)$ $C^k$ close to
${{}^{d_i} B}_{\de,\frac{1}{\de}}$
in the sense of Definition \ref{ckconvdef}, using the map $f$, in view
of the Approximation Theorem, Theorem \ref{approximationthm}.
Since $V_r \subseteq W_r$, we have
  $\vol(V_{1/2}) \leq \vol W_{1/2} \leq c \al^3$ which goes to
zero as $\al \to 0$
[Explanation. Let $F_i\of f(\hat z_i) =: x_i$. $x_i$ is at a distance
$1 \pm \ep(i)$ away from $0$. 
We use the fact that $f(W_{1/2}) \subseteq E$ in the following without
further mention: this follows from the fact  that $f(W_{1/2})
\cap \{x_1, \ldots, x_L\} = \emptyset$, which follows from the
Approximation Theorem.

Using the fact that
$(F_i \of f)_*(g_i) \to g$ on ${{}^gB}_{\de-\ep(i), \frac{1}{\de} +
  \ep(i)}$ as $i\to \infty$, we see that $(F_i\of f)_*(S_i) \subseteq \ti S_i$, where
$\ti S_i:= \{ v \in T_{x_i} (\R^4 \backslash \{0\} / \Gamma ) \ | \
g(x_i)(n_i,v) \leq \al +\ep(i), |v|_g \in (1-\ep(i),1+\ep(i))\}$ and $n_i:= (F_i \of f)_*(\hat v_i) =
(F_i)*(v_i)$ is a vector of length almost one.
Hence, using a compactness argument, \begin{eqnarray*}
&& F_i \of f(W_{1/2})   \subseteq 
\{ \exp(x_i)(r m) \ | \ r\in [0,1/2+\de], m \in T_{x_i} ( \R^4\backslash \{0\} / \Gamma) ,
\cr 
&& \ \ \ \ \ \ \ \ \ \ \ \ \ \ \ \ \ \ \ \ \ \ \ \ \ \  \angle(m,n_i) \leq \al +\de, ||m|_g-1| \leq \ep(i) \}
\end{eqnarray*}
for all $\de >0$, for $i \geq I(\de) \in \N$
large enough,  and hence 
\begin{eqnarray*}
&& \vol W_{1/2} \cr
 && \leq (1+\de(i)) \vol \Big( \{ \exp(x_i)(r m) \ | \ r\in [0,1/2+\de(i)], m \in T_{x_i} ( \R^4\backslash
\{0\} / \Gamma) ,\cr
&& \ \ \ \ \ \ \ \ \ \ \ \ \ \  \ \ \ \ \ \ \ \ ||m|_g-1| \leq \ep(i) ,
\angle(m,n_i)  \leq \al +\de(i)\},g \Big) \cr
&& \to \vol(\ti W_{1/2}) \mbox{ as } i \to \infty \cr
&&\leq \vol(\pi^{-1}(\ti W_{1/2})),
\end{eqnarray*}
where $\ti W_{1,2} = \{ \exp(rm) \ | \ r\in [0,1/2], m \in T_x ( \R^4\backslash
\{0\} / \Gamma) ,
\angle(m,n)  \leq \al , |m|_g=1\}$ and
$\pi$ is the standard projection from
$\R^4\backslash \{0\}$ to $(\R^4\backslash \{0\})/ \Gamma$. That is\\
$\vol(V_{1/2}) \leq \vol( W_{1/2 }) \leq
\vol(\pi^{-1}(\ti W_{1/2})) +\de(i)\leq c\al^3$ since geodesics in $\R^4$ are
straight lines, and $\pi^{-1}(W_{1/2})$ is a cone of angle $\alpha$
and length $1/2$ in $\R^4 \backslash \{0\}$. End of the Explanation].

Using these two facts in \eqref{volest2}, gives us 
\begin{eqnarray}
(\theta)^{1/8} \leq (\vol V_3)^{1/8} \leq  c_*c^{1/8} \al^{3/8} + \frac{c}{\al^{3/8}}  (\int_M |\Rc|^4)^{1/8}. 
\label{volcontra2}
\end{eqnarray}
This leads to a contradiction if $\al$ is  chosen small
enough and then $i$ is chosen large enough,  since $(\int_M |\Rc|^4)^{1/2}$ goes to
zero as $i \to \infty$.

That is, there cannot be two distinct components $E $ and $G$ as
described above.
\end{proof}

\section{Extending the flow}
Since $(X,d_X)$ is a $C^0$ Riemannian orbifold, it is possible to extend the flow
past the singularity using the orbifold Ricci flow.
We have 
\begin{theo}\label{extend}
Let everything be as above.
Then there exists a smooth orbifold, $\ti X$, with finitely many orbifold
points, $v_1, \ldots,v_L$, and a smooth solution to the orbifold
Ricci flow, $(\ti X,h(t))_{t\in (0,S)}$ for some $S>0$, such that
$(\ti X,d(h(t))) \to (X,d_X)$ in the Gromov-Hausdorff sense as $t \downto
0$.
\end{theo}
\begin{proof}
Fix $x_i \in \{x_1, \ldots x_L\} \subseteq X$, where $\{x_1, \ldots
x_L\} $ are defined in Theorem \ref{manifoldstructure}. On ${{}^{d_X}B}_{\ep}(x_1)$ we have
a potentially non-smooth
orbifold structure given by the map $\phi$: the non-smoothness may
also be present without considering the Riemannian metric, as we now
explain.
As explained above, if we consider $\ti V:=
  \phi^{-1}(\phi(D_{0,\ep} ))$ and $V:= \phi(D_{0,\ep} )$,
then  $\phi|_{\ti V}: \ti V \to V$ is a covering map, $\ti V$,$V$
are connected, and $\ti V$ is simply connected, if $\ep>0$ is
small enough.

Let $x \in \ti V$ be fixed, and $G_1, \ldots, G_N:\ti V \to \ti V$ the deck
transformations, which are uniquely determined by $G_i(x) = x_i$,
where $x_1, x_2, \ldots, x_N \in \ti V $ are the distinct points with
$\phi(x_i) = \phi(x_j)$ for all $i, j \in \{1, \ldots, N\}$.

We can extend $G_1, \ldots, G_N$ to maps $ G_1, \ldots,  G_N: 
 \ti V \cup \{0\} \to  \ti V \cup\{0\}$ by defining $G_i(0) =0$ for all $i \in
\{1,\ldots, N \}$. Then the maps $G_i: \ti V \cup \{0\} \to  \ti V \cup\{0\}$  are homeomorphisms, but not necessarily smooth at
$0$. 
In this sense, the structure of the orbifold may not be smooth.
Also, as we saw above, we can extend the metric to a continuous metric
on $ \ti V \cup \{0\} $ by defining $g_{ij}(0) = \de_{ij}$,  but this extension is not
necessarily smooth.
In order to do Ricci flow of this $C^0$ orbifold, we will proceed as
follows:
{\bf Step 1.}
modify the metric $g$ and the maps $G_1, \ldots G_L:\ti V \to \ti V$ inside
$D_{0,\frac{1}{2^{i}}}$ to obtain a new metric $\ti g$ on $\ti V$ and new maps
$\ti G_1, \ldots \ti G_L: \ti V \to \ti V$ which are isometries of
$\ti V$ with
respect to $\ti g$, and such that these new objects can be
smoothly extended to $0$.
We do this in a way, so that the metric and maps are only slightly
changed (see below for details).
With the help of $\ti g$ and $\ti G_1, \ldots \ti G_L$ we will
define a new smooth Riemannian orbifold: essentially this construction
smooths out the $G_i's$ near the cone tips (the points $x_1, \ldots
,x_L \in X$) in such a way, that a group structure is preserved, and
the rest of the orbifold is not changed.  For $i \in \N , i\to \infty $, we denote the smooth Riemannian orbifolds which we 
obtain in this way  by $(X_i,d_i)$. The construction will guarantee that
$(X_i,d_i) \to (X,d)$ in the Gromov-Hausdorff sense, actually in the
Riemannian $C^0$ sense: see below.
In  {\bf Step 2}, we flow each of these spaces $(X_i,d_i)$ by Ricci
flow, and we will see, that the solution exists on a time interval
$[0,T)$ with $T>0$ being independent of $i$, and that each of the
solutions satisfies estimates, independent of $i$.
In {\bf Step 3}, we take an orbifold limit of a subsequence of the
solutions constructed in Step 2 to obtain a limiting smooth orbifold
solution to Ricci flow $(\ti X, h(t))_{t\in (0,T)}$ which satisfies
$(\ti X,d(h(t))) \to (X,d_X)$ as $t \downto 0$, in the
Gromov-Hausdorff sense.  

Now for the details.

Let $G_1, \ldots, G_N: \ti V \to \ti V$ be the deck transformations of $\phi:
\ti V \to V$, let $g:= \phi^*(l),$ and $\hat \phi: \hat V \to V$ be
$ \hat \phi (\hat x) = \phi(\frac{\hat x}{c})$, where $\hat V:= c \ti V.$
We use, in the following, the notation $\hat x = cx$.
Then $\hat \phi$ is a covering map, with deck transformations
$H_1,\ldots, H_N: \hat V \to \hat V$, 
$H_i(\hat x) = cG_i(\frac{\hat x}{c})$.
We know that $G_1, \ldots, G_N: \ti V \to \ti V$ are isometries with respect to $g$.
Let  $\hat l:=c^2l$ and $\hat g:= (\hat \phi)^*(\hat l)$. Then $\hat
g_{ij}(\hat x) = g_{ij}(x),$ and $H_1, \ldots, H_N: \hat V
\to \hat V$ are local isometries
w.r.t. $\hat g$, and hence global isometries w.r.t. $\hat g$:\\
$\hat g(\hat x)(DH_i(\hat x)(v), DH_i(\hat x)(w))  = g(x)(DG_i(x)(v),DG_i(x)(w)) = g(x)(v,w).$
Scaling with $c=2^{i+2}$ we see $\hat \phi|_{[2-14\de,4-16\de]} =
 \ti \phi_i|_{[2-14\de,4-16\de]} $, as shown above.

We go back to the construction of the map $\ti \phi_i$.
Remember that $\ti \phi_i:  D_{1/2 + \ep(i)
  \de, 4- \ep(i)} \to X \backslash \{x_1\}$ was defined by
$\ti \phi_i:= v_i \of \pi_i \of (\eta \ti \psi_i +
(1-\eta) \ti \eta_{i+1}):D_{1/2 +\ep(i), 4-\ep(i)} \to  { {}^{d_i} B}_{1/2,4}(x_1)
,$ where $\eta:\R^4 \to \R^+_0$ is a smooth cutoff function, with $\eta =
1$ on $D_{2-\de,\infty}$ and $\eta = 0$ on $D_{0,2 -2\de} $ and $ \eta \ti \psi_i +
(1-\eta) \ti \eta_{i+1}$ is $C^k$ close to the identity on $D_{1/2
  +\ep(i), 4-\ep(i)}$ (see \eqref{tiphiidef}). As we pointed out during the construction of $\ti \phi_i$, this means
that $(v_i)^{-1} \of \ti \phi_i: D_{1/2 +
  \ep(i), 4-\ep(i)} \to ({}^{g(i)} B_{1/2,4}(0),g(i))$
is $C^k$ close to $\pi_i:D_{1/2 +\ep(i), 4-\ep(i)} \to ({}^{g(i)} B_{1/2 +
  \ep(i),4-\ep(i)}(0),g(i)).$ We define
\begin{eqnarray}
&& \al_i:D_{0, 4-\ep(i)} \to  ({ {}^{g(i)} B}_{0,4}(0),g(i)) \cr
&& \al_i := \pi_i \of (\eta \ti \psi_i +
(1-\eta) Id)
\end{eqnarray}
Then $\al_i$ is $C^k$ close to   $\pi_i:D_{0,4-\ep(i)} \to ({ {}^{g(i)}
  B}_{0,4-\ep(i)}(0),g(i)),$ and equal $\pi_i$ on $D_{0,2 -2\de}.$
Hence, using the same argument we used above to show that $\phi: \ti V
\to  V$ was a covering map, and $\ti V$ is simply connected,
we have 
$\al_i: \hat Z:= (\al_i)^{-1} (  \al_i( D_{0,4-\ti \ep}) )\to Z:=
\al_i( D_{0,4-\ti \ep})$ is a
covering map, if $i$ is large enough ($\ti \ep>0$ fixed and small), $\hat Z$ is simply
connected, and $Z, \hat Z$ are connected. We also have 
$v_i \of \al_i = \ti \phi_i $ on the set 
  $D_{2-\de,4-\ti \ep}.$ In particular, $\al_i$ has the same
  number of deck transformations as $\ti \phi_i$ and hence as $\phi$ [Explanation:
  $\hat \al_i:= v_i \of \al_i: \hat Z \to v_i(Z)$ is a covering map. Choose $w \in
  D_{5/2,3}$ and let $w=w_1,w_2, \ldots ,w_{\ti N}$ be the distinct points in $\hat Z$
  with $\al_i(w_j) =\al_i(w)$ for all $j = 1, \ldots, \ti N$. 
Then $w_1, \ldots, w_{\ti N} \in D_{2,7/2}$ and furthermore
$\hat \al_i(w_j) =\hat \al_i(w)$ for all $j = 1, \ldots, \ti N$,
and hence $\ti \phi_i(w_j) = \ti\phi_i(w)$ for all $j = 1, \ldots, \ti
N$.
Hence $\ti N \leq N$. Similarly, by considering the distinct points
$w=\ti w_1, \ldots,
\ti w_N  \in D_{5/2,3}$ such that $\ti \phi_i(w) = \ti \phi_i(\ti
w_j)$ for all $j = 1, \ldots, \ti N$, we see $\ti N \geq N$.]

The Riemannian metric $l_i$ on $X$ can be pulled back to
$(B_{1/2+\ep(i),4-\ep(i)},g(i)) $ with $v_i$: $ h_i:= (v_i)^*(l_i)$. This
metric $h_i$ is $C^k$ close to $g(i)$. We interpolate between $h(i)$ and
$g(i)$ on $ (B_{1+ \de,2-4\de},g(i)) $ by  $$\be(i):= 
\hat \eta h_i + (1-\hat \eta) g(i)$$  where $\hat \eta \geq 0$ is a smooth cut-off function
on $(B_{1,4},g(i)) $ with $\hat \eta =0$ on $B_{0,1+2\de}$ and $\hat \eta
= 1$  on $(B_{1+4\de,\infty},g(i))$.
Note that $\beta(i) = h_i$ on $ D_{2-\de,4-\ti \ep}$.

Let $\hat H_1, \ldots, \hat H_N:
\hat Z \to \hat Z$ be the deck transformations of the covering map
$\al_i: \hat Z \to Z$.
These maps are isometries w.r.t. $\hat k(i):= (\al_i)^*(\be(i))$ on
$\hat Z$. Scaling these maps leads to maps $ H_k: \ti Z \to \ti Z $, $
H_k(x):= \frac{1}{2^{i+2}}\hat H_k( x 2^{i+2} )$ for $k \in \{ 1, \ldots, N \}$,
$\ti Z:= \{ \frac{x}{2^{i+2}} \ | \ x \in \hat Z \}.$
These maps are isometries w.r.t. $k(x):=k(i)(x):= 
  \hat k(i)(\hat x)$ on $\ti Z$  (see the beginning of the proof).

Note that
$k(x):= k(i)( x) = 
  \hat k(i)(\hat x)  =  (\al_i)^*(\be(i))(\hat x) =
    (\al_i)^*(h_i)(\hat x) $\\
$=(v_i \of \al_i)^*(l_i)(\hat x) =
(\ti \phi_i)^*(l_i)(\hat x) = \phi^*(l)(x) =
g(x)$ on $\ti Z\cap D_{  \frac{2-\de} {2^{i+2}   }, \frac{ 4-\ti \ep}{ 2^{i+2}  }}$.
Where we used the fact that $v_i\of \al_i$ is equal to $\ti \phi_i$  on the set 
  $D_{2-\de,4-\ti \ep}$.
Hence the Riemannian metric $\ti g$, which is defined to be the metric $k$ on  
$D_{0, \frac{1}{2^{i+1} }}$ and $g$ on $D_{ \frac{1}{2^{i+1}}, \infty}
    \cap \ti V$, is smooth and well defined. It satisfies: $\ti g(x) =
    k(x) =  \hat k(\hat x) = \de$  for $|x| \leq c(i)$ small enough.
Furthermore, $|\ti g- \de|_{C^0(\ti V)} \leq \si$ where $\si>0$, can
be made as small as we like, by choosing $\ep>0$ (in the definition of
$\ti V$) small.
\\
Using the fact that  $v_i\of \al_i$ is equal to $\ti \phi_i$  on the set 
  $D_{2-\de,4-\ti \ep}$ again, 
we see that $\hat G_1, \ldots, \hat G_N $ are the same as  
$\hat H_1, \ldots, \hat H_N $ when all of these transformations are
restricted to  $D_{ 2-\de+4\ti \ep, 4-
4\ti \ep} $ (we assume $\ti \ep << \de$).
Let $w \in D_{ 2-\de+4\ti \ep, 4-
4\ti \ep} $ and  $w=w_1, w_2, \ldots, w_N \in D_{ 2-\de+2\ti \ep, 4-
2\ti \ep} $ be the distinct points with $\ti \phi_i(w_1) = \ldots = \ti
\phi_i(w_N)$.
Let $0<s<< \min(\ti \ep,i_0/100)$ be a fixed small number and
$i$  large enough. Then we have
\begin{eqnarray}
\hat G_k|_{B_{s}(w)}
&& = ((\ti \phi_i)|_{B_{s} (w_k)})^{-1} \of (\ti \phi_i)_{B_{s} (w)} \cr
&& = (\ti \phi_i|_{ B_s(w_k)})^{-1} \of  (v_i)^{-1} \of v_i \of  (\ti
\phi_i)_{B_{s} (w)} \cr
&& =(v_i \of \ti \phi_i|_{ B_s(w_k)})^{-1}  \of (v_i \of  (\ti
\phi_i)_{B_{s} (w)} ) \cr
&& = (\al_i|_{ B_s(w_k)} )^{-1} \of (\al_i)_{B_s (w)} ) \cr
&& = \hat H_k|_{B_{s}(w)} \cr
\end{eqnarray}

This means the maps $H_i$ can be extended smoothly to all of
$ \ti V \cup \{0\}$, by defining $H_i = G_i$ on $\ti V
\cap (\ti Z)^c$ and $H_i(0) = 0$ :
call these new maps $\ti G_i$.
Note that these maps are now smooth. 
Near $0$, $k(x) = \de$, and $H_j(D_{0,s}) \subseteq D_{0,2s}$, $H_j:
\ti Z \to \ti Z$ are isometries, and hence $H_j|_{D_{0,s}} \in O(4)$
  for $s$ small enough.

Note also, that for $x \in \ti Z$, we always have $\ti G_j(x) = H_j(x) \in
\ti Z$, and for $y \in \ti V \cap (\ti Z)^c$, we have $\ti G_j(y) \in
\ti V \cap (\ti Z)^c$. To see that the last statement is true, assume
that $\ti G_j(y) \in \ti Z$ holds for some $y \in \ti V \cap (\ti
Z)^c$. Then we must have $y = (\ti G_j)^{-1}(\ti G_j)(y)) \in \ti Z$ in view of the fact that $(\ti
G_j)^{-1}(\ti Z) \subseteq \ti Z$, and this is a contradiction to the
fact that $y \in \ti V \cap (\ti
Z)^c$. 
This shows also that the $\ti G_j's$ are  diffeomorphisms, with
$(\ti G_i)|_{\ti Z} = H_i$ and
 $(\ti G_i)|_{(\ti Z)^c \cap \ti V}=  G_i|_{(\ti Z)^c \cap \ti V}$
for all $i \in \{1,
\ldots, N\}$.
In particular, $\{ \ti G_1, \ldots \ti G_N \}$ forms a subgroup of the
family of diffeomorphisms on $ \ti V  \cup \{0\}$.
The metric $\ti g$ agrees with $k$ on $\ti Z$ and 
agree with $g$ on $(\ti Z)^c \cap \ti V$
Also, the $\ti G_i $'s are  isometries on $(\ti Z,k) = (\ti Z, \ti
g)$, since $\ti G_i = H_i$ on $\ti Z$, and the $\ti G_i $'s are
isometries on $((\ti Z)^c\cap \ti V,g) = ((\ti Z)^c\cap \ti V,\ti g)$ 
since $\ti G_i = G_i$ on $(\ti Z)^c\cap \ti V$. Hence 
$\{ \ti G_1, \ldots \ti G_N \}$  are global isometries on $\ti V \cup \{ 0 \}$,
each with one fixed point, $0$.
The orbifold structure can now be defined as follows: let $\ti W:= \ti
V \cup \{0\}.$
$( \ti W,\ti G_1, \ldots, \ti G_N)$ determines one orbifold chart
$\psi: \ti W \to \ti W / \{\ti G_1, \ldots, \ti G_L\}$, where $\psi(x):= [x] = \{\ti  G_i(x) \ | \ i = 1,
\ldots, N \}.$
On  $X  \backslash ( {{}^{d_X}
  B}_{\ep/100}(x_1) \cup {{}^{d_X}
  B}_{\ep/100}(x_2) \cup \ldots \cup {{}^{d_X}
  B}_{\ep/100}(x_L)  )$, we take a covering by the inverse of $K$ manifold charts, for
example, geodesic coordinates: 
$(\theta_{\al}): {{}^l B}_{\ti \ep_0}(0)  \to  {{}^l B}_{\ti \ep_0}(y_{\al})
\subseteq  ( X \backslash ( {{}^{d_X}
  B}_{\ep/1000}(x_1) \cup {{}^{d_X}
  B}_{\ep/1000}(x_2) \cup \ldots \cup {{}^{d_X}
  B}_{\ep/1000}(x_L)  ))$, $\al \in \{1, \ldots, K\}$ (for orbifold charts the maps
always go from an open set in $\R^4$ to an open set in the orbifold). 
These are fixed for this construction and don't
depend on $i$.
Since we don't change anything on 
$ X  \backslash ( {{}^{d_X}
  B}_{\ep/1000}(x_1)\cup {{}^{d_X}
  B}_{\ep/1000}(x_2)\cup \ldots \cup {{}^{d_X}
  B}_{\ep/1000}(x_L)  )$, these charts, along with $\ti  g$, define an
Riemannian orbifold $(\hat X,\ti g)$. To be a bit more specific: 
define $\hat X = X \backslash ( {{}^{d_X}
  B}_{\ep/100}(x_1)\cup {{}^{d_X}
  B}_{\ep/100}(x_2) \cup \ldots \cup {{}^{d_X}
  B}_{\ep/100}(x_L)   ) \cup \ti W / \{\ti G_1, \ldots, \ti G_L\}$ where
we identify points $z \in  
X \backslash (  {{}^{d_X}
  B}_{\ep/100}(x_1) \cup{{}^{d_X}
  B}_{\ep/100}(x_2) \cup \ldots \cup {{}^{d_X}
  B}_{\ep/100}(x_L)   ) $
with points $[v] \in \ti W / \{\ti G_1, \ldots, \ti G_L\}$ if $z \in
\phi(\ti V)$ and $[\phi^{-1}(z)] = [v]$.
The topology is defined by saying $x_i \to x \in \hat X$ if and only
if $x,x_i \in \ti W/ \{\ti G_1, \ldots, \ti G_L\}$  for all $i \geq
N(x) \in \N$ and $x_i \to x$ in $\ti W/ \{\ti G_1, \ldots, \ti G_L\}$
, or $x,x_i \in  X \backslash ( \overline{ {{}^{d_X}
  B}_{\ep/100}(x_1)} \cup \overline{ {{}^{d_X}
  B}_{\ep/100}(x_2)} \cup \ldots \overline{ \cup {{}^{d_X}
  B}_{\ep/100}(x_L)}   )  $ for all $i \geq
N(x) \in \N$  and $x_i \to x$ in $X \backslash ( \overline{ {{}^{d_X}
  B}_{\ep/100}(x_1)} \cup \overline{ {{}^{d_X}
  B}_{\ep/100}(x_2)} \cup \ldots \overline{ \cup {{}^{d_X}
  B}_{\ep/100}(x_L)}   ).$
The charts are given above.

Call the resulting orbifold space $(X_i,\ti g_i)$.

This finishes the construction of the modified orbifolds and metrics.
\hfill\break\noindent
{\bf Step 2.}\hfill\break\noindent
Now we have a smooth orbifold and a smooth metric, so we may evolve it
with the orbifold Ricci flow, to obtain a smooth solution
$(X_i,Z_i(t))_{t\in (0,T_i)}$ to the orbifold Ricci flow: see Section
2 of \cite{HaThreeO} and Section 5 \cite{KLThree}.
The new metric $g_i(0)$ at time zero on $D_{\si}$ is $\ep$ away
from $\de$, and smooth. In particular, 
\begin{eqnarray}
|g_i(0)- g_j(0)|_{C^0(D_{\si},g_j(0))   } \leq 2\ep \mbox { for all } i,j \in \N.
\end {eqnarray}
if $\si>0$ is small enough.
One method to construct a solution to the orbifold Ricci flow is using
the so called {\it DeTurk trick} (\cite{DeT}). We can use any valid smooth
background metric $h$ to do this: taking $h= g_j(0)$ for a fixed $j \in
\N$, we have
$|g_i(0) -h|_{C^0(D_{\si},h)}  \leq \ep$ on the whole of $(X_i,g_i(0))$. Now we use the
$h$-flow in place of the Ricci-flow, that is locally the equation
looks like,
\begin{eqnarray} 
\partt g_i  = && (g_i)^{\al \be} \grad^2_{\al \be}  g_i  + \Riem(h) * (g_i)
* (g_i)^{-1} * (h)^{-1} \cr
&& + (g_i)^{-1} * (g_i)^{-1} * (  \grad g_i) *
( \grad g_i),
\end {eqnarray}
where here, $\grad = {{}^h\grad}$. Using the estimates contained in the proof of Theorem 5.2 in \cite{SimC0}, we see that the solution $
g_i (t)_{t\in [0,T_i)}$ can be extended to  $
g_i (t)_{t\in [0,S)}$  for some fixed $S = S(h) >0$ and that the
  solution satisfies
\begin{eqnarray}
|g_i(t) - h|_{C^0(X_i,h)} &\leq& 2 \ep \cr
|\gradh^k  g_i(t)|^2_{C^0(X_i,h)}  &\leq& \frac{c(K,h)}{t^k}  \mbox { for all }
0 \leq t \leq S 
\end{eqnarray}
for all $k \leq K \in \N$, as long as $t \leq S$, where $c(K,h)$ doesn't
depend on $i \in \N$.
We also have 
\begin{eqnarray}
|g_i(t) - g_i(0)|_{C^0(X,h)} &\leq& c(h,t) \leq 2\ep \mbox { for all }
0\leq t \leq S 
\end{eqnarray}
where $c(h,t) \to 0$ as $t \downto 0$, and $c(h,t)$ doesn't
depend on $i \in \N$, in view of the inequalities (5.5) and (5.6) in
\cite{SimC0} (the $\ep>0$ appearing in (5.5) and (5.6) there is arbitrary: see the proof of Theorem 5.2 in \cite{SimC0}).
In particular, 
\begin{eqnarray}
d_{GH}( (X_i,d(g_i(t)) ) ,  (X_i,d(g_i(0)) )) \leq c(t) 
\end{eqnarray}
with $c(t) \to 0$ as $t \downto 0$.
Using the smooth time dependent orbifold vector fields $V^k(\cdot,t) = -g_i(\cdot,t)^{sm}(\Gamma_{sm}^k(g_i )(\cdot,t)
-\Gamma^k_{sm}(h)(\cdot))$ and the orbifold diffeomorphisms
$\phi_t: X_i \to X_i$ with $\partt \phi_t = V $, $\phi_0 = Id$ we obtain a solution
to the orbifold Ricci flow, $Z_i(t):= \phi_t^* g_i (t)$ which satisfies
\begin{eqnarray}
&&d_{GH}(    (X_i,d(Z_i(t)) ) ,  (X_i,d(Z_i(0) ) )   ) \leq c(t) \cr
&&|\grad^j \Riem(Z_i)|(\cdot,t) \leq \frac{c(j,h)}{t^{1+ (j/2)}}\mbox{ for all } 0 \leq t
  \leq S,\label{orbieq}
\end{eqnarray} 
see for example \cite{Shi}  for details.
This finishes Step 2. In Step 2  we obtained various estimates which
are necessary for Step 3.
\hfill\break\noindent
{\bf Step 3.}\hfill\break\noindent

Using the Ricci flow orbifold compactness theorem, see \cite{Lu}  and
Section 5.3 in \cite{KLThree},
we can  now take a limit in $i \to \infty $ for $t\in(0,S)$ , and we
obtain an orbifold solution $(\ti X, Z(t))_{t \in (0,S)}$ to the Ricci flow
with
\begin{eqnarray}
&&d_{GH}( (X,d_{X}), (\ti X, d_{Z(t)})) \leq c(t) \cr
&&|\grad^j \Riem(Z)|(\cdot,t) \leq \frac{c(j,h)}{t^{1+(j/2)}} \mbox{ for all } 0 < t
  \leq S\label{orbieq2}
\end{eqnarray} 
where $c(t) \to 0$ as $t \downto 0$.  Here we used, that
$(X_i,d(Z_i(0)) ) \to (X,d_X)$ in the
Gromov-Hausdorff sense, which follows by the
construction of the spaces $(X_i,d(Z_i(0))$.
Hence we have a found a solution  $(\ti X, Z(t))_{t \in (0,S)}$ to the orbifold Ricci flow, with
initial value $(X,d_X(0))$ in the sense that \\
$d_{GH}( (X,d_{X}), (\ti X, d_{Z(t)}))  \to 0$ as $t \downto 0$.
In this sense we have extended the flow $(M,g(t))_{t \in (0,T)}$
through the singular limit $(X,d_X)$.
\end{proof}
\begin{remark}
Some of the estimates above can be obtained using Perelman's
first pseudolocality theorem and Shi's estimates. However, the estimate on the
Gromov-Hausdorff distance, which we require when showing that the
initial value of the limit solution is $(X,d_X)$, does not
immediately follow from the pseudolocality theorem. We use the estimates
given in \cite{SimC0} to show that the initial value of the solution is $(X,d_X)$.
\end{remark}
  


\begin{appendix}

\section{Cut off functions and Ye Li's result}\label{Moserapp}

Our new time dependent cut-of function $\phi$ will satisfy:
\begin{eqnarray}
&& \partt \phi \leq \lap \phi +\frac{c}{(r-r')^2} + \frac{\phi}{t} \cr
&& |\grad \phi |^2_{g(t)} \leq \frac{c}{(r-r')^2} \cr
&& \phi|_{ {{}^t B}_{r'}(y)} =e^{ct}, \cr
&& \phi|_{   ({{}^tB}_{r}(y))^c} =0,
\end{eqnarray}
for all $ t\leq S$, for some fixed universal constants, $S,c$,
wherever it is differentiable (and as long as the solution is
defined). We explain here in more detail, how to
construct this function, where here $\frac{1}{4} <r'<r \leq \frac{1}{2}$.
The function $\phi $ is constructed using a method of G. Perelman.
As explained in the paper \cite{SimSmoo}, Perelman's work shows us that
\begin{eqnarray} 
(\partt d_t -\lap d_t)(x) \geq -\frac{c_1}{\sqrt{t}}
\end{eqnarray}
 at points in space and time where  this
function is smoothly differentiable,
for some $c_1 =c_1(c_0)$ for all $t \leq S(c_0)$, if
$|\Riem| \leq \frac{c_0}{t}$ on ${{}^tB}_2(p)\cap(  ({{}^tB}_{\sqrt{t}}(p))^c)$. In our situation $c_0
= 1$.
Choose a standard cut-off function $\psi: [0,\infty) \to \R$ with
$\psi' \leq 0$, $\psi|_{[0,r']} =1,$ $\psi_{[r,\infty)} =0$ and
$|\psi'|^2\leq \frac{200}{(r-r')^2}\psi$ and $\psi'' \geq
-\frac{200}{(r-r')^2}\psi$ (see for example Lemma 3.2 in
\cite{SimSmoo}, where now
$k_0(A,B) = k_0(B) = \frac{1}{(r-r')^2}$).
Now we define $\phi(x,t) = \psi(d_t(x))$. Away from the cut locus we
have
\begin{eqnarray}
(\partt - \lap )\phi(x,t) && = \psi'(d_t(x))( (\partt - \lap )d_t(x))
-\psi''(d_t(x)) \cr
&& =  -|\psi'(d_t(x))|( (\partt - \lap )d_t(x))
-\psi''(d_t(x)) \cr
&& \leq |\psi'(d_t(x))|\frac{c_1}{\sqrt t}  + \frac{200}{(r-r')^2} \cr
&& \leq |\phi(x,t)|^{\frac 1 2} \frac{c_1 200}{|r-r'|\sqrt t} +
\frac{200}{(r-r')^2}\cr
&& \leq \frac{\phi(x,t)}{t} + \frac{c}{(r-r')^2} \label{phieqn}
\end{eqnarray}
as required.
This $\phi$ is continuous, Lipschitz in space and time, and smoothly differentiable in
space and time on $M\times [0,T) \backslash \CUT_p$, where
$\CUT_p:= \{ (x,t) \in M \times [0,T) \ |  \ x  \in \cut(g(t))(p)
\},$ where $\cut(g(t))(p) =\{ x \in M \ | \  x$ is a cut point of $p$
w.r.t to $g(t)  \}.$
$\CUT_p $  is  closed  in  $M \times [0,T),$ 
and $D:= M \times [0,T) \backslash \CUT_p$ is open in 
$M \times [0,T)$  (see
the proof of Lemma 5 of \cite{MT}). On  $M \times [0,T)$  the forward and backward
time difference quotients of $\phi$ are bounded  in the following
sense:
for all $t \in (0,T)$ there exists an $\de>0$ such that
\begin{eqnarray}
&& |\frac{\phi(\cdot,t+h) -\phi(\cdot,t)}{h}|  \leq   C, 
\label{forwback}
\end{eqnarray}
for some constant $C = C(r,r',M,g(r))_{r \in [0,S]})$ for all $h \in \R$  $ |h| \leq
\de$ ($h<0$ is allowed) for all $t\in [0,S]$.

This can be seen as follows.
Arguing as in the proof of Lemma 17.3 and Theorem 17.1 of \cite{HaForm}  respectively
Lemma 3.5 of \cite{HaFour}, we see that $d(x,t)$ is Lipschitz in time
and that
$|d(x,y,t) -d(x,y,s)|\leq c|s-t|$ for some fixed $c$ for all $x,y \in
M$ for all $t,s \in [0,S]$. 

This gives us the required estimates \eqref{forwback}.
In particular this shows us that
the time derivative of $f(t):=  \int_M \phi(x,t) l(y) d\mu_t(x) $ is
well defined for any smooth function $l:M \to \R$, as we now show. For any open
set $U \subseteq M$, we have
\begin{eqnarray}
&& \frac{f(t+h) - f(t)}{h} \cr 
&&  = \int_U l(x)\frac{\phi(x,t+h)  -  \phi(x,t)}{h}
d\mu_t(x)  \cr
&& \ \ + \int_{M\backslash U}l(x)\frac{\phi(x,t+h) -  \phi(x,t) }{h}
d\mu_t(x)\cr
&& \ \ 
+ (\frac{\int_M \phi(x,t+h)l(x)  d\mu_{t+h}(x) - \int_M
    \phi(x,t+h)l(x)  d\mu_{t}(x)}{h}) \label{yeliweak}
\end{eqnarray}

The last term in brackets converges to $\int_M
\phi(x,t)l(x,t)  \partt d\mu_t$.
Choose $V$ to be an open, star shaped set in $T_p M = \R^4$, so that 
$U:= \exp(g(t))(p)(V) \subseteq M \backslash \cut(t)(p)$, and so that
$d\mu_{g(t)}(M \backslash U ) \leq \ep$  (see for example the book
\cite{Chav} for a proof that this is possible along with Section 3 of
\cite{Wei} for a proof of the fact that the cut locus has measure zero). 
Due to the fact that $M \times (0,T) \backslash \CUT_p$ is open, we can
find a small $\de>0$, such that $U \times (t-\de,t+\de) \subseteq (M
\times (0,T)) \backslash  \CUT_p$.
This implies that 
$U  \subseteq M \backslash \cut(s)(p)$ for all $ s \in (t - \de, t +\de)$.  
Due to the continuity of volume, we may also assume that 
$d\mu_{g(t)}(M \backslash U) \leq 2\ep$ for all  $ s \in (t - \de, t
+\de<T)$ for some $\de = \de(t,U,\si_1,\si_2,T,M_0)$: this follows
from the facts that
$|\partt \int_U d\mu_{g(t)}| \leq c \int_U d\mu_{g(t)}$ and $|\partt \int_M d\mu_{g(t)}|
\leq c \int_M d\mu_{g(t)}$  for $c = c(S,T,\si_1,\si_0)$ for any open set $U \subseteq M$ as
long as $t \leq S <T$.
Using these facts, and taking the $\limsup_{h \to 0}$
respectively $\liminf_{h\to 0}$   of the above, we get
\begin{eqnarray}
&& \limsup_{h \to 0} \frac{f(t+h) - f(t)}{h} \cr 
&&  = \int_U l(x)\partt \phi(x,t)
d\mu_t(x)  \cr
&&\ \ + C_1(\ep,l,\phi,t) +
\int_M \phi(x,t) l(x) \partt d\mu_t(x)\cr
&& = \int_M l(x)\partt \phi(x,t)
d\mu_t(x)  \cr
&&\ \ + \ti C_1(\ep,l,\phi,t) +
\int_M \phi(x,t) l(x) \partt d\mu_t(x)\cr
\label{yeliweak1}
\end{eqnarray}
respectively

\begin{eqnarray}
&& \liminf_{h \to 0} \frac{f(t+h) - f(t)}{h} \cr 
&&  = \int_M l(x)\partt \phi(x,t)
d\mu_t(x)  \cr
&& \ \ + \ti C_2(\ep,l,\phi,t) +
\int_M \phi(x,t) l(x) \partt d\mu_t(x)
\label{yeliweak2}
\end{eqnarray}
where $ |\ti C_1(\ep,l,\phi)|, |\ti C_2(\ep,l,\phi)|  \to 0$ as $ \ep
\downto 0$, and we have defined $\partt \phi(\cdot,t) =0$ on
$\cut(t)(p)$ (using this definition, $\partt \phi:M \to \R$ is a bounded measurable function).
Letting $\ep \to 0$, we obtain

\begin{eqnarray}
\partt f(t) = \int_M \partt \phi(x,t) l(x)d\mu_t + \int_M
\phi(x,t)l(x)  \partt d\mu_t(x).
\end{eqnarray}
Examining the argument above, we see that for any Lipschitz (that is
$W^{1,\infty}(M)$) function
$l: M \to \R $ we have
\begin{eqnarray}
 \int_M \partt \phi(x,t) l (x) d\mu_t = \int_U \partt \phi(x,t)
 l(x) d\mu_t  +  C(\ep) \label{approximationphi}
\end{eqnarray}
where $C(\ep) \to 0 $ as $\ep \downto 0$, that is, as we choose
better open starshaped sets $U$ contained in $M \backslash \cut(t)(p)$
to approximate $M$ (we drop the dependence on $\phi$ and $l$ in the
notation, as they will be fixed for this argument).
Assume now that $l \geq 0$.
The evolution inequality  for $\phi$ , Inequality \eqref{phieqn},
combined with \eqref{approximationphi}, tells
us that 

\begin{eqnarray*}
&&  \int_M \partt \phi(x,t) l(x) d\mu_t(x)  \cr
 && \ \ \leq  \int_U \lap \phi l d\mu_t(x) + C(\ep) + \int_{ {{}^tB}_r(y)  }
 \frac{Cl}{(r-r')^2}  +\frac{C\phi l}{t} d\mu_t(x)  \cr
&& \ \ = \int_{\boundary U} l(x)\frac{\partial \phi } {\partial \nu}(x,t) 
d\si_{\boundary U,t} - \int_U g(t)(\grad \phi(x,t), \grad l(x))
d\mu_t(x)  \cr
&& \ \ \ \  + \int_{  {{}^tB}_r(y) } (\frac{Cl}{(r-r')^2} +\frac{C\phi
    l}{t}) d\mu_t(x)
+ C(\ep) \cr
&& \ \ \leq - \int_U g(t)(\grad \phi(x,t), \grad l(x)) d\mu_t(x) 
+  \int_{ {{}^tB}_r(y)  } (\frac{C\phi l}{t} + \frac{Cl}{(r-r')^2}  )d\mu_t(x) 
+C(\ep) \cr
&&\ \  \leq  - \int_M g(t)(\grad \phi(x,t), \grad l(x)) d\mu_t(x)  + \int_{ {{}^tB}_r(y)  }
(\frac{C\phi l}{t} + \frac{Cl}{(r-r')^2} ) d\mu_t(x)  + \ti C(\ep)
\end{eqnarray*}
where $\ti C(\ep) $ goes to zero as $\ep \to 0$, and hence
\begin{eqnarray}
&& \int_M \partt \phi(x,t) l(x) d\mu_t(x)  \cr
&& \ \   \leq  - \int_M g(t)(\grad \phi(x,t), \grad l(x)) d\mu_t(x)  + \int_{ {{}^tB}_r(y)  }
(\frac{C\phi l}{t} + \frac{Cl}{(r-r')^2} ) d\mu_t(x).\label{W1L}
\end{eqnarray}
Here we used 
\begin{eqnarray}
 \frac{\partial \phi } {\partial \nu}(x,t)   &&= g(t)(\grad
 \phi(x),\nu(x,t))\cr
&&= \phi'(d_t(x))g(t)(\grad
d_t(x), \nu(x,t)) \cr
&&\leq 0
\end{eqnarray}
which holds in view of the fact that $\phi' \leq 0$ and $g(t)(\grad
d_t(x), \nu(x,t))  \geq 0$ on $\boundary U$ (by construction of $U$
).

Using this $\phi$ in Lemma 1 of \cite{Li} we obtain the following estimate
\begin{eqnarray}
  \partt (\int_M \phi^2f^p) + \int_M |\grad (\phi f^{\frac p 2})|^2 && \leq
  c\int_M |\grad \phi |^2f^p \cr
  && \ \  + \int_{   {{}^t B}_r(y) } \frac{C}{(r-r')^2}f^p + 100(p +
  c)^3\mu^3 A^2t^{-1}\int_M \phi^2 f^p. \label{estLinew}
\end{eqnarray}
(we use now the notation of Ye Li: $\int_M f = \int_M f(\cdot,t) d\mu_{g(t)}$). 
The proof of this estimate is the same as that of the proof
of  Lemma 1 in \cite{Li}, accept that we must estimate the extra terms
$\int_M (\partt (\phi)^2) f^p$ coming from the time derivative of
$\phi$. In the following we use the fact that $\partt \phi, \grad \phi:
(M \times (0,T)) \backslash \CUT_p \to \R$ are smooth, and $\int_0^s \int_M |\grad \phi|^2  <
\infty$ for $s<T$. In particular, this means that we may use the Theorem of
Fubini freely for $|\grad \phi|^2$ (see Theorem 1 in Section 1.4 of
\cite{EG}), and we do so without further comment.
 
This can be done as follows (note that $p>2$ will always be assumed)
 \begin{eqnarray}
  \int_M  (\partt \phi^2) f^p 
&& =   \int_M  2 \phi \partt \phi   f^p \cr
&& \leq  \int_M -2  g(\grad \phi , \grad (\phi f^p)) + \int_{
  {{}^t B}_r(x) }\frac{C }{(r-r')^2}f^p +\frac{C\phi^2 f^p}{t} \cr 
&&  = \int_M -2f^p|\grad \phi |^2 -2pf^{p-1} \phi g(\grad \phi, \grad f) + \int_{
  {{}^t B}_r(x) }\frac{C}{(r-r')^2}f^p  +\frac{C\phi^2 f^p}{t}  \cr
&&\leq -2p\int_M f^{p-1} \phi g(\grad \phi, \grad f) +
\int_{   {{}^t B}_r(x) }\frac{C}{(r-r')^2}f^p  +\frac{C\phi^2 f^p}{t} \cr
&&\leq \frac {4} {\ga} \int |\grad \phi|^2f^p + \ga p^2 \int_M  \phi^2
f^{p-2} |\grad f|^2 + \int_{   {{}^t B}_r(x)} \frac{C}{(r-r')^2}f^p +\frac{C\phi^2 f^p}{t} . \label{yelizw}
\end{eqnarray}

We choose $\ga = \frac {1}{1000}$.
We  estimate the second term.
First note that:
\begin{eqnarray}
|\grad (\phi f^{p/2})|^2 &&= |f^{p/2}\grad \phi + \phi
\grad(f^{p/2})|^2\cr 
&&= f^p | \grad \phi|^2+2 \phi f^{p/2} g(\grad \phi, \grad(f^{p/2})) +
\phi^2 |\grad (f^{p/2})|^2\cr
&& = f^p | \grad \phi|^2+2f^{p/2} (\grad \phi, \grad(f^{p/2}\phi)) - 2f^{p} |\grad \phi|^2
+ \phi^2 |\grad (f^{p/2})|^2\cr
&& =  - f^p |\grad  \phi|^2 +2f^{p/2} g(\grad \phi, \grad(f^{p/2}\phi)) 
+ \frac{p^2}{4}f^{p-2}\phi^2 |\grad f|^2\cr
\end{eqnarray}
and hence 
\begin{eqnarray}
p^2f^{p-2}\phi^2 |\grad f|^2 && = 4|\grad (\phi f^{p/2})|^2
+4f^p |\grad  \phi|^2
-8 f^{p/2} g(\grad \phi, \grad(f^{p/2}\phi))  \cr
&& \leq 8 |\grad (\phi f^{p/2})|^2 + 8f^p |\grad  \phi|^2
\end{eqnarray}
and hence
\begin{eqnarray}
 \ga {p^2} \int_M \phi^2 f^{p-2} |\grad f|^2     \leq 8\ga \int |\grad (\phi f^{p/2})|^2 +  8 \ga \int f^p |\grad
\phi|^2. \label{yelimid}
\end{eqnarray}
Substituting this into 
\eqref{yelizw} we get

\begin{eqnarray}
  \int_M  (\partt \phi^2) f^p 
&&\leq (8\ga  + \frac {4} {\ga}) \int |\grad \phi|^2f^p + 8\ga \int |\grad
(\phi f^{p/2})|^2 \cr 
&&\ \ +   \int_{   {{}^t B}_r(x)} \frac{C}{(r-r')^2}f^p +\frac{C\phi^2 f^p}{t} . \label{yelizw2}
\end{eqnarray}
The first and last two terms are of the required form.
In the last line of the proof of Lemma 1 of \cite{Li} we choose $\ep$
(appearing in his proof) such that
$2\ep^{\frac 1 3}(p+c)A =1$ (instead of his choice of $\ep^{\frac 1
  3}(p+c)A =1$), then his estimate becomes
\begin{eqnarray}
\partt (\int\phi^2 f^p) + \int |\grad (\phi f^{\frac p 2})|^2 && \leq
  2\int |\grad \phi |^2f^p  - \frac{1}{2}\int |\grad (\phi f^{p/2})|^2 \cr 
&& \ \ \ \ \ + 10(p + c)^3\mu^3 A^2t^{-1}\int \phi^2 f^p.
\end{eqnarray}
(in \cite{Li}, this estimate occurs with a different constant: that is the term
$- \frac{1}{2}\int |\grad (\phi f^{p/2})|^2 $ doesn't appear in \cite{Li}).
We use this second last term to absorb the term $8\ga \int |\grad (\phi
f^{p/2})|^2 =  \frac  8 {1000} \int |\grad (\phi
f^{p/2})|^2$ appearing in \eqref{yelizw2}. This finishes the proof of the
claimed estimate \eqref{estLinew}.

Continuing as in the paper \cite{Li}, we get  (only adding the extra
terms we obtained in our estimates)
\begin{eqnarray}
 && \int_M \phi^2f^p d\mu_{g(t)} + \int_{\tau'}^{t}\int_M  |\grad (\phi
 f^{\frac p 2})|^2 (\cdot,s) d \mu_{g(s)} ds \cr
&& \ \  \ \leq
  c \int_{\tau}^{T} \int_M (|\grad \phi |^2 + \chi_{   {{}^t B}_r(y) }     \frac{C}{(r-r')^2})f^p  +    \Big(\hat{C}(p,\tau')  +
  \frac{1}{\tau' - \tau}  \Big) \int_{\tau}^{T} \int_M \phi^2 f^p,
\end{eqnarray}
for all $0< \tau < \tau'\leq t \leq T$,
where, using the notation of \cite{Li}, $\hat C(p,s):= \frac{100(p+c)^3c}{s} \leq
  \frac{\ti c p^3}{s}$
  where $c,C,\ti c$ are constants independent of $s$ and $p,r,r'$.
Define
\begin{eqnarray}
H(p,\tau,r) = \int_{\tau}^T\int_{B_r(x)} f^p
\end{eqnarray}
where $ \frac 1 2 \leq r \leq
 1$. Now using our estimates on $\phi$ we get (just as in \cite{Li}, except
the first constant $\hat A$ appearing on the right hand side  of the estimate below
is perhaps larger than that appearing in \cite{Li})  
\begin{eqnarray}
&&H(\frac 3 2 p,\tau',r') \leq \hat A \Big(\hat C(p,\tau') + \frac{1}{\tau' -
  \tau}  +   \frac{1}{(r-r')^2}  \Big)^{\frac 3 2} H(p,\tau,r)^{\frac
  3 2},
\end{eqnarray}
for all $0< \tau < \tau' \leq  T$, for all $\frac{1}{2} \leq r'<r
< 1$,
which is Lemma 3 of \cite{Li}.

Now Theorem 2 of \cite{Li} is also valid, up to a constant:
$f = |\Rc|$ satisfies (take $p_0 = 4$ in Theorem 2 of \cite{Li})
\begin{eqnarray}
|\Rc(x,t)| = |f(x,t)| && \leq C(1 + \frac {1}{t})^{\frac 3 4}
(\int_0^T \int_{ ^t B_1(p) }|\Rc|^4 d \mu_{g(t)})^{1/4} \cr
&&   \leq \frac {C}{ t^{\frac 3 4} }
\end{eqnarray}
for all $x \in {}^tB_{1/2}(y)$.
The proof is the same as that in \cite{Li}, where we have used here that
in our setting 
\begin{eqnarray}
\int_0^T \int_{M}   |\Rc|^4 d\mu_{g(t)} \leq  K_0 < \infty.
\end{eqnarray}
In fact, we may  assume that 
$\int_0^T \int_{M}   |\Rc|^4 d\mu_{g(t)} \leq \de^5$ is small, 
as we have scaled (and translated) the original
solution by large constants: if 
$\ti g(\ti t):= cg(\frac{\ti t}{c})$, $\ti T= cT$, $\ti t = ct$, then 
$\int_0^{\ti T} \int_M
|\ti \Rc|^4 d\mu_{\ti g(\ti t) } d\ti t =  \frac{1}{c}\int_0^T \int_M
|\Rc|^4 d\mu_{g(t)} dt \leq \frac{K_0}{c}.$

\section{Harmonic coordinates Theorem}\label{harmonicapp}

Let $(M,g)$ be a smooth, Riemannian manifold without boundary, and 
$\phi : V  \to U $ be a smooth coordinate chart on $M$, such that $U$
is compactly contained in $M$.  The metric $g$
 is given in coordinate form by $g_{ij}:U \to \R$,$ i,j \in \{1,
 \ldots,n\}$ where
 $g_{ji}(x) = g_{ij}(x):=
 g(\phi^{-1}(x))({}^{\phi}\part_i(\phi^{-1}(x)),
 {}^{\phi}\part_j(\phi^{-1}(x)))$. 
 We define the quantity $\norm D g\norm_{L^{12}}(U)$ by
 \begin{eqnarray}
\norm D g\norm_{L^{12}}(U):= (\int_{U} \sum_{i,k,r=1}^n |\partial_i
g_{kr}|^{12}  dx)^{\frac 1 {12}},
\end{eqnarray} 
where $dx$ refers to Lebesgue measure, $
{}^{\phi}\part_i(q) \in T_q M$  denotes a coordinate vector, and
  $\partial_i g_{kr} $ refers to the standard euclidean partial
  derivative in the $i$th direction of the function $g_{kr}$.
Clearly this quantity is dependent on the chosen coordinates.

\begin{defi}\label{harmonic}
Let $q \in M^4$. $B_S(q)$ will be a fixed {\it reference ball}, which is
compactly contained in $M$.  We define $r_h(p):= $ supremum over all
$r \geq 0  $ such that 
there exists a smooth ($C^{\infty}$) chart $\psi: V \to  \psi(V) = B_r(0)$ where
  $V \subseteq B_S(q)$ is open in $(M,g)$ with the following properties (here $g_{ij}$ is the metric in
these coordinates, $g_{ij}(y):=
g(\psi^{-1}(y))({}^{\phi}\part_i(\psi^{-1}(y),
{}^{\psi}\part_j(\psi^{-1}(y))$, and we use the notation used above)

\begin{itemize}
\item[(i)] $ 1/2 \de_{ij} < g_{ij} < 2 \de_{ij}$ on $B_r(0)$
\item[(ii)]  $r^{2/3} \norm D g \norm_{L^{12}(B_r(0))} <2$
\item[(iii)] $\psi: V \to B_r(0) $ is harmonic: $\lap_g \psi^k = 0
  $ on $V$ for
  all $k \in \{1,\ldots,n\} $.
\end{itemize}
\end{defi}

\begin{remark}
Notice in the definition, that $\psi:V \to B_r(0)$ refers to a {\it coordinate
chart }, that is, it must also be a smooth homeomorphism, whose inverse is $C^{\infty}$.
Note also, that nowhere in the definition do we require that the image
$\psi^{-1}(B_r(0))$ be a geodesic ball: it is simply an open set in
$M$ which is diffeomorphic to a ball.
\end{remark}
\begin{remark}
$r^{2/3} = r^{1- \frac n p}$ with $n=4$, $p=12$.
\end{remark}
\begin{remark}
We are using the definition given in \cite{Petersen}, as we will use the
notation from that paper below to state and prove the theorem that we
require. This differs from the original definition of Anderson,
\cite{And1}, where the harmonic radius is defined similarly, but using geodesic balls
in the manifold.
\end{remark}
\begin{remark}
As explained in \cite{Petersen},
$\lap_g \psi = 0$ implies that $\Gamma_{kl}^i(g)g^{kl} =0$ everywhere
in $B_r(0)$: this is crucial when it comes to the regularity theory
for the transition functions associated to these coordinates
(the regularity of metrics in harmonic coordinates was considered in
many papers, for example  \cite{JK}, \cite{SS}, \cite{DeTK} just to
name some).
Assume  
 $\ti \phi: \ti V \to B_r(0)$ and $\hat \phi: V \to B_r(0)$ are
 harmonic coordinates with $\ti V \cap V \neq \emptyset$, and $\ti g$
 and $\hat g$ refer to the metric in the coordinates $\ti \phi$
 respectively $\hat \phi$.
Then, on $B_{\ep}(v) \subseteq B_r(0)$ with $Z:= \ti \phi^{-1}(B_{\ep}(v))
\subseteq \ti V \cap V $ we have for $s:= (\hat \phi \of  (\ti \phi)^{-1})$

\begin{eqnarray}
0 &&= \lap_g (\hat \phi)^k  \cr
 &&= \lap_{\ti g} (\hat \phi \of  (\ti \phi)^{-1}))^k \cr
&& = \lap_{\ti g}s^k\cr 
&& = {\ti g}^{ab}\partial_a \partial_b  s^k,\label{transition}
\end{eqnarray}
on $B_{\ep}(v)$
where we used $\ti g^{kl}\Gamma(\ti g)^s_{kl} = 0$ on $B_r(0)$.
Notice that the derivatives of the metric do not appear in this
equation.
\end{remark}
\begin{remark}
The other important fact about harmonic coordinates, is that the Ricci
tensor  satisfies  
\begin{eqnarray}
g^{ab} \partial_a \partial_b g_{kl} = (g^{-1} * g^{-1} \partial g
* \partial g)_{kl}  -2\Ricci(g)_{kl}\label{gequation}
\end{eqnarray}
in harmonic coordinates. This star notation will be more explicitly
described below: it refers to a combination of the quantities
involved. This is a quasi-linear elliptic equation of
second order.
 \end{remark}
We state the theorem that we require in this paper here once again
\begin{theo}\label{hrtheo2}
Let  $(M^4,g)$ be a smooth manifold without boundary (not necessarily complete)  and
${B}_{3}(q) \subseteq M$ be an arbitrary ball which is compactly
contained in $M$. 
Assume that
\begin{itemize}
\item [(a)] $\int_{ B_{3}(q)  } |\Riem|^2 d\mu_{g} \leq \ep_0$
and $\int_{B_3(q)} |\Rc|^4 d\mu_{g} \leq 1$, 
\item[(b)]  $\si_0 r^4 \leq \vol(B_r(x)) \leq \si_1 r^4$ for all $r
  \leq 1 $, for all $x \in B_3(q)$,
\end{itemize}
where $\ep_0= \ep_0(\si_0,\si_1)>0$ is small enough.
Then there exists a constant $V= V(\si_0,\si_1)>0$ 
such that 
\begin{eqnarray}
r_h(g)(y) \geq V \dist_{g}(y, \boundary(B_{1}(q)))
\end{eqnarray}
for all $y \in B_1(q)$.
Here $B_3(q)$ is the {\it reference ball} used in the definition of the
harmonic radius.
\end{theo}

\begin{proof}
Proof by contradiction. 
The proof method is essentially that given in the proof of Main
Lemma 2.2 in \cite{And1} (see Remark 2.3 (ii)) using some notions from
\cite{AnCh} on the $W^{1,p}$ harmonic radius.

Assume the result is false. Then we can find
smooth Riemannian manifolds without boundary $(M_i,g(i))$ 
and balls
$B_{3}(p_i) $ which are compactly contained in $(M_i,g(i))$
such that $\int_{ B_{2}(p_i) } |\Riem|^2 \leq \frac{1}{i} \to 0$, as $i
\to \infty$,  and we can find
points $y_i \in B_i:=
B_{1}(p_i) $ such that the following
holds for all $ y \in B_i $: 
\begin{eqnarray}
\frac{r_h(g(i))(y)}{ \dist_{g(i)}(y, \boundary B_i) } \geq
\frac{r_h(g(i))(y_i)}{ \dist_{g(i)}(y_i, \boundary B_i)}  \to 0 \mbox { \ \ as \ \ } i\to \infty.
\end{eqnarray}
We define $\mu^2_i:= (r_h(g(i))(y_i))^{-2}.$ Notice that $\mu_i \to
\infty$ as $i \to \infty$, since \\$\dist_{g(i)}(y_i, \boundary B_i) \leq
1$ for all $i \in \N$.
We rescale our solution
by $\ti g(i):= \mu^2_i g(i)$
which leads to
\begin{eqnarray}
&&r_h(\ti g(i))(y_i) =1 \cr
&& \int_{\ti B_ {3\mu_i}(p_i)} |\ti \Rc|^4 du_{\ti g(i)} =  \frac{1}{\mu^2_i}  
\int_{B_{3}(p_i)} |\Rc|^4 du_{ g(i)}  \to 0 \mbox{  as  } \ i \to
\infty \cr
&& \int_{\ti B_{2\mu_i}(p_i)  } |\ti \Riem|^2 du_{\ti g(i)}  \leq \frac{1}{i}
\end{eqnarray}
Here we used the fact that the harmonic radius scales like:
$r_h(\ti g)(y) =c r_h(g)(y)$ if  $\ti g = c^2 g$,  where it is to be
understood that we use 
$\ti B:= {{}^{\ti g} B}_{cS}(q)$ as our reference ball for the definition of
harmonic radius for $\ti g$ if  $B = {{}^{g}B}_{S}(q)$ was the
reference ball for the initial definition of $r_h$. 
Notice also that the quantity $\dist_{g(i)}(x,  \boundary {{}^{g(i)} B}_{1}(p_i)) $ scales similarly:
$\dist_{\ti g(i)}(x,
  \boundary {{}^{\ti g(i)} B}_{\mu_i} (p_i)) = \mu_i
  \dist_{ g(i)}(x,
  \boundary  {{}^{g(i)} B}_{1}(p_i)) .$ 
Hence, for $B_i =  {{}^{g(i)} B}_{1}(p_i)$ and $\ti B_i = {{}^{\ti g(i)}
  B}_{\mu_i} (q)$, we have
\begin{eqnarray}
\frac{r_h(\ti g(i))(y)}{ \dist_{\ti g(i)}(y, \boundary \ti B_i) } && \geq
\frac{r_H(\ti g(i))(y_i)}{ \dist_{\ti g(i)}(y_i, \boundary \ti B_i)}
\cr
&& =   \frac{ 1}{  \dist_{\ti g(i)}(y_i, \boundary \ti B_i) } \to 0 \mbox { \ \ as \ \ } i\to \infty,
\end{eqnarray}
for all $y \in \ti B_i$.
In particular, $y_i \in \ti B_i$ satisfies $\dist_{\ti g(i)}(y_i,
\boundary \ti B_i)  \to \infty$
 as $i \to \infty$.

Furthermore,  for any fixed  $\rho >0$, we get 
\begin{eqnarray}
r_H(\ti g(i))(y) && \geq  \frac{ \dist_{\ti g(i)}(y, \boundary \ti B_i)
}{  \dist_{\ti g(i)}(y_i, \boundary \ti B_i) }\cr
&& \geq   \frac{ \dist_{\ti g(i)}(y_i, \boundary \ti B_i)- \dist_{\ti g(i)}(y,y_i)
}{  \dist_{\ti g(i)}(y_i, \boundary \ti B_i) }\cr
&& \geq \frac{ \dist_{\ti g(i)}(y_i, \boundary \ti B_i)- \rho
}{  \dist_{\ti g(i)}(y_i, \boundary \ti B_i) }\cr
&&\geq \frac 1 2 \label{testy}
\end{eqnarray}
for all $y \in \ti B_i$ with $d_{\ti g(i)}(y,y_i)  < \rho$ as long as
$i>>1 $ is large enough.
For ease of reading we remove the tildes from the $g(i)$'s in that
which follows, and simply write $g(i)$ again.

Take a maximal disjoint subset of balls $({{}^{g(i)}B}_{1/2000}(y_{i,s}))_{s =1}^N $
whose centres are  in ${{}^{g(i)}B}_L(y_i)$. By maximal, we mean: if
we take any other ball
${{}^{g(i)}B}_{1/2000}(y) $ whose centre
  is in  ${{}^{g(i)}B}_L(y_i)$, then it must intersect the
  collection of balls\\ $({{}^{g(i)}B}_{1/2000}(y_{i,s}))_{s =1}^N
    $. Then, 
$({{}^{g(i)}B}_{1/1000}(y_{i,s}))_{s =1}^N $  must cover ${{}^{g(i)}B}_L(y_i)$, and
  each of the $y_{i,s}$ satisfies $r_H(g(i))(y_{i,s}) \geq \frac 1 2$, as explained above.
Also, due to the non-inflating/non-collapsing estimates, we see that
$N$ is bounded by $N(L,\si_0,\si_1)$ .
For the same reason, the intersection number of
$({{}^{g(i)}B}_{4/100}(y_{i,s}))_{s =1}^N $ is bounded by $Z(\si_0,\si_1)$:
any subcolletion, $({{}^{g(i)}B}_{4/100}(y_{i,s_k}))_{k =1}^Z$ which intersects
must be contained in ${{}^{g(i)} B}_{1}(p)$ for any $p$ which is
the centre of some (any) ball contained in this subcollection, and
hence, 
\begin{eqnarray}
\si_1 &&\geq   \vol( {{}^{g(i)} B}_{1}(p)) \cr
&& \geq \vol(
\cup_{k=1}^Z \vol({{}^{g(i)}B}_{4/100}(y_{i,s_k}))) \cr
&& \geq \sum_{k=1}^Z \vol({{}^{g(i)}B}_{1/2000}(y_{i,s_k})) \cr
&& \geq Z \si_0c_1
\end{eqnarray}
Let
$\phi_{i,s}:= \psi_{i,s}^{-1}|_{ B_{1/100}(0) }: B_{1/100}(0) \to U_{i,s}:= \psi_{i,s}^{-1}(
B_{1/100}(0)) $ where $\psi_{i,s}:V_{i,s} \to B_{1/2}(0) \subseteq
\R^4 $ is a harmonic
coordinate chart centred at a point $y_{i,s}$ (that is
$\phi_{i,s}(y_{i,s}) = 0$). Since $g$ satisfies (i) and (ii) in the coordinates $\psi_{i,s}^{-1}:B_{1/2}(0) \to V_{i,s}$, we see
that $\phi_{i,s}:= (\psi_{i,s})^{-1}|_{B_{1/100}(0)}: B_{1/100}(0) \to U_{i,s}:= \psi_{i,s}^{-1}(
B_{1/100}(0)) $  satisfies $B_{1/400}(y_{i,s}) \subseteq
U_{i,s} \subseteq  B_{4/100}(y_{i,s}), $ and hence the intersection
number of the collection of sets $(U_{i,s})_{s=1}^N$ is bounded by
$Z$.

Using these facts, we see that Fact 1 of \cite{Petersen} is true for our
charts (in view of the (i) in the definition of Harmonic radius above), Fact 2 is true for our charts (if $i$ is large enough), Fact
3 is true for fixed $l$ if $i$ is large enough, and Fact 4. is true:
Fact 3 is used in \cite{Petersen}  to show Fact 4. We obtain Fact 4  using our
non-inflating non-collapsing arguments: there
exists a limit space $(X,d_X,p) = \lim_{GH}(M_i,d(g(i)),p_i)$ where the
limit is the pointed Gromov-Hausdorff sense (see Theorem 7.4.15 in \cite{BBI}).
In order to obtain Fact 5, we need to show that a condition like (n4)
in \cite{Petersen} is satisfied for our coordinate transition functions (compare
Section 4 of \cite{Petersen}). 
We use the equation for the transition functions, \eqref{transition}, mentioned above to
show that such a condition holds.
Let $B_{2\ep}(v) \subseteq B_{1/100}(0)$ be a small ball for which
$s_{i,r,t}:= (\phi_{i,r})^{-1} \of \phi_{i,t}:B_{2\ep }(v) \to \R^4$ is well defined on $B_{2\ep}(v)$.
Then Equation \eqref{transition}, the Schauder theory, and the fact that
$g(i)$ is bounded in $C^{0,\al}(B_{\ep}(v))$ by a constant which is
independent of $i$, show us
that $\norm s_{i,r,t}\norm_{C^{2,\al}(B_{\ep}(v))} \leq C$ (**) for some
constant independent of $i$.

Now arguing exactly as in \cite{Petersen} after the proof of Fact 4 and at 
the beginning of Fact 5, we see first that 
$X$ is a $C^0$ manifold, with inverse coordinate charts
$(\phi_r)^{-1}:V_r\to B_{1/100}(0)$, and their construction implies the
following:
if $\phi_t(B_{2\ep}(v)) \subseteq V_r \cap V_t$,then 
$ s_{i,r,t}:= \phi_{i,r}^{-1}\of \phi_{i,t}:B_{\ep}(v) \to \R^4$ converges with
respect to the $C^0$ norm to $s_{r,t}:= (\phi_r)^{-1} \of \phi_t: B_{\ep}(v) \to
\R^4$. But using the estimate (**) from above and the theorem of
Arzela-Ascoli, we see that $s_{i,r,t} \to s_{r,t}$ in
$C^{2,\be}(B_{\ep}(v),\R^4)$ (and hence in $W^{2,12}(B_{\ep}(v))$
weakly: see (9) of Section 8.2.1 (b) of \cite{Evans})  and in particular, $s_{s,t} \in
C^{2,\be}(B_{\ep}(v),\R^4)$.
That is the manifold $X$ is $C^{2,\be}$.
Now we may argue as in the rest of Fact 5 and Fact 6 of \cite{Petersen} to see that
$(X,d,p) = (X,h,p)$ where $h$ is a $C^{\al}$ metric.

Now we examine the convergence of the metrics $g(i)$ to $h$ in various
Sobolev spaces. 

Let us denote  the metric $g(i)$ in
the local coordinates given by $\phi_{i,r}$ by
$g(i)_{kl}$, and $h$ in the local coordinates given by $\phi_r$ by $h_{kl}$: $r$ is fixed for the moment.

By construction (see \cite{Petersen}), we have: $g(i)_{kl} \to h_{kl}$ in $C^{0,\al}(B_\ep(u))$ for any 
$B_{2\ep}(u) \subseteq B_{1/100}(0) \subseteq \R^4$.  Using the fact
that (ii) holds, we see that $g(i)_{kl}
\to h_{kl}$ {\bf weakly} in $W^{1,12}(B_{\ep}(u))$, after taking a
subsequence (see for example (9) of Section 8.21 (b) of \cite{Evans}).

Note, from the theorem of Rellich/Kondrachov (see for example \cite{Evans}, Theorem 1 of
Section 5.7), this tells us that $g(i) \to h$
{\it strongly} in $L^{12}(K)$.
The metric $g_{kl} = g(i)_{kl}$ satisfies
\begin{eqnarray}
 g^{ab}\partial_a\partial_bg_{kl} =  (g^{-1} * g^{-1} * \partial g
 * \partial g)_{kl} -2 \Rc(g)_{kl} \label{toodle}
\end{eqnarray}
smoothly in $B_{1/100}(0)$, since the coordinates are harmonic.
Here the $(0,2)$ tensor $(g^{-1} * g^{-1} * \partial g * \partial g)$ can be written
explicitly, but in order to make the argument more readable we use
this star notation. 
This tensor has the property that is linear in all of its terms: 
for $Z,\ti Z,W,\ti W$ symmetric and positive definite local $(2,0)$ Tensors,
and $R,\ti R, S,\ti S$ local $(0,3)$ Tensors,
we have $Z*W*R*S$ is a local $(0,2)$ Tensor, with the property that
$ (Z + \ti Z)*W*R*S = Z*W*R*S + \ti Z*W*R*S$ and
$Z*(W+ \ti W)*R*S = Z*W*R*S + Z*\ti W*R*S$ and so on.
Furthermore $|Z*W*R*S|_g \leq c|Z|_g |W|_g |R|_g |S|_g$, where $c$
depends only on $n$, $n=4$ here. 
We know from the construction that $\Ricci(g) \to 0$ in $L^4$ and the
other terms on the right hand side are bounded in $L^4$ (because
$\partial g$ is bounded in
$L^{12}$ and $g,g^{-1}$ are bounded). Hence the right hand side is bounded
in $L^4$ by a constant $c$ which doesn't depend on $i$, if $i \in \N$
is sufficiently large.
Also the terms $g^{ab}$ in front
of the first and second derivatives are
continuous, bounded and  satisfy $0< c_0|\psi|^2  \leq
g^{ij}\psi_i\psi_j \leq c_1|\psi|^2$ for $c_0,c_1$ independent of $i
\in \N$. 
Hence, using the $L^p$ theory (see for example \cite{GT} Theorem 9.11),
$|g(i)|_{W^{2,4}(K)} \leq \int_{B_{1/100}(0)} |g(i)|^4 + c \leq \ti c$ on any smooth compact subset
$K \subseteq  B_{1/100}(0)$, where $\ti c$ is a constant which is
independent of $i$ (but does depend on $K$).
In particular $g(i) \to h$ {\it strongly} in $W^{1,p}(K)$ on smooth compact
subsets $K$ of $B_{1/100}$, for any $p \in (1,\infty)$ in view of the Theorem
of Rellich/Kondrachov (see for example \cite{Evans}, Theorem 1 of
Section 5.7). 
We also have
\begin{eqnarray}
 h^{ab}\partial_a\partial_b g_{kl} = && (h^{ab}-g^{ab})\partial_a\partial_b g_{kl} +
g^{ab}\partial_a\partial_b g_{kl} \cr
= && (h^{ab}-g^{ab})\partial_a\partial_bg_{kl} 
 + (g^{-1} * g^{-1}* \partial g * \partial g)_{kl} -2 \Ricci(g)_{kl} \cr
\end{eqnarray}
Hence, for $g = g(i)$ (written in the coordinates given by
$\phi_{i,r}$) and $\ti g = g(j)$ (written in the coordinates given by
$\phi_{j,r}$) , $i,j \in \N$, we have
\begin{eqnarray}
 h^{ab} \partial_a\partial_b (g - \ti g)_{kl} 
 = && \curlL_{kl}:= (h^{ab}-g^{ab})\partial_a\partial_b g_{kl} -
 (h^{ab}-\ti g^{ab})\partial_a\partial_b \ti g_{kl} \cr
&& \ \ \
 + (g^{-1} * g^{-1}* \partial g * \partial g)_{kl} - ((\ti g)^{-1} *
 (\ti g)^{-1}* \partial \ti g * \partial \ti g)_{kl} \cr
&& + \Ricci(g)_{kl} - \Ricci(\ti g)_{kl}. \label{fiddlyg1}
\end{eqnarray}
The right hand satisfies: for all $\ep >0$ there exists an $N\in
\N$ such that $|\curlL|_{L^3(K)} \leq \ep$ if $i,j \geq N$, as we now
show. The first term in $\curlL$ is estimated by
\begin{eqnarray}
\int_K | (h^{ab}-g^{ab})\partial_a\partial_b g_{kl}|^3dx 
&&\leq  (\int_K |h-g|^{12})^{1/4} (\int_U |\partial^2 g|^4)^{3/4}
\cr
&& \leq  (\int_K |h-g|^{12})^{1/4}  C \cr
&& \to 0,
\end{eqnarray}
as $i \to \infty$, since $g= g(i)$ is bounded in $W^{2,4}$ on smooth compact subsets of $
B_{1/100}(0),$ as we showed above, and $g(i) \to h, g(i)^{-1} \to h^{-1}$ in $C^{\al}$ on
smooth compact subsets of $
B_{1/100}(0).$
The second term may be estimated in a similar fashion: 
$\int_K | (h^{ab}-\ti g^{ab})\partial_a\partial_b \ti g_{kl}|^3dx  \to
0$ as $j \to \infty$.
The third plus the fourth term of $\curlL$ converges to $0$ on $L^p(K)$ for any $p\in
(1,\infty)$ since $\partial g$ and $\partial \ti g$ converge to $h$ on $L^p(K)$ for any $p\in
(1,\infty)$  and $g(i)$ converges to $h$ in
$C^{0,\al}(K)$. The sum of the last two terms converge to $0$ in
$L^3(K)$, since they converge to $0$ in $L^4(K)$.

 Hence, we can rewrite \eqref{fiddlyg1} as 
\begin{eqnarray}
 h^{ab} \partial_a\partial_b (g(i) - g(j)) _{kl} 
 = && f(i,j)_{kl} \label{fiddlygnewer}
\end{eqnarray}
with $\int_K |f(i,j)|^3 \leq \ep(i,j)$, and $\ep(i,j) \leq \ep$ for
arbitrary $\ep >0$ if $i, j\geq N(\ep)$ is large enough.

Hence, using the $L^p$ theory again  (Theorem 9.11 in \cite{GT}), we get 
\begin{eqnarray}
 | g(i) -  g(j)|_{W^{2,3}(\ti K)} && \leq c(\int
 |\curlL|_{K}^3 + \int  |g(i) -  g(j)
 |_{K}^3)\cr
&& \leq \de(i,j)
\end{eqnarray}
on any compact subset $ \ti K \subsub K \subseteq B_{1/100}(0)$ 
where $\de: \N \times \N \to \R^+$ satisfies: for all $\ep>0$ there
exists an $N  =N(\ep) \in \N$ such that $\de(i,j) \leq \ep$ if $i,j \geq N$.
This implies $g(i)$ is a Cauchy sequence in  $W^{2,3}(\ti K)$
and hence $g(i) \to h$ strongly in $W^{2,3}(\ti K)$ and $h$ is
in $W^{2,3}(\ti K)$.
Using these facts in \eqref{toodle}, and taking a limit as $i \to
\infty$ in $L^3(K)$, we also see that
$$h^{ab} \partial_a\partial_b (h_{kl})  + (\partial h
* \partial h)_{kl} = 0 $$
must be satisfied in the $L^3$ sense, and hence in the $L^2$ sense.
In particular, we have 
\begin{eqnarray}
h^{ab} \partial_i\partial_j (h_{kl})  = P_{kl}  \label{theheqn1}
\end{eqnarray}
 in the usual $W^{1,2}$ sense 
 where $P_{kl} = - (\partial h
* \partial h)_{kl}$ is an $L^p$ function on smooth compact subsets
$K$, for any $p \in(1,\infty)$ , since $ h  \in
W^{1,p}(K)$ for any $p \in (1,\infty)$.
The $L^p$ theory tells us, that $h \in W^{2,p}$ on smooth compact subsets of
$B_{1/100}$, and hence $P \in W^{1,p}$ on smooth compact subsets of
$B_{1/100}$. The $L^2$ theory (see for example Theorem 8.10 in
\cite{GT}) then tells us, that
$h$ is in $W^{3,2}$ on smooth compact subsets. So we may differentiate
the equation \eqref{theheqn1} to get a new equation, 
\begin{eqnarray*}
h^{ab} \partial_a\partial_a (\partial_s h_{kl})  = \ti P_{skl} 
\end{eqnarray*}
where $\ti P_{skl} = \partial_s P_{kl}  - \partial_s
h^{ab} \partial_a\partial_b (h_{kl}) $
is in $L^{p}$ on any compact subset.
Continuing in this way, we obtain $h$ is $C^{\infty}$ and bounded in
$C^k$ on any smooth compact subset of $B_{1/100}$.

Let $x \in U_{r}\cap U_{t}$ in $X$. As explained above, this means
$s_{i,r,t}= (\phi_{i,r})^{-1} \of \phi_{i,t}:B_{2\ep}(v) \to \R^4$ is
well defined with $x = \phi_t(v)$ for some small $\ep>0$, and satisfies
$s_{i,r,t} \to s_{r,t} = (\phi_{r})^{-1} \of \phi_{t} $ weakly in $W^{2,12}(B_{\ep}(v))$. 
The equation satisfied by $s_i:= s_{i,r,t}$ (see \eqref{transition})
is
\begin{eqnarray}
g(i)^{ab}\partial_a \partial_b s_i = 0 \label{eqnforsi}
\end{eqnarray}
on $B_{2\ep}(v)$. 
Using the $L^p$ theory, we see that $s_i$ is bounded in
$W^{2,p}(B_{\ep/2}(v))$ for all $p \in (1,\infty)$, independently of $i$.
Hence, we have
\begin{eqnarray}
g(i)^{ab}\partial_a \partial_b (s_i -s_j) && =
-g(i)^{ab}\partial_a \partial_b s_j\cr
&& = (g(j) -g(i))^{ab} \partial_a \partial_b s_j
\end{eqnarray}
and the right hand side goes to $0$ in $L^p (B_{\ep/2}(v))$, as does
$s_i-s_j$, as $i,j \to \infty$, and hence
using the $L^p$ theory again, $s_i \to s$ in $W^{2,p}(B_{\ep/4}(v)$.
Hence, Equation \eqref{eqnforsi} holds in the limit, 
\begin{eqnarray}
h^{ab}\partial_a \partial_b s =0.
\end{eqnarray}
Using the regularity theory for elliptic equations (for example the
argument we used above to show that $h$ is smooth), we see that
$s$ is smooth.
That is
$(X,h,p)$ is a $C^{\infty}$ Riemannian manifold.
It also holds, that
$(M,g(i),p_i) \to (X,h,p)$ in the pointed Gromov-Hausdorff $W^{1,12}$  sense as
$i \to \infty$: for all $l$, there exists a map $F_{i,l}:  U_l \to
M$, such that $F_{i,l}(p) =p_i$ and $F_{i,l}: U\to M$
is a $C^{\infty}$ diffeomorphism onto its image, ${{}^h B_{l}}(p)
\subseteq U_l$,
and $F_{i,l}^*(g(i)) \to h$ in $W^{1,12}(K)$ strongly
for any compact set $K \subseteq U_l$, where $U_l$ is open in $X$.
In fact here, this convergence will be in $W^{1,p}(K)$ for any fixed
$p \in (1,\infty)$.

As soon as one knows, that the transition functions  converge in
$W^{2,p}$, and the metric converges on coordinate neighbourhoods (of
the type above) in $W^{1,p}$ to $h$, $(X,h)$ is $C^{\infty}$ and $p> n$, then it is always
possible, as explained in the introduction of the paper \cite{AnCh}, 
to construct diffeomorphisms of this type :
see \cite{Ka} or \cite{Ch1}, \cite{Ch2} for earlier works.
We give some more detail for the readers convenience. We use mainly
the reference \cite{Ka}, and the construction of the diffeomorphisms
given here is somewhat different from that presented in
\cite{Petersen}.
In the following $0<\ep <<r$ is fixed.
Let $\cup_{j=1}^N \hat U_{i,j}:= \phi_{i,j}(B_{r -50\ep})$ be a
covering of ${{}^{g_i}B}_{2l}(p_i),$
$\cup_{j=1}^N \hat U_{j}:= \phi_{j}(B_{r -50\ep})$ be a covering of
${{}^{h}B}_{2l}(p),$ with $\phi_{i,j}:U_{i,j} \to \R^4$ converging to
$\phi_{j}:U_j \to \R^4$ as described by \cite{Petersen} (see
above). Here $U_j = \phi_j(B_{r})$, and $U_{i,j} = \phi_{i,j}(B_r)$.
Define $\hat \Omega_{i}:=  \cup_{j=1}^N ( \hat U_{i,j}:= \phi_{i,j}(B_{r
  -50\ep})) $,
$  \ti \Omega_{i}:=  \cup_{j=1}^N ( \ti U_{i,j}:= \phi_{i,j}(B_{r
  -20\ep})),$
$\Omega_i^{*} := \cup_{j=1}^N ( U^{*}_{i,j}:= \phi_{j}(B_{r
  -60\ep}))$
$\hat \Omega :=  \cup_{j=1}^N ( \hat U_{j}:= \phi_{j}(B_{r
  -50\ep})) $,
$  \ti \Omega:=  \cup_{j=1}^N ( \ti U_{j}:= \phi_{j}(B_{r
  -20\ep}))$
$\Omega^{*} := \cup_{j=1}^N ( U^{*}_{j}:= \phi_{j}(B_{r
  -60\ep})),$
$\Omega^{'} := \cup_{j=1}^N \phi_j(B_{r-65\ep}(0) ).$
We assume that ${{}^{g_i}B}_{2l}(p_j) \subseteq \Omega_i^{*}
\subseteq  \hat \Omega_{i} \subseteq \ti \Omega_{i}$ and
 ${{}^{h}B}_{2l} (p) \subseteq \Omega^{'} \subseteq \Omega^{*}
\subseteq  \hat \Omega  \subseteq \ti \Omega $.

Let $\eta: [0,\infty) \to \R^+_0$ be a smooth rotationally symmetric cut-off
function, $0\leq \eta \leq 1$,
$\eta(x)=1$ for $x\in (0,r-6\ep)$,
$ \eta(x) \geq 1-C\de$ for $x\in (r-6\ep,r-5\ep)$,
$ \eta(x) = 0 $ for all $x\in [r -4\ep,\infty)$. Define, as in \cite{Ka}, the smooth functions
$\eta_{i,j}: \Omega_{i,L} \to \R^+_0$ by
$\eta_{i,j}(x) = \eta( |\phi_{i,j}^{-1}|)$, for $j =1, \ldots, N$. The
embeddings, $\theta_i : \Omega_{i} \to \R^k$, of \cite{Ka} are then
defined by 
\begin{eqnarray}
&& \theta_i(\cdot)   := \Big(\eta_{i,1}(\cdot)(\phi_{i,1})^{-1}(\cdot),
\eta_{i,2}(\cdot)(\phi_{i,2})^{-1}(\cdot), \ldots,\cr
&& \ \ \ \ \ \ \ \ \ \ \ \ \ \ \eta_{i,N}(\cdot)(\phi_{i,N})^{-1}(\cdot), \eta_{i,1}(\cdot),
\eta_{i,2}(\cdot), \ldots,
\eta_{i,N}(\cdot)\Big).
\end{eqnarray}
The maps $\theta_i: \ti \Omega_{i} \to \R^k$, $i \in \{1, \ldots,
N\}$ are embeddings, see \cite{Ka}, and locally
$\theta_i \of \phi_{i,1}$ (for example) is a graph: 
$\theta_i \of \phi_{i,1}|_{B_{r-20\ep}(0) }:  B_{r-20\ep}(0)
\to \R^k$ is given by
$$ \theta_i \of \phi_{i,1} (x) = (x,f_{i,2}(x)F_{i,2}(x), \ldots, f_{i,N}(x)F_{i,N}(x),
f_{i,1}(x), \ldots, f_{i,N}(x)),$$ where
$F_{i,j}:= (\phi_{i,j})^{-1} \of \phi_{i,1}$, $f_{i,j}( \cdot)
:=\eta(|F_{i,j}(\cdot)|) $.
That is $ \theta_i \of \phi_{i,1}(x)    =(x,u_{i,1}(x))$
and the maps $u_{i,1}: B_{r-20\ep}(0) \to \R^{k-n}$ are bounded in
$W^{2,p}$ and converge in\\$W^{2,p}( B_{r-20\ep}(0),\R^k)$ to the
smooth function
$u_{1}: B_{r-20\ep}(0) \to \R^{k-n},$ where $u_{1}:= (f_2F_2,f_3F_3,
\cdots, f_NF_n, f_1, \ldots,f_N)$ , and 
$F_{j}:= (\phi_{j})^{-1} \of \phi_{1}$, $f_{j}( \cdot)
:=\eta(|F_{j}(\cdot)|) $.
Defining $\theta :  \ti \Omega  \to \R^k$ similarly to $\theta_i$, 
\begin{eqnarray}
&& \theta(\cdot)   := \Big(\eta_{1}(\cdot)(\phi_{1})^{-1}(\cdot),
\eta_{2}(\cdot)(\phi_{2})^{-1}(\cdot), \ldots,\cr
&& \ \ \ \ \ \ \ \ \ \ \ \ \ \ \eta_{N}(\cdot)(\phi_{N})^{-1}(\cdot), \eta_{1}(\cdot),
\eta_{2}(\cdot), \ldots,
\eta_{N}(\cdot)\Big).
\end{eqnarray}
where $\eta_j: \ti \Omega \to \R^+_0$ are defined 
$\eta_{j}(x) = \eta(| (\phi_{j})^{-1}(x)|)$, for $j =1, \ldots, N$, we
see that $\theta$ is also an embedding, and hence
$\ti M:= \cup_{j=1}^N(   \{ t_j(x,u_j(x)) \ | \ x   \in B_{r-20\ep}(0)
\}  = \theta(\ti \Omega)$ and $\hat M := \cup_{j=1}^N( \{  t_j(x,u_j(x)) \ | \ x   \in B_{r-50\ep}(0)
\}  = \theta(\hat \Omega)$  are  $C^{\infty}$ embedded submanifolds
of $\R^k$, where $t_j: \R^k \to \R^k$, is the function (which swaps
position of coordinates) defined by\\
$t_j(x,m_1, \ldots, m_{N-1}, y_1, \ldots, y_N) $\\
$:= (m_j,m_1,m_2,
\ldots,m_{j-1}, x,m_{j+1}, \ldots,m_{N-1},y_1, \ldots, y_N)$,\\
for $m_i
\in \R^4,$, $i \in \{1, \ldots,N-1\}$,  $y_k \in \R,$ $k \in \{1,
\ldots,N\}$, where $j \in \{1, \ldots,N\}$.
We define the $C^{\infty}$ submanifolds, $\hat M_i := \theta_i(\ti
\Omega_i)$, $M^*_i := \theta_i(
\Omega^*_i)$, $M^* := \theta(
\Omega^*)$, 
$ \ti M_i :=  \theta_i(\ti \Omega_i)$ analogously.  
To see that these 
maps are embeddings one shows the following.
Let $z:= \phi_1(x) \in M$, $x \in  B_{r-20\ep}(0)$(for example), and $(x,u_1(x)) = (x,y) = 
\theta \of \phi_1(x).$ 
We claim $(B_s(x) \times B_{\de}(y) ) \cap \ti M =  \{ (z,
u_{1}(z))  \ | \ z \in B_s(x) \}$ for $s,\de$ small enough.
We know $|Du_{1}(\cdot)| \leq C$ and $|Du(\cdot)| \leq C$ on $B_{r-4\ep}(0)$ for a constant $C$.
Hence, $\{ (z,
u_{1}(z))  \ | \ z \in B_s(x) \} \subseteq    (B_s(x) \times
B_{\de}(y) )\cap \ti M$ for $\de,s $ small $\de =c(c)Cs$
(we choose $s$ small, so that $B_s(x) \subseteq B_{r-20\ep}(0)$).
 Now let
$m$ be arbitrary with $m \in ( B_s(x) \times B_{\de}(y) ) \cap
\ti M,$ and $\ti m:= \theta^{-1}(m)$.
Then $|\theta(\ti m) -  \theta(\phi_{1}(x))| =
|\theta(\ti m) -(x,u_{1}(x))| \leq C \de $ and hence, using the
definition of $\theta$, $|\eta_{1}(\ti m) -1| \leq C\de$ which tells us that
$\ti m \in  U_1 $, and $|v:= (\phi_{1})^{-1}(\ti m)| \leq r-5\ep$. If $v
\in B_{r-6\ep}(0)$ then  $m= \theta(\ti m) = \theta \of
\phi_{1}(v) = (v,u_{1}(v))$ that is $m  \in B_s(x) \times
B_{\de}(y) )\cap \ti M$, can be written $m= (v,u_1(v))$ with $v \in
B_s(x)$ and $u_1(v) \in  B_{\de}(y) $ as required.

If $|v| \in ( r-6\ep,r-5\ep)$, then we know that
$\theta(\ti m) = \theta\of \phi_{1}(v) =
(\eta(|v|)v,u_{1}(v)) \in B_s(x) \times B_{\de}(y)$ which
leads to a contradiction, since $|(\eta(|v|)v| \geq
(1-\de)(r-6\ep) \geq (r-7\ep)$ ($\de << \ep$) and hence
$\eta(|v|)v  \notin B_s(x)$, since $x \in B_{r-20\ep}(0)$ ($s,\de << \ep$).

%
%
We construct a diffeomorphism $ \hat F_{i}: \ti \Omega_i \to \ti M$ for which
$\hat M  \subseteq \hat F_{i}( \Omega_i )$, 
and $F_i:= \theta^{-1} \of \hat F_i $ will have the
desired properties.
Since $\ti M$ is a $C^{\infty}$ embedded submainfold of $\R^k$, there
exists a neighbourhood  $Z$ of the zero section in the 
normal bundle $\ti M^{\perp}$ of $\ti M$ ($\ti M$ is open, without
boundary, and not necessarily complete) in $\R^k$ such that 
$\exp_{\perp}|_Z: Z \to O := \exp_{\perp}(Z) \subseteq \R^k$ is a
diffeomorphism onto its image, and
$\ti M \subseteq    O$ by definition (see for example Section 7,
Proposition 26 in \cite{On} for the existence of $Z$ and the
definition of $\exp_{\perp}$ the exponential map from the normal bundle). 
Using the fact that the closure
of $\hat M$ is contained in $\ti M$, we see that the closure of $\hat M$ is a
compact subset of $O$. That is $B_{\si}(\hat M) \subseteq O$ for
$\si >0$ sufficiently small.

Hence the natural projection map $\pi: B_{\si}(\hat M) \to \ti M$ is well defined and
smooth, $\pi(x):= (\exp_{\perp}|_Z)^{-1}(x)$ for all $x \in
B_{\si}(\hat M) \subseteq O$ (see for example Section 7, Proposition 26 in \cite{On} for the
existence of $Z$), and $\pi(p) = p$ for all $p \in \hat
M$.
Furthermore, by choosing $\si$ smaller in necessary, we can assume
that $|\pi(y) -y| \leq \ep$ for all $y \in  B_{\si}(\hat M) $.
We show that $\hat F_i := \pi \of \theta_i: \hat \Omega_i \to \ti M$
is well defined and a diffeomorphism onto its image, for large enough $i$.
Using the fact that $u_{i,j}:B_{r-\ep}(0) \to u_{j}$ in  $W^{2,p}(B_{r-\ep}(0))
\cap C^{1,\al}(B_{r-\ep}(0))$ as $i \to \infty$ , we see that
$d_{H}(\hat M_i, \hat M) \to 0$  as $i \to \infty:$ 
$\hat M_i \subseteq B_{\ep}(\hat M)$ and $\hat M \subseteq
B_{\ep}(\hat M_i).$ Hence, for $m \in \hat U_{i,1} \subseteq
\hat \Omega_i$, we have $\theta_i(m) = \theta_i \of \phi_{i,1}(x) =
(x,u_{i,1}(x)) \in \hat M_i \subseteq B_{\ep}(\hat M)$ and hence $\pi \of
\theta(m)$ is well defined for such $m$. 
Similarly, 
$\pi \of
\theta_i(m)$ is well defined for all $m \in  \hat \Omega_i $.
Let $m \in \hat \Omega$. Without loss of generality, $m \in \hat U_1,$ $x
:= (\phi_1)^{-1}(m),$ $\theta(m) = (x,u_1(x)) = (x,y) \in
\theta(\hat U_1).$
$\theta(\hat U_1)$ is an open set in $\hat M$, so we can find a small $s$
such that $B_s(x) \times B_{cs}(y) \cap \hat M \subseteq  \theta(\hat
U_1)$, and $\pi(B_s(x) \times B_{cs}(y)) \subseteq \theta(\hat U_1)$. 
$(\cdot, u_{i,1}(\cdot)) =\theta_i \of \phi_{i,1}(\cdot) \to \theta
\of \phi_{1}(\cdot) = (\cdot,u_{1}(\cdot))$ in
$W^{2,p}(B_{r-\ep}(0),\R^k) \cap C^{1,\al}(B_{r-\ep}(0),\R^k ) $ as $i\to \infty$.
Hence $\theta_i \of \phi_{i,1}(v) \in B_s(x) \times B_{cs}(y)
\subseteq \theta(\hat
U_1)$ for
all $i \geq N$ large enough, for all $v \in B_{\ti s}(x)$, if  $\ti s$
is small enough, and hence 
$(\phi_1)^{-1} \of \theta^{-1} \of \pi \of \theta_i \of
\phi_{i,1}: B_{\ti s}(x) \to B_{r-4\ep}(0)$ is well defined for $i$
large enough. By construction, and the Theorem of Vitali (the $L^p$ version), we have 
\begin{eqnarray}
 &&(\phi_1)^{-1} \of (\theta)^{-1} \of \pi \of \theta_i \of \phi_{i,1} \cr 
&& = (\theta \of \phi_1)^{-1}  \of \pi \of ( \cdot,u_{i,1}(\cdot))\cr
&&  \to   (\theta \of \phi_1)^{-1} \of     \pi \of (
\cdot,u_{1}(\cdot)) \cr
&& =   (\theta \of \phi_1)^{-1} \of  \theta \of \phi_1  \cr
&& = Id
\end{eqnarray}
in $W^{2,p}(B_{\ti s}(x)),\R^4)$. 
Hence, $(\phi_1)^{-1} \of (\theta)^{-1} \of \pi \of \theta_i \of
\phi_{i,1} : B_{\ti s}(x) \to \R^4 $ is a diffeomorphism onto
its image for large enough $i$.
Hence $\theta^{-1} \of \pi \of \theta_i : \Omega_i^* \to X$ is a local
diffeomorphism, and $\theta^{-1} \of \pi \of \theta_i (U_{i,1}^*)
\subseteq \hat U_{i,1}$ for $i$ large enough, and 
\begin{eqnarray}
&& \phi_1 \of \theta^{-1} \of \pi \of \theta_i \of \phi_{i,1}   :  B_{r-55\ep}(0) \to
B_{r-50\ep}(0) \mbox{ is well defined, and} \cr
&& \phi_1 \of \theta^{-1} \of \pi \of \theta_i \of \phi_{i,1}  \to Id \mbox{ in }
W^{2,p}(B_{r-55\ep}(0),\R^4) \cr
&& \ \ \ \ \ \ \ \ \mbox{ as } i \to \infty \label{convId} 
\end{eqnarray}
Hence $\theta^{-1} \of \pi \of \theta_i : \Omega_i^*
\to X$ is a diffeomorphism onto its image for large enough $i$, and
$\Omega^{'} \subseteq \theta^{-1} \of \pi \of \theta_i (\Omega_i^*)$ for
large enough $i$, as we now show. Assume
$\theta^{-1} \of \pi \of \theta_i (m) = \theta^{-1} \of \pi \of
\theta_i (n)$ for $m,n \in \Omega_i^*$. Without loss of generality, $m \in U_{i,1}^*$
and hence  $\theta^{-1} \of \pi \of \theta_i (m) \in \hat U_1$ and
hence $(\phi_{1})^{-1} \of \theta^{-1} \of \pi \of \theta_i (m) = (\phi_1)^{-1} \of
\theta^{-1} \of \pi \of
\theta_i (n)$, and hence $(\phi_{1})^{-1} \of \theta^{-1} \of \pi \of
\theta_i \of \phi_{i,1}  (\ti x) = (\phi_1)^{-1} \of
\theta^{-1} \of \pi \of
\theta_i \of \phi_{i,1} (\ti z)$, where $\phi_{i,1}(\ti x) = m$ and
$\phi_{i,1}(\ti z) = n$, where  $\ti x,\ti z \in B_{r-60\ep}(0)$. 
But this contradicts \eqref{convId} for $i$ large enough.
That is $\theta^{-1} \of \pi \of \theta_i : \Omega_i^* \to X$ is a
diffeomorphism, and
$ B_{2l}(p) \subseteq \cup_{j=1}^N \phi_1(B_{r-65\ep}(0) ) = \Omega^{'} \subseteq \theta^{-1} \of \pi
\of \theta_i (\Omega_i^* )$ as required.

Also,
\begin{eqnarray}
 && ((\phi_1)^{-1} \of F_i  \of \phi_{i,1})_*(g(i)) \cr
&& = (\theta \of \phi_1)^{-1}  \of \pi \of ( \cdot,u_{i,1}(\cdot))_*g(i)\cr
&&  \to   h
\end{eqnarray}
in view of \eqref{convId} and the Theorem of Vitali, as required
(note: the inverse of $(\phi_1)^{-1} \of (\theta)^{-1} \of \pi \of \theta_i \of
\phi_{i,1} $ also converges to the identity in
$W^{2,p}(B_{r-56\ep}(0),\R^4) $ in view of \eqref{convId} and for
example Cramer's law).

The condition $\int_X |\Riem(h)|^2 = 0$ must hold for the limiting
space , as we now show.
For fixed $r$,
\begin{eqnarray}
\int_{U_{r}} |\Riem(h)|^2 d \mu_h  && = \int_{U_{r}}  |      h*(h)^{-1}
*D^2h + h*h^{-1} * Dh * Dh  |^2 d\mu_h \cr
&& = \lim_{i\to \infty}   \int_{B_{1/100}}  |
g_{i,r}*(g_{i,r})^{-1}*D^2(g_{i,\al})       \cr
&& \ \ +
 g_{i,r}*(g_{i,r})^{-1} *D g_{i,\al} *Dg_{i,|al}) |^2d\mu_{g(i)}  \cr
&&  =  \lim_{i\to \infty}  \int_{ U_{i,r } } |\Riem(g(i))|^2 d\mu_{g(i)} \cr
&&\leq \frac{1}{i}
\end{eqnarray}
since the metrics converge in $W^{2,3}$ locally.
Here $*D^2 h$ $(*D^2(g(i,\al))) $ and so on are combinations of the second
partial derivatives of $h$ and can be explicitly written down, and
satisfies $|*D^2 h| \leq c(n)|D^2 h|^2$ pointwise where the second
derivative is defined.
 
Hence  $\int_{U_{r}} |\Riem(h)|^2 dvol_h =0$. $r$ was arbitrary, and
$h$ is smooth implies that $|\Riem(h)|=0$ everywhere.

Hence, $(X,h)$ is flat, and hence
must be the standard Euclidean space, since $(X,h)$ has euclidean growth, $\vol(B_r(x)) \geq \si_0r^4$
for all $r>0$ (from the non-collapsing estimate). 
This leads to a contradiction, using the same argument given in the
proof of Main Lemma 2.2 in \cite{And1}.

Here, for the readers convenience, we explain the argument which leads to a contradiction.
Push the metrics $g(i)$ forward to $X=\R^4$ with the $C^{\infty}$
diffeomorphisms $F_{i,l}:{{}^hB}_l(p_i) \subseteq  (M,g(i)) \to X =
\R^4$ just constructed. Call these
metrics  $\hat g(i)$. We know, that $\hat g(i) \to h = \de$ on
$B_l(0) \subseteq \R^4$, $l$ is fixed but large.
Solve $\lap_{\hat g(i)}  \phi(i)^k = 0$, on $B_{R}(0)$, with the
boundary conditions  $\phi(i)|_{\boundary B_{R(0)}} = \Id$. 
Then $\lap_{\hat g(i)}(\phi(i) -\Id)^k = \Gamma(\hat g(i))_{kl}^m(\hat
g(i))^{kl} \to 0$ in $L^p(B_{R}(0))$, and $(\phi(i) - Id)(\cdot)=0$ on $\boundary
B_{R}(0)$.
Using the $L^p$ theory (Lemma 9.17 of \cite{GT} for example),
with the fact that $\phi(i) -Id =0$ on the
boundary, we see that 
\begin{eqnarray}
| \phi(i) -Id|_{W^{2,p}(B_{R}(0))} \to 0, \mbox{ as } \ i \to \infty.
\end{eqnarray}
for $i$ large enough, since 
$\hat g(i) \to h$ in $W^{1,p}(B_R(0))$ strongly.
Hence, $\phi(i): B_{R/2}(0) \to V_i:= \phi(i)(B_{R/2}(0)) \subseteq
\R^4$
is a diffeomorphism for $i$ sufficiently large, and
$\phi(i) \to Id$ in $W^{1,p}(B_{R/2}(0)) \cap C^{1,\al}(B_{R/2}(0)),$
as $i \to \infty$. 
The inverse $\psi(i):V_i \to  B_{R/2}(0)$ converges to the identity 
in $W^{1,p}(B_{R/4}(0)) \cap C^{1,\al}(B_{R/4}(0)),$ ($B_{R/4}(0)
\subseteq V_i$ for $i$ large enough), once again in view of Vitali's
Theorem.
Hence, $\psi(i)_*(\hat g)(i) =: \ti g_i$ converges to $\de$, once again in view of Vitali's
Theorem.
But then : $v_i:= (F_i)^{-1} \of \psi(i): B_{R/4}(0) \to W_i \subseteq
M_i$ is a diffeomorphism, and $w_i := (v_i)^{-1}:W_i \to B_{R/4}(0) $
satisfies
\begin{eqnarray}
\lap_{g(i)} w_i(x) &&= \lap_{g(i)} (\phi(i) \of F_i)(x)  \cr
&& = \lap_{\hat g(i)} (\phi(i) )( F_i(x)) \cr
&& = 0,
\end{eqnarray}
and $(w_i)_*(g(i))$ is in $C^{1,\al}(B_{R/4}(0)) \cap
W^{1,p}(B_{R/4}(0) )$ as close as we like to $\de$. That is the
harmonic radius at $x$ is larger than $R/4$. This contradicts the fact
that, $r_h(y_i)(g(i)) =1$.

This finishes the proof. 

\end{proof}

\section{Volume estimates of P.Peterson and G-F.Wei}\label{PeWeapp}
Let $(M^n,g)$ be a smooth complete Riemannian manifold without
boundary satisfying and $p > n/2$ and $q \in M$ is fixed.
We  choose
coordinates so that $g$ satisfies $g(q)=\de$ on $T_q M = \R^n$, $\de$ the
Euclidean metric on $T_q M = \R^n$. 
Let $S$ be an open set contained in the sphere $S_q \subseteq T_q M$  with $d\theta(S) = \mu >0$, $d\theta  $ the standard
 measure on $\Sp^{n-1}_1(0)$, where $S_q M= \Sp^{n-1}_1(0)$ is the sphere of
radius one, $S_q M:= \{ v \in T_p M $ such that $|v| =1 \}$. So $S \subseteq \Sp^{n-1}_1(0) \subseteq
\R^n$. Let  $W_r  = \{  \exp(s v) \ | \ v \in S, s \leq r  \}$, and $V_r =
\{ \exp(sv) \ | \ v \in S, s \leq r$ and $\exp(\cdot  v):[0,s] \to M$
is length minimising $\}.$
We consider the corresponding set $E_r$ in $\R^n$: $E_r = \{ s v \ | \
s
\leq r, v \in S \}$.
The estimates of P.Peterson and G-F.Wei in this setup
that we require are as follows. 
\begin{eqnarray}
&& \Big(\frac{ \vol V_R(x) }{ \vol(E_R) }\Big)^{1/2p} -   \Big(\frac{ \vol V_{r}(x) }{ \vol(E_{r}) }\Big)^{1/2p} \cr
&& \ \ \ \leq  c(n,R,p)\frac{1}{\mu^{1/(2p)} }(\int_{B_R(p)} |\Rc|^p)^{1/(2p)}
\end{eqnarray}
where $\mu = d\theta(S) >0$ and $r< R$.
\begin{remark}
Compare with Theorem 2.1 in \cite{TZ}.
\end{remark}

\begin{proof}
We argue as in \cite{PeWe}, but we replace
their $S$ (the sphere in $\R^n$) by our set $S$, and we replace their $\lambda$ 
by $ \lambda = 0$.
We denote with $\hat C_p$ the cut locus, $\hat C_p:= \{ \exp(sv) \ | \ v \in
\Sp^{n-1} $ and $\exp|_{[0,s]}(\cdot v): [0,s] \to M$ is distance
minimising, but $\exp|_{[0,s+ \ep]}(\cdot v): [0,s+\ep] \to M$ is
{\bf  not} distance minimising for all $\ep>0\}$.

Let $\hat D_p \subseteq M^n$ be the set  $\hat D_p = \{ \exp(rv) \ | \ v \in
\Sp^{n-1},  r<
c(v) \}.$

Here, $c:\Sp^{n-1} \to \R^+$ is the function which tells
us how far we have to travel along a geodesic, pointing in a given
direction, before we hit the cut
locus:
$c(v):=s$ where $\exp|_{[0,s]}(\cdot v): [0,s] \to M$ is distance
minimising, but $\exp|_{[0,s+ \ep]}(\cdot v): [0,s+\ep] \to M$ is
{\bf not} distance minimising for all $\ep>0$.
$D_p,C_p$ will denote the corresponding sets in $\R^n$: 
$C_P:= \{ c(v)v  \ |  \ v \in \Sp^{n-1}\}$, $D_p:= \{ rv  \  | \ v \in \Sp^{n-1},
r < c(v) \}$.  The set $D_p$ is star shaped, and the function $c$ is
continuous: see for example the book \cite{Chav} for a proof of these facts and other
related facts.
Define $V_r:= \{\exp(s v) \ | \ v \in S, s \leq r $, and
$\exp(\cdot v)|_{[0,s]}:[0,s] \to M$ is distance minimising$\}$. That is
$V_r$ is the set of
points obtained by going along a distance minimising geodesic in the direction $v$ for a
distance $s$ less than or equal to $r$, where $v\in S$ is arbitrary.
Notice that $V_r \subseteq  W_r  = \{  \exp(s v)  \ \ | \ \ v \in S, s \leq r  \}$.

On $D_p $ we can write the metric with respect to spherical coordinates
as
$d \mu_g(r,\theta) = \omega(r,\theta) dr \wedge d\theta$ where $d\theta$  are the standard
coordinates on $S^{n-1}_1(0)$, and $\omega (0,\cdot) = 0$.
Let us denote the corresponding volume form in Euclidean space by 
$d\mu_E(r,\theta) = r^{n-1} dr \wedge d\theta = \omega_E(r,\theta) dr \wedge
d\theta $ that is  $\omega_E(r,\theta) = r^{n-1}$ doesn't depend on $\theta$. 
$\omega$ is a smooth function on  $D_p \backslash \{0\}$, and
$\omega(\cdot,v):(0,c(v)) \to \R^+$ is smooth and satisfies 
$\partt \frac{\omega(v,t)}{\omega_E(t)} = \psi(t,v)
\frac{\omega(v,t)}{\omega_E(t)}$
where $\psi(t,v):= h(t,v)- h_E(t)$ and $h(t,v)$ is the mean curvature
at $(t,v)$ of the geodesic sphere at distance $t$ from $p$ in $(M,g)$
at the point $\exp(tv)$, and $h_E(t)$ is the mean curvature of the
sphere of radius $t$ in euclidean space, that is $h_E(t) =
\frac{(n-1)}{t}$: see \cite{CLN}. One uses here, that
 $\partt \omega = h \omega$ and $\partt \omega_E = h_E \omega_E$ as
 shown in \cite{CLN} (see equation 1.132 there).
Integrating with respect to the $r$ direction we get
\begin{eqnarray}
\frac{\omega(v,r)}{\omega_E(r)}  - \frac{\omega(v,t)}{\omega_E(t)}
\leq \int_t^r \psi(s,v) \frac{\omega(v,s)}{\omega_E(s)} ds \label{starting}
\end{eqnarray}
for all $ 0<t\leq r \leq  c(v)$. 
We extend $\psi$ and $\omega$ to all of $\R^n$ by defining them to be zero
on $(D_p)^c$. These are then measurable functions since 
$D_p$ can be approximated by smooth open sets $D_{\ep}$ contained in $D_p$ and
so we can approximate any of   these functions pointwise (call it $f$)  by $\eta
f$ where $\eta $ is a cut off function with $\eta =1$ if we are a
distance (euclidean) $\ep$ away from $\boundary D_{\ep}$ but in
$D_p$.
Furthermore $\omega$ is
bounded on any ball $B_R(0)$: see
Theorem III.3.1 in \cite{Chav} and use that the solution of linear ordinary
differential equations with smooth coefficients are themselves smooth.
The equation \eqref{starting} holds for all $v \in S^{n-1}$ and all $
0 \leq t\leq r < \infty$ as
one readily verifies: the only new case one needs to consider is $r
> c(v)$ , and in this case the left hand side is less
than or equal to zero and the right hand side is always larger than  or equal
to zero, and so the equation holds trivially.
All functions are measurable and non-negative and so we may apply
Fubini to them. That is, we may integrate \eqref{starting} over $S$
and change the order of integrals. 
We do so in the following, sometimes without any further comment.

\begin{eqnarray}
\int_S\frac{\omega(\theta,r)}{\omega_E(r)}d\theta  - \int_S \frac{\omega(\theta,t)}{\omega_E(t)}d\theta
\leq \int_t^r\int_S \psi(s,v) \frac{\omega(v,s)}{\omega_E(s)} d\theta ds \label{starting2}
\end{eqnarray}
Dividing by $d\theta(S) = \mu$ we get 
\begin{eqnarray}
\frac{\int_S \omega(\theta,r) d\theta }{\int_S \omega_E(r) d\theta}
  - \frac{\int_S \omega(\theta,t) d\theta }{\int_S \omega_E(t) d\theta}
\leq \frac{1}{d\theta(S)}\int_t^r\int_S \psi(s,v) \frac{\omega(v,s)}{\omega_E(s)} d\theta ds \label{starting3}.
\end{eqnarray}
This corresponds to the second inequality at the top of page 1037 of
\cite{PeWe} (with there $S^{n-1}$ replaced by our $S$).
We continue on now as in \cite{PeWe}, to obtain
\begin{eqnarray}
&& (\int_{S} \omega(r,\theta) d\theta) (\int_{S} \omega_E(t,\theta)
d\theta)
- (\int_{S} \omega_E(r,\theta) d\theta) (\int_{S} \omega(t,\theta)
d\theta) \cr
&&  \ \ \ \leq d\theta(S)  \omega_E(r) \Big(\int_0^r\int_S \psi^{2p}(s,\theta) \omega(s,\theta) d\theta
\wedge ds\Big)^{1/2p}  \times \cr
&& \ \ \  \ \ \ \ \ \ \ \ \ \ \ \ \  \ \ \Big(\int_0^r \int_S  \omega(s,\theta) d\theta \wedge
ds\Big)^{1- 1/(2p)}\cr
&& = d\theta(S)  \omega_E(r) (\vol(V_r))^{1-1/(2p)} \Big(\int_0^r\int_S \psi^{2p}(s,\theta) \omega(s,\theta) d\theta
\wedge ds\Big)^{1/2p}  
\end{eqnarray}
Since $\omega$ is bounded on any ball $B_R(0)$, we see that $f(r):=  \vol(V_r) = \int_{0}^r \int_S \omega(t,\theta) d\theta \wedge
dt $ is Lipschitz continuous in $r$. 
Note we use here, that for $K_{S,r}:= \{ m v \ | \ v \in S, m \leq r
\}$, $V_{r} = \exp(K_{S,r} \cap D_p)  \cup \exp(K_{S,r} \cap C_p)$ and that
the measure of the second set is zero, since the cut-locus $\hat C_p =
\exp(C_p)$ has measure zero: see Section 3 of
\cite{Wei} for a proof of the fact that the cut locus has measure
zero. Notice also that this also shows that $V_r$ is measurable: the
second set has measure zero (and so is measurable) and the first set
is $ \exp(K_{S,r} \cap D_p) = f^{-1}(K_{S,r} \cap D_p)$ where $f:\hat
D_p \to D_p$ is the smooth inverse of $\exp(p):D_p \to \hat D_p$.

Hence, the function $f:[0,\infty)
\to \R_0^+$ is in
$W^{1,\infty}$ and the derivative exists almost everywhere and is
equal to the weak derivative whenever it exists: see proof of Theorem 5
4.2.3 in \cite{EG}.
Also, using the Lebesgue-Besicovitch differentiation Theorem (see
Theorem 1 in 1.7 of \cite{EG}), we see that the derivative $f'(r)$  of $f$ exists
almost everywhere and is equal to
$f'(r) = \int_S \omega(r,\theta) d\theta \wedge dr$.
Using this notion of derivative, we argue as in \cite{PeWe} again to
obtain (5th line of page 1038 of \cite{PeWe}):
\begin{eqnarray}
\partr \frac{\vol(V_r)}{v_r} && \leq \frac{d\theta(S) r \omega_E(r) (\vol(V_r))^{1-1/(2p)} (\int_{B_r(0)} \psi^{2p} d\mu_g)^{1/(2p)}}
{(v^2_r)} \cr
&& \leq   \frac{d\theta(S) r  \omega_E(r) }{v_r} (v_r)^{-1/(2p)} 
 \Big(\frac{\vol(V_r)}{v_r}\Big)^{1-1/(2p)}    (\int_{B_r(0)} \psi^{2p}
d\mu_g)^{1/(2p)} 
\end{eqnarray}
where 
\begin{eqnarray}
v_r:= \int_S\int_0^r \omega_E(s) d\theta \wedge ds = d\theta(S)
\int_0^r s^{n-1} ds = \frac{ d\theta(S) r^n}{n}
\end{eqnarray}
since $\omega_E(r) = r^{n-1}$.
The first term  $\frac{d\theta(S) r  \omega_E(r)}{v_r}$
is equal to $ \frac{ nd\theta(S) r^n }{d\theta(S)r^n} = n $.
The term $\int_{B_r(0)} \psi^{2p} d\mu_g$ is shown in Lemma 2.2 of
\cite{PeWe} to be bounded by
$\int_{B_R(0)} |\Rc|^p d\mu_g$.
Hence we get:
\begin{eqnarray}
\partr \frac{\vol(V_r)}{v_r} 
&& \leq n  (v_r)^{-1/(2p)}\Big( \frac{\vol(V_r)}{v_r}\Big)^{1-1/(2p)}   \Lambda \cr
&& = n  (\frac{  n }{d\theta(S) r^n    })^{1/(2p)}\Big(\frac{\vol(V_r)}{v_r}\Big)^{1-1/(2p)}   \Lambda \cr
&& =  \frac{c(n,p)}{ \mu^{1/(2p)} }\Big(\frac{\vol(V_r)}{v_r}\Big)^{1-1/(2p)}
\frac{1}{r^{n/(2p)} }  \Lambda
\end{eqnarray}
for all $r \leq R$,
where $\Lambda= 
(\int_{B_R(0)} |\Rc|^{p} d\mu_g)^{1/(2p)}.$

That is 
\begin{eqnarray}
\partt f(t) \leq \frac{c(n,p)}{\mu^{1/(2p)} }  \Lambda f^{1-1/(2p)}(t) g(t) 
\end{eqnarray}
where $f(t):=  \frac{\vol(V_r)}{v_r}$ and $g(t):=
\frac{1}{t^{n/(2p)} }$.
This implies (using the chain rule for weakly differentiable
functions) that 
\begin{eqnarray}
\partt f^{1/(2p)}(t) \leq \Lambda \frac{\ti c(n,p)}{\mu^{1/(2p)}} g(t)
\end{eqnarray}
and hence
\begin{eqnarray}
\partt\Big( f^{1/(2p)}(t)  - \frac{1}{1-n/(2p)}  \Lambda\frac{\ti
  c(n,p)}{\mu^{1/(2p)}}    t^{1- n/(2p)} \Big)
\leq 0.
\end{eqnarray}
But this means that the function $l(t):= f^{1/(2p)}(t)  -
\frac{\Lambda}{1-n/(2p)}  \frac{\ti c(n,p)}{\mu^{1/(2p)}}  t^{1- n/(2p)}$ is
non-increasing: let $\psi$
be the $W^{1,\infty}$ function given by
$\psi(x) = (1/\ep)(x-r)$ for $x \in (r,r+\ep)$ and 
$\psi(x) = 1 $ for $x \in (r+\ep,s-\ep)$,
$\psi(x) =(1/\ep)(s-x)$ for $x \in(s - \ep,s)$, and $\psi(x) = 0$ for
all other $x$, where $0<r<s \leq R$.
Mollifying this $\psi$, we obtain a smooth function $\ti \psi\geq 0 $ with compact support
such that 
\begin{eqnarray}
l(s) -l(r) && \sim -\int_0^R  l(t) (\psi)'(t) dt  \cr
&& \sim -\int_0^R  l(t) 
(\ti \psi)'(t) dt \cr
&& =   \int_0^R l'(t)\ti \psi(t) dt \cr
&& \leq 0
\end{eqnarray} and
taking a limit with $\ep \to 0$ gives us the claimed monotonicity of $l$
(use also that $l$ is Lipschitz continuous here). 
The monotonicity of $l$ gives us:
\begin{eqnarray}
( \frac{\vol(V_s(p))}{\vol(E_s)}) ^{1/(2p)}  - (
\frac{\vol(V_r)}{\vol (E_r)}) ^{1/(2p)} 
\leq   \Lambda \frac {\ti \ti  c(n,p)} {\mu^{1/(2p)}}  R^{1- n/(2p)}
\end{eqnarray}
for all $ 0\leq r \leq s \leq R$, which is the claimed estimate.
\end{proof}
\end{appendix}


\begin{thebibliography}{10}





\bibitem[AT]{AT} Ambrosio, L., Tilli, P., \emph{Topics on analysis in
    metric spaces}, Oxford Lecture Series in Math. and its
  Applications, (2004).


\bibitem[And1]{And1} Anderson, M., \emph{Convergence and rigidity of manifolds under Ricci curvature bounds}
Invent. math. 102, 429-445 (1990) 

\bibitem[And2]{And2} Anderson, M., \emph{Ricci curvature bounds and Einstein
    metrics on compact manifolds}
Journal American Math. Soc., Vol. 2, No. 3 (1989) 455-490

\bibitem[AnCh]{AnCh}
\emph{$C^{\alpha}$ compactness for manifolds with Ricci curvature and
  injectivity radius bounded below},     J. Differential Geom.
    Volume 35, Number 2 (1992), 265-281.


\bibitem[AnCh2]{AnCh2} Anderson, M.,  and J. Cheeger, J., 
\emph{Diffeomorphism finiteness for manifolds with Ricci curvature
and $L^{n/2}$-norm of curvature bounded}, Geom. Funct. Anal., 1(3):231-252, (1991)

\bibitem[ACK]{ACK} Angenent, S., Caputo C., Knopf, D., \emph{   Minimally
    invasive surgery for Ricci flow singularities. }, 
 J. Reine Angew. Math. (Crelle) 672 (2012) 39-87.







\bibitem[BZ]{BZ} Bamler, R., Zhang, Q., {\it Heat kernel and curvature
    bounds in Ricci flows with bounded scalar curvature}
  arXiv:1501.01291 (2015)


\bibitem[BBI]{BBI} Burago,D.,  Burago, Y.,  Ivanov, S.
 \emph{A Course in Metric Geometry }, Graduate Studies in Mathematics,
  vol.33. A.M.S., Providence, RI, 2001.
\bibitem[CaoX]{CaoX} Cao, X.-, \emph{Curvature Pinching Estimate and Singularities of the Ricci Flow},
Comm. Anal. Geom., Vol. 19 (5): 975-990, (2011).


\bibitem[BKN]{BKN} Bando, S.; Kasue, A.; Nakajima, H. \emph{On a
    construction of coordinates at infinity on manifolds with fast
    curvature decay and maximal volume growth}, Inv. math., 97, pp.
  313 - 349, (1989)


\bibitem[Chav]{Chav} Chavel, I.
\emph{ Riemannian Geometry: A Modern Introduction }, Second Edition,
Cambridge Studies in advanced mathematics 98 (2006)

\bibitem[Ch1]{Ch1}Cheeger, J., {\it Comparison and finiteness theorems for
  Riemannian manifolds}, Ph.D. Thesis, Princeton University, 1967


\bibitem[Ch2]{Ch2}Cheeger, J., {\it Finiteness theorems for
  Riemannian manifolds}, Amer. J. Math.. 92 (1970), 61-74




\bibitem[ChowII]{ChowII} Chow, B., 
\emph{On the entropy estimate for the Ricci flow on compact 2-orbifolds},
J. Diff. Geom. 33, p. 597-600 (1991)

\bibitem[ChowWu]{ChowWu} Chow, B., Wu,L.-F., \emph{The Ricci 
flow on compact 2-orbifolds with curvature negative somewhere}, 
Comm. Pure Appl. Math. 44, p. 275-286 (1991)

\bibitem[CGT]{CGT} 
Cheeger,J.,   Gromov,M.,  and Taylor,M.
\emph{ Finite propagation speed, kernel estimates for functions of the Laplace operator, and the geometry of complete Riemannian manifolds}
Journal Differential Geom. Volume 17, Number 1 (1982), 15-53. 

\bibitem[CTZ]{CTZ}  Chen, B.-L., Tang,S.-H.,  Zhu,
  X.-P.,  \emph{Complete classification of compact four manifolds with positive isotropic curvature},     J. Differential Geom.
    Volume 91, Number 1 (2012), 41-80.

\bibitem[ChenYWangI]{ChenYWangI}  Chen, X., Wang,Y.  \emph{Bessel functions, heat kernel and the conical Kaehler-
Ricci flow}, arXiv:1305.0255 (2013)

\bibitem[ChenYWangII]{ChenYWangII}    Chen, X., Wang,Y.  \emph{On the long time behavior of the conical Kaehler-
Ricci  flows}, arXiv:1402.6689 (2014)






\bibitem[ChenWang]{ChenWang} Chen, X., Wang, B., \emph{ On the conditions to extend
    the Ricci flow III}, Int. Math. Res. Notices (2013) 2013 (10): 2349-2367. doi: 10.1093/imrn/rns117 




\bibitem[CLN]{CLN}
Chow, B., Lu,P., Ni,L.
\emph{ Hamilton's Ricci flow}
Graduate studies in Mathematics, Vol. 77, AMS Science Press., (2000) 

\bibitem[DeT]{DeT}
DeTurck,D.,  \emph{Deforming metrics in the direction of their Ricci tensors}, J. Differential Geom. 18 (1983), no. 1, 157–162. 

\bibitem[DeTK]{DeTK}
DeTurck, D., Kazdan, J. 
\emph{Some regularity theorems in Riemannian geometry}, 
Annales Sci.de l'{\'E}cole Norm. Sup.  14 (3): 249–260, ISSN
0012-9593, MR 644518. (1981)

\bibitem[Evans]{Evans} Evans, L., \emph{Partial differential equations},
  Second edition, Graduate Studies Vol. 19, AMS (2010)

\bibitem[EG]{EG} Evans, L., Gariepy, R., \emph{Measure Theory and fine 
properties of functions} , Studies in advanced mathematics, CRC Press (1992)


\bibitem[EGZ]{EGZ} Eyssidieux, Guedj, V., Zeriahi, A., \emph{Weak
    solutions to degenerate complex monge-Ampere flows I}, Math. Ann.
DOI 10.1007/s00208-014-1141-4, (2014)

\bibitem[EGZII]{EGZII} Eyssidieux, Guedj, V., Zeriahi, A., \emph{Weak
    solutions to degenerate complex monge-Ampere flows II},  arXiv:1407.2504 (2014)

\bibitem[FIK]{FIK} Feldman, M., Ilmanen, T., Knopf, D.
 \emph{Rotationally symmetric shrinking and expanding gradient
   Kaehler-Ricci solitons},  J. Differential Geom. 65 (2003), no. 2, 169-209.

\bibitem[GT]{GT} Gilbarg, D., Trudinger, N., \emph{Elliptic Partial
    Differential Equations of Second Order}, 
Third revised edition, Springer, (2000)

\bibitem[HM]{HM} Haslhofer, R., Mueller, R.
\emph{A compactness theorem for complete Ricci shrinkers}, 
Geom. Funct. Anal. 21(5):1091--1116, (2011)

\bibitem[HM2]{HM2} Haslhofer, R., Mueller, R.
\emph{A note on the compactness theorem for 4d Ricci shrinkers}
arXiv:1407.1683 (2014)

\bibitem[HaComp]{HaComp}
Hamilton,R.S.
\emph{A compactness property of the Ricci Flow}
American Journal of Mathematics, 117, 545--572, (1995)

\bibitem[HaThree]{HaThree}
Hamilton,R.S.,
\emph{Three-manifolds with positive Ricci curvature},
Journal of Differential Geometry 17 (2):255-306(1982)



\bibitem[HaThreeO]{HaThreeO}
Hamilton,R.S.,
\emph{ Three-orbifolds with positive Ricci curvature},
in Collected papers on Ricci  flow, Ser. Geom. Top. 37, International Press,
(2003)


\bibitem[HaFour]{HaFour}
Hamilton,R.S.,
\emph{Four-manifolds with positive curvature operator},
Journal of Differential Geometry 24: 153-179, (1986)

\bibitem[HaForm]{HaForm}
Hamilton,R.S.,
\emph{The formation of singularities in the Ricci flow},
{Collection: Surveys in differential geometry}, 
Vol. II (Cambridge, MA), 7-136, (1995).

\bibitem[JK]{JK} Jost, J., Karcher, H.,
\emph{ Geometrische methoden zur Gewinnung von a-priori Schranken
  f\"ur harmonische Abbildungen}, manuscripta math. 40, 27 - 77 (1982)

\bibitem[Ka]{Ka} Kasue, A., {\it A convergence theorem for Riemannian
    manifolds and some applications}, Nagoya Math. J., Vol. 114 (1989),21-51.

\bibitem[KlLo]{KlLo} Kleiner, B., Lott, J., \emph{Singular Ricci flows
    I}, (2014), arXiv:1408.2271 

\bibitem[KLThree]{KLThree} Kleiner, B., Lott, J., 
\emph{Geometrization of three-dimensional orbifolds via Ricci  flow},  	arXiv:1101.3733 (2011)



\bibitem[KL]{KL} Koch, H., Lamm, T., \emph{Geometric flows with rough
    initial data}, Asian J. Math.
    Volume 16, Number 2 (2012), 209-235.


\bibitem[Lee]{Lee} Lee, J., \emph{Introduction to smooth manifolds},
  Graduate texts in Math., Springer (2002) 

 \bibitem[Lee2]{Lee2} Lee, J., \emph{Introduction to topological
     manifolds}, 2nd. Edition, Springer (2010)

\bibitem[Li]{Li}Li, Y. \emph{Smoothing Riemannian metrics with bounded
    Ricci curvature in dimension four, II}
  Ann. Glob. Anal. Geom. (2012) 41:407-421

\bibitem[LiuZhang]{LiuZhang} Liu, J.,  and Zhang, X.,  
\emph{The conical Kaehler-Ricci  fow on Fano manifolds}, 
arXiv:1402.1832 (2014) 






\bibitem[Lu]{Lu}, Lu, P.,  \emph{A compactness property for solutions of the
    Ricci flow on Orbifolds}, American Journal of Mathematics, Vol. 123, No. 6, Dec., 2001  



\bibitem[MT]{MT}, McCann R., Topping, P.,  
\emph{Ricci flow, entropy and optimal transportation}, American Journal of Math., 132 (2010) 711-730 


\bibitem[MRS]{MRS},  Mazzeo, R., Rubinstein, .Y., Sesum, .N.,
 \emph{Ricci flow on surfaces with conic
    singularities},  	arXiv:1306.6688 (2013)

\bibitem[On]{On} O'neill, Barrett, \emph{Semi-Riemannian geometry with
  applications to relativity}, Academic Press (1983)


\bibitem[Petersen]{Petersen}
Peter Petersen, \emph{Convergence theorems in Riemannian geometry},
Comparison Geometry, MSRI Publications, Volume 30, 1997


\bibitem[PeWe]{PeWe}
Petersen, P., Wei, G.-F. \emph{Relative volume comparison with integral curvature bounds}, GAFA 7 (1997) 1031-1045

\bibitem[Pe1]{Pe1}
Perelman,G.,
\emph{The entropy formula for the Ricci flow and its geometric applications},
MarthArxiv link: math.DG/0211159, (2002)



\bibitem[SS]{SS} Sabitov, I.,  Sefel, S. 
\emph{Connections between the order of smoothness of a surface and
  that of its metric}, Akademija Nauk SSSR. Sibirskoe Otdelenie. Sibirskii Matematiceskii Zurnal 17 (4): 916–925, ISSN 0037-4474, MR 0425855.


\bibitem[Sesum]{Sesum} Sesum, N. , \emph{Curvature tensor under the
    Ricci flow}, American Journal of Mathematics, Vol. 127, (2005). 

\bibitem[SesumLe]{SesumLe} Sesum, N. ,  Le. N.Q., \emph{Remarks on the
  curvature behaviour at the first singular time of the Ricci flow},
Pacific  Journal of Mathematics, Vol. 255, No. 1, (2012)







\bibitem[Shi]{Shi} Shi, W.-X.,  \emph{Ricci deformation of the metric
    on complete noncompact Riemannian manifolds}, J. Differential Geom. 30 (1989) 303-394.

\bibitem[Si1]{Si1} Simon, M., \emph{ Some integral curvature estimates for the Ricci flow in four
  dimensions}, Arxiv Preprint.


\bibitem[SimC0]{SimC0} Simon, M. \emph{Deformation of $C^0$ Riemannian metrics in
    the direction of their Ricci curvature}, Comm. Anal. Geom. 10 (2002), no. 5, 1033-1074.


\bibitem[SimSmoo]{SimSmoo} Simon, M. \emph{Local smoothing results for the Ricci
    flow in dimensions two and three },
Miles Simon. Geometry and Topology 17 (2013) 2263–2287.







\bibitem[SongWeinkove1]{SongWeinkove1}  \emph{Contracting exceptional divisors by the Kaehler-Ricci flow} Duke Math. J. 162 (2013), no. 2, 367-415

\bibitem[SongWeinkove2]{SongWeinkove2}  \emph{Contracting exceptional
    divisors by the Kaehler-Ricci flow II },
  Proc. Lond. Math. Soc. (3) 108 (2014), no. 6, 1529-1561

\bibitem[Tian]{Tian} Tian, G. \emph{On Calabi's conjecture for complex
    surfaces with positive first Chern class}, Inventiones
  mathematicae, Volume: 101, Issue: 1, page 101-172 (1990).


\bibitem[TZ]{TZ} Tian, G., Zheng, Z. \emph{Regularity of
    Kaehler Ricci flows on manifolds}, arXiv:1310.5897 

\bibitem[TME]{TME} Topping, P., Mueller, R., Enders, J.,
\emph{ On Type I Singularities in Ricci flow} Communications in Analysis and Geometry, 19 (2011) 905--922

\bibitem[Topping]{Topping} Topping, P.,\emph{Diameter Control under Ricci flow}, Comm. Ana. Geom. 13, (2005) 1039-1055


\bibitem[Wang1]{Wang1} Wang, B., \emph{ On the conditions to extend
    the Ricci flow}, Int. Math. Res. Notices, (2008) 2008 doi:
  10.1093/imrn/rnn012 

\bibitem[Wang2]{Wang2} Wang, B., \emph{ On the conditions to extend
    the Ricci flow (II) },
 Int Math Res Notices (2012) 2012 (14): 3192-3223. doi:
 10.1093/imrn/rnr141 

\bibitem[WangY]{WangY} Wang,  Y.,  \emph{Smooth approximations of the conical Kaehler-Ricci 
flows}, arXiv:1401.5040 (2014)

\bibitem[Wei]{Wei}Wei, Guofang, \emph{Manifolds with a lower Ricci
    curvature bound}, Surveys in Differential Geometry XI
  ,eds. J. Cheeger and K. Grove, International Press,
  Somerville, MA, pp. 203-228, (2007)

\bibitem[WuLF]{WuLF}, Wu, L.-F., \emph{ The Ricci  flow on 2-orbifolds
    with positive curvature},
 J. Diff. Geom. 33, p. 575-596 (1991)

\bibitem[YangD]{YangD} Yang, D., \emph{Convergence of Riemannian manifolds
  with integral bounds on curvature I}, Ann. Sci. de l'ENS, 4, tome 25,
no 1(1992), pp.77-105



\bibitem[Ye]{Ye} Ye, Rugang, \emph{The logarithmic Sobolev inequality
    along the Ricci flow}, arXiv:0707.2424

\bibitem[Yin]{Yin} Yin, H.,  \emph{ Ricci flow on surfaces with
    conical singularities}
, J. Geom. Anal. 20, p. 970-995 (2010)

\bibitem[YinII]{YinII} Yin, H.,  \emph{ Ricci flow on surfaces with
    conical singularities II}, arXiv:1305.4355 (2013)

\bibitem[Zhang]{Zhang}  Zhang, Z., \emph{ Scalar curvature behaviour for finite-time
  singularity of Kahler Ricci flow}, Michigan Math. J. Volume 59, Issue 2 (2010), 419-433.

\bibitem[Zhang1]{Zhang1}
Zhang, Q. A \emph{Uniform Sobolev Inequality Under Ricci Flow},
Int. Math. Research Notices, Vol. 2007, Article ID rnm056, 17 pages.
doi:10.1093/imrn/rnm056

\bibitem[Zhang2]{Zhang2}
Zhang, Q. S. \emph{Erratum to: A Uniform Sobolev Inequality Under
  Ricci Flow}, 
Int. Math. Research Notices, Vol. 2007, Article ID rnm096, 4 pages.
doi:10.1093/imrn/rnm096

\bibitem[Zhang3]{Zhang3} 
Zhang, Q. S. \emph{Bounds on volume growth of geodesic balls under
  Ricci flow}
Mathematical Research Letters, 2012; 19 (1): 245-253

\end{thebibliography}
\end{document}